\documentclass[11pt,reqno]{amsart}
\usepackage[letterpaper,margin=1in]{geometry}
\usepackage{amsmath,amssymb,amsthm,mathtools}
\usepackage[T1]{fontenc}
\usepackage{lmodern}
\usepackage{microtype}
\usepackage{booktabs,array,enumitem}
\usepackage{flafter}

\usepackage{tikz}
\usetikzlibrary{arrows.meta,calc,decorations.pathreplacing,patterns}
\definecolor{pathblue}{RGB}{40,87,137}
\definecolor{pathorange}{RGB}{187,91,25}
\definecolor{pathgreen}{RGB}{43,119,87}
\definecolor{witnessgold}{RGB}{163,119,24}
\tikzset{
 staticpath/.style={draw=pathblue,line width=1.15pt},
 dynamicpath/.style={draw=pathorange,line width=1.15pt},
 proxypath/.style={draw=pathgreen,line width=1pt,dashed},
 guide/.style={draw=black!35,densely dashed,line width=.45pt},
 every picture/.style={font=\small},
 >=Stealth
}
\usepackage[backend=biber,style=alphabetic,doi=false,maxalphanames=10,maxnames=50]{biblatex}

\usepackage[colorlinks=true,linkcolor=blue,citecolor=blue,urlcolor=blue]{hyperref}

\newcommand{\RR}{\mathbb R}
\newcommand{\ZZ}{\mathbb Z}
\newcommand{\PP}{\mathbb P}
\newcommand{\EE}{\mathbb E}
\newcommand{\cL}{\mathcal L}
\newcommand{\sL}{\mathcal{L}}
\newcommand{\0}{\mathbf{0}}

\newcommand{\cE}{\mathcal E}
\newcommand{\cG}{\mathcal G}
\newcommand{\cI}{\mathcal I}
\newcommand{\cB}{\mathcal B}
\newcommand{\1}{\mathbf 1}
\newcommand{\norm}[1]{\left\lVert#1\right\rVert}
\newcommand{\abs}[1]{\left|#1\right|}
\newcommand{\hor}{\mathrm{hor}}
\newcommand{\wgt}{\mathrm{wgt}}

\DeclareMathOperator{\Leb}{Leb}
\newtheorem{theorem}{Theorem}
\newtheorem{proposition}[theorem]{Proposition}
\newtheorem{lemma}[theorem]{Lemma}
\newtheorem{corollary}[theorem]{Corollary}
\theoremstyle{remark}

\newcommand{\cP}{\mathcal P}

\newcommand{\W}{\mathsf W}
\newcommand{\cH}{\mathcal H}
\newcommand{\cU}{\mathcal U}
\newcommand{\ind}{\mathbf 1}

\hypersetup{
 pdftitle={A non-trivial dynamics on the directed landscape},
 pdfauthor={Manan Bhatia},
 pdfsubject={Critical-time overlap and nontrivial Ornstein--Uhlenbeck dynamics on the directed landscape}
}
\title[A non-trivial dynamics on the directed landscape]
{{A non-trivial dynamics on the directed landscape}}
\author{Manan Bhatia}
\address{Manan Bhatia, Department of Mathematics, University of Geneva, Geneva, Switzerland}
\email{mananbhatia1701@gmail.com}
\date{}

\begin{document}
\begin{abstract}
\begingroup
{We prove the existence of} a non-trivial continuous stationary reversible dynamics
on the directed landscape as a subsequential scaling limit of Brownian
last passage percolation under the Ornstein--Uhlenbeck dynamics.
Strong passage-time stability estimates from the companion paper
\cite{Bha26}, together with static endpoint regularity, give tightness
at the critical dynamical scale $n^{-1/3}$. {To establish
non-triviality of the subsequential limiting dynamics, we show that,
in dynamical BLPP, the fractional overlap between geodesics
at times zero and $a n^{-1/3}$ tends to one as first $n\to\infty$
and then $a\downarrow0$. The proof refines the excursion argument of
Ganguly and Hammond \cite{GH24}, using strong passage-time stability
and the fact that a fixed dynamical time is increasingly subcritical
at finer spatial scales.}
\par\endgroup
\end{abstract}

\maketitle
\setcounter{tocdepth}{1}
\tableofcontents
\enlargethispage{2pt}

\section{Introduction}\label{sec:introduction}
\begingroup
Last passage percolation (LPP) is a model of random geometry
{in which directed paths attempt to maximise the weight they collect from}
a random environment. {In lattice LPP, for example,
independent identically distributed weights are assigned to the
vertices of $\mathbb Z^2$;}
the passage time between ordered vertices is the largest total
weight of an up-right path, and a {maximising} path is called a
geodesic. {Under mild restrictions on the vertex weight distribution, LPP models are believed to belong to the
Kardar--Parisi--Zhang universality class \cite{KPZ86}, with
characteristic exponents $1/3$ for passage-time fluctuations
and $2/3$ for transversal fluctuations {of geodesics}. The directed landscape,
constructed by Dauvergne, Ortmann and Vir\'ag \cite{DOV22}, is
the full scaling limit of the passage-time field for several
integrable LPP models \cite{DOV22,DV21}, and is conjectured to
arise from a broad class of first and last passage percolation models.}

{In the static model, the environment is sampled once and
then kept fixed. One may instead let the weights evolve randomly,
preserving their distribution at each fixed time, to obtain a
dynamical version of LPP.}
The passage-time field and its geodesics then evolve as well.
It is natural to ask whether this evolution has a scaling
limit: can one obtain a continuous process whose state at each
deterministic time is a directed landscape, and whose states
at different times are not all identical?

{In this paper we prove the existence of such a process
as a subsequential scaling limit of dynamical Brownian LPP.}
The environment consists of independent two-sided Brownian
motions indexed by integer rows, and paths collect Brownian
increments along their horizontal pieces. Each Brownian motion
evolves under stationary Ornstein--Uhlenbeck (OU) dynamics,
with correlation $e^{-|t-s|}$ between the same increment at
dynamical times $s,t$.
\begingroup
{Previous work \cite{GH24} identifies $n^{-1/3}$, up to
subpolynomial corrections, as the transition scale from overlap
of order $n$ to sublinear overlap between geodesics of horizontal
length $n$. Here, overlap refers to the total horizontal length shared
by geodesics with the same endpoints at dynamical times zero and
$t$. We remove the subpolynomial correction from the stability
{estimate, thereby establishing that overlap remains close to its full
length at time $a n^{-1/3}$ when the fixed coefficient $a$ is small.
This identifies $n^{-1/3}$ itself as the critical scale for geodesic
overlap, without a subpolynomial correction.}}
\par\endgroup

\begingroup
{This critical-scale overlap estimate, together with the
strong passage-time stability estimates of the companion paper
\cite{Bha26}, yields non-trivial limiting dynamics. Indeed, the
stability estimates and the static endpoint regularity of
\cite{DOV22} readily give tightness jointly in dynamical time and
the four endpoint coordinates. These upper bounds alone, however,
do not exclude a} limit that is constant in dynamical time: the true scale for a
change of order $n^{1/3}$ could still be larger by a diverging
subpolynomial factor. Our overlap bound rules this out through
the covariance--overlap identity of \cite{Cha14,GH24}, which
expresses the increment variance as an integral of expected overlap.
Overlap of order $n$ over a time interval of order $n^{-1/3}$
therefore gives an increment variance of order $n^{2/3}$.
Thus, even while the geodesics share most of their horizontal
length, the evolving weights produce passage-value fluctuations
on the static scale. {Every subsequential limit is continuous,
stationary and reversible.}
\par\endgroup

\begingroup
\begingroup
Several questions about these subsequential limiting dynamics remain open.
\begin{enumerate}[label=(\arabic*),leftmargin=*]
\item Are the subsequential limits unique, and are they Markov?
Neither of these properties is established here.
\item Can the dynamics be constructed directly from geometric
information in the directed landscape and additional randomness,
in analogy with continuum dynamical percolation and its pivotal
measures~\cite{GPS18}? Such a description could provide an intrinsic
{characterisation} and a route to uniqueness.
\item Given the intimate relationship \cite{DOV22} between the directed
landscape and the Airy line ensemble~\cite{CH14}, is there a natural description of the landscape dynamics on the involved ensembles?
\item Do the fields at widely separated dynamical times become
{independent? More precisely, with $\cL^t$ denoting the landscape at time
$t$ in a subsequential limiting dynamics, does $(\cL^0,\cL^t)$ converge
in law, as $t\to\infty$, to a pair of independent directed landscapes?
This is a continuum counterpart of the question of passage-time
decorrelation in BLPP at supercritical dynamical times.}
\end{enumerate}
\par\endgroup
\par\endgroup
\par\endgroup

\subsection{\texorpdfstring{{Dynamics and noise sensitivity}}{Dynamics and noise sensitivity}}\label{sec:background}

\begingroup
{The transition around the scale $n^{-1/3}$ gives a quantitative
example of noise sensitivity, in the sense that a perturbation acting
for a vanishing amount of time can substantially change a macroscopic
geodesic.}
More generally, the concept of noise sensitivity, first developed in \cite{BKS99} for Boolean functions and percolation, concerns the response of macroscopic
observables to small random perturbations.

{Critical site percolation on the triangular lattice is a
well-studied model whose natural resampling dynamics exhibits
noise sensitivity.} Although there is almost surely no infinite open
cluster at any fixed time, independent resampling of the sites
produces exceptional times at which such a cluster exists
\cite{SS10}; the set of these times has Hausdorff dimension $31/36$
\cite{GPS10}. {The full static scaling limit was
constructed in~\cite{CN06}.} At the level of the scaling limit, \cite{GPS13}
constructed pivotal measures that are measurable functions of the
static continuum configuration. These measures and additional
Poisson randomness were then used to construct continuum dynamical
percolation, prove convergence of the rescaled discrete dynamics,
and establish the Markov property~\cite{GPS18} of the continuum dynamics.
The static percolation scaling limit was also proved in~\cite{SS11}
to be a two-dimensional black noise, in the sense of Tsirelson and
Vershik~\cite{TV98}.

\begingroup
For Gaussian optimisation problems, Chatterjee \cite{Cha14}
related superconcentration of the optimal value to instability of
the optimiser under dynamics, with applications to last passage
percolation and directed polymers, including positive-temperature
models. For BLPP, \cite{GH24} then identified
the overlap transition at $n^{-1/3}$ up to subpolynomial corrections,
as described above. {Further results on stability, chaos and noise
sensitivity in lattice LPP appear in \cite{ADS24,AHT26}.
For related work on exceptional times in critical first-passage
percolation, see \cite{DHHL23,DHHL26}.}
\par\endgroup

{Extending the one-dimensional black-noise result of \cite{HP24},
the recent work \cite{GGH26} proves that the directed landscape is
a two-dimensional black noise, thereby giving a decomposition into
independent local randomness associated with space--time rectangles.
This black-noise property can be viewed as intrinsic noise sensitivity
of the continuum field: if the randomness in each rectangle is
independently resampled with any fixed positive probability, then the
covariance of a square-integrable observable evaluated in the original
and perturbed landscapes tends to zero as the rectangular mesh is
refined; see \cite[Definition XII.61]{GS15} for this interpretation.}
This statement concerns the static continuum field; it does not
require a limiting dynamics to have been defined.

{Motivated by the appearance of infinite open clusters at
exceptional times in dynamical critical percolation, a related
line of work asks whether non-trivial bigeodesics can appear at
exceptional times in dynamical LPP.} A bigeodesic is a
bi-infinite directed path every finite portion of which is a geodesic;
here the entirely horizontal and vertical paths are excluded.
For dynamical exponential LPP, \cite{BE25} gives a lower bound
$c/\log n$ for the probability that, at some time in $[0,1]$, a
geodesic joining suitable opposite segments at distance of order $n$
passes through the origin. This is a finite-scale near-existence
result. For Brownian LPP under discrete resampling,
\cite{Bha25} bounds the number of geodesic switches and obtains
an upper bound of $1/2$ on the Hausdorff dimension of the set of
times admitting non-trivial bigeodesics. The companion paper~\cite{Bha26}
improves this bound to zero by controlling the trace swept out by
geodesics, and also supplies the passage-time stability estimates
used here. Whether non-trivial bigeodesics actually occur at
exceptional times {remains an open question.}
\par\endgroup

\begingroup

\subsection{The model}\label{sec:model}

\begingroup
{We briefly recall Brownian LPP and its OU dynamics.}
For each integer row $i$, we evolve a two-sided standard
Brownian motion according to stationary Ornstein--Uhlenbeck dynamics,
independently for different rows. Write $B_i^t(x)$ for its value at
horizontal coordinate $x$ and dynamical time $t$. More precisely,
$(B_i^t(x))_{i\in\ZZ,\,t,x\in\RR}$ is the {centered} Gaussian field
with covariance
\[
 \EE\bigl[B_i^s(x)B_k^t(y)\bigr]
 =\1_{\{i=k\}}e^{-|t-s|}
       \frac{|x|+|y|-|x-y|}{2}.
\]
At each fixed dynamical time, the processes $B_i^t$ are independent
two-sided standard Brownian motions, pinned to zero at the origin, and corresponding increments at dynamical times $s,t$ have correlation
$e^{-|t-s|}$.
\par\endgroup
For {$u=(x,i),v=(y,k)\in\RR\times\ZZ$} with $x\le y$ and $i\le k$, a
\emph{staircase} from $u$ to $v$ consists of horizontal segments
on the rows $i,\ldots,k$, joined by upward vertical segments.
It is specified by a nondecreasing sequence
$x=z_i\le z_{i+1}\le\cdots\le z_{k+1}=y$: its segment on
row $\ell$ is $[z_\ell,z_{\ell+1}]\times\{\ell\}$.
{We write $\pi(m)=z_{m+1}$ for its departure coordinate from
row $m\in\{i,\ldots,k\}$, and call $k-i$ its row duration.}
{We regard a staircase as a planar curve, including its vertical
segments. In statements about BLPP passage values or subpath
weights, the endpoints are always required to lie in $\RR\times\ZZ$.}
Its weight and the last passage value at time $t$ are
\[
 \wgt^t(\pi)=\sum_{\ell=i}^k
       \bigl(B_\ell^t(z_{\ell+1})-B_\ell^t(z_\ell)\bigr),
 \qquad T_u^{v,t}=\max_{\pi:u\to v}\wgt^t(\pi).
\]
{The nondecreasing sequences specifying these staircases
form a compact simplex, and their weights are continuous functions
of the coordinates. As a consequence, the maximum is always attained. A {maximising}
staircase is called a geodesic and is denoted by $\Gamma_u^{v,t}$.
For deterministic endpoints $u,v$ and a fixed deterministic dynamical
time $t$, it is almost surely unique; see \cite[Lemma B.1]{Ham19}.}
\begingroup
For $\mathbf{0}=(0,0)$ and $\mathbf{n}=(n,n)$, at each fixed
deterministic time $t$ we use the shorthand
\[
 \Gamma^t=\Gamma_{\mathbf{0}}^{\mathbf{n},t}.
\]
Its total horizontal length, meaning the sum of the Euclidean
lengths of its horizontal segments, is $n$.
\par\endgroup
We measure the overlap of the time-zero and time-\(t\)
geodesics by
\begin{equation}\label{st:eq:overlap}
 O_n(t)=|\Gamma^0\cap\Gamma^t|_{\hor}
       =\sum_{i=0}^{n}
         \Leb\{x:(x,i)\in\Gamma^0\cap\Gamma^t\}.
\end{equation}
In particular, \(0\le O_n(t)\le n\). {This is the BLPP
overlap used in~\cite{GH24}, corresponding to the overlap observable
in Chatterjee's Gaussian framework~\cite{Cha14}.}
\par\endgroup

\begingroup
\subsection{A non-trivial dynamics on the directed landscape}\label{sec:oumain}
\begingroup
The directed landscape \cite{DOV22} is a random continuous
function $\cL:\RR^4_\uparrow\to\RR$, where
\[
 \RR^4_\uparrow=\{(x,s;y,t)\in\RR^4:s<t\}.
\]
Its value $\cL(x,s;y,t)$ is the continuum passage value from
$(x,s)$ to $(y,t)$. A geodesic between these points is a continuous
path $\gamma:[s,t]\to\RR$, with $\gamma(s)=x$ and $\gamma(t)=y$,
such that, for every finite partition $s=r_0<\cdots<r_k=t$,
\[
 \sum_{i=1}^k
 \cL\bigl(\gamma(r_{i-1}),r_{i-1};\gamma(r_i),r_i\bigr)
 =\cL(x,s;y,t).
\]
We refer to \cite{DOV22} for the construction and {characterisation}
of its law.

{We now specify the rescaling from BLPP to the directed landscape,
keeping the longitudinal coordinates $s,t$ distinct from the
rescaled dynamical time $\tau$.} Set
\[
 p_n(x,s)=(ns+2n^{2/3}x,ns),\qquad s\in n^{-1}\ZZ.
\]
\par\endgroup
For $z=(x,s;y,t)\in\RR^4_\uparrow$ with $s,t\in n^{-1}\ZZ$,
define
\begin{equation}\label{eq:ourescaling}
 \sL_n^\tau(x,s;y,t)
 =n^{-1/3}\left(
 T_{p_n(x,s)}^{p_n(y,t),\,n^{-1/3}\tau}
       -2n(t-s)-2n^{2/3}(y-x)\right).
\end{equation}
Let $\widehat{\sL}_n^\tau$ be the continuous interpolation that
is linear separately in $s$ and $t$ between successive points of
$n^{-1}\ZZ$, with $x,y$ fixed. On every fixed compact subset of
$\RR^4_\uparrow$, all endpoint pairs used in this interpolation
are ordered for sufficiently large $n$. {For each fixed
deterministic $\tau$, the static scaling-limit theorem
\cite[Theorem 11.1]{DOV22} gives convergence of
$\widehat{\sL}_n^\tau$ in law, locally uniformly in the endpoint
coordinates, to the directed landscape.} {We now state the main result of the paper.}

\begin{theorem}[Nontrivial subsequential dynamics]\label{cor:nontrivial}
There are constants $a_0,c>0$ with the following properties.
The family $(\widehat{\sL}_n^\tau)_{\tau\in\RR}$ is tight in
the local uniform topology. Every subsequential limit
$(\cL^\tau)_{\tau\in\RR}$ is continuous, stationary and
reversible in dynamical time, and for each fixed
$\tau$, the field $\cL^\tau$ has the law of the directed
landscape. Moreover, for every $\tau\in\RR$ and $0<h\le a_0$,
\begin{equation}\label{eq:limitlower}
 ch\le
 \EE\!\left[
  \bigl(\cL^{\tau+h}(0,0;0,1)-\cL^\tau(0,0;0,1)\bigr)^2
 \right]\le2h.
\end{equation}
In particular, the law of the limiting process is not
supported on paths constant in dynamical time.
\end{theorem}

Here, by reversibility, we mean that $(\cL^\tau)_{\tau\in\RR}$ and
$(\cL^{-\tau})_{\tau\in\RR}$ have the same law.
The bounds in \eqref{eq:limitlower} describe the order of
the mean-square change throughout a small interval of
rescaled dynamical times, uniformly over all subsequential
limits.

{Tightness follows readily from the OU increment bounds
of \cite[Theorem 5]{Bha26} and the static endpoint estimates of
\cite{DOV22}; Section~\ref{sec:outightness} gives the details and
also verifies stationarity, reversibility and the marginal statement.
The main work of this paper is the overlap theorem below.
Combined with the covariance--overlap identity and uniform
integrability, it yields the increment bounds in
Section~\ref{sec:nontriviality}.}
{Note that we do not establish uniqueness or the Markov property for
the limiting dynamics. Although the driving OU environment in BLPP is Markov,
this property for the limiting dynamics on the directed landscape does not follow from weak convergence alone.}

\subsection{Overlap at the critical time scale}

{The geometric input for non-triviality of the limiting
dynamics is the following estimate.}

\begin{theorem}[Overlap at the critical time scale]\label{thm:main}
For every fixed $\kappa_*\in(0,1)$,
\begin{equation}\label{eq:mainprob}
 \lim_{a\downarrow0}\limsup_{n\to\infty}
 \PP\bigl(O_n(a n^{-1/3})<\kappa_*n\bigr)=0.
\end{equation}
\end{theorem}

In \cite[Theorem 3.1(1)]{GH24}, a fixed positive overlap
fraction is obtained for times at most
$n^{-1/3}\exp\{-h(\log\log n)^{68}\}$, with $h>0$ chosen
according to the desired error bound. Here $a>0$ is held fixed
as $n\to\infty$, and the overlap fraction can be taken arbitrarily
close to one by subsequently sending $a\downarrow0$. In particular,
since $0\le O_n(t)/n\le1$, our conclusion also gives
\begin{equation}\label{eq:meanfull}
 \lim_{a\downarrow0}\liminf_{n\to\infty}
             \frac{\EE O_n(a n^{-1/3})}{n}=1.
\end{equation}

\par\endgroup

\subsection{Outline of the proof}\label{sec:outline}
\begingroup
To prove Theorem~\ref{cor:nontrivial}, we first obtain subsequential
limiting passage-time fields, and then show that they actually
change in dynamical time. {For the first task, consider, for $\tau\ge0$,}
\[
 {X_n(\tau)=n^{-1/3}\bigl(
 T_{\mathbf{0}}^{\mathbf{n},\tau n^{-1/3}}
       -T_{\mathbf{0}}^{\mathbf{n},0}\bigr).}
\]
The strong passage-time stability estimate of
\cite[Theorem 5]{Bha26}, imported as Lemma~\ref{lem:rotation}, gives
\begin{equation}\label{eq:outlinemoment}
 {\|X_n(\tau)\|_p\le C\sqrt{p\tau},\qquad p\ge2.}
\end{equation}
Together with static endpoint regularity, this gives a modulus of
continuity jointly in dynamical time and all four endpoint coordinates.
Section~\ref{sec:outightness} proves tightness and identifies the
fixed-time marginals as directed landscapes.

Tightness alone would still allow a limiting process that samples
one landscape and then keeps it fixed. To rule this out, we use the
covariance--overlap identity from Chatterjee's Gaussian framework
\cite{Cha14}, in its BLPP form
{\cite[Section 2.3, ``Subcritical energy stability'']{GH24}}:
\[
 {\EE X_n(\tau)^2
 =2\int_0^\tau e^{-u n^{-1/3}}
          \frac{\EE O_n(u n^{-1/3})}{n}\,du.}
\]
The expected overlap $t\mapsto\EE O_n(t)$ is non-increasing (see \cite[Theorem 4.1(2)]{GH24}).
Moreover, \eqref{eq:outlinemoment} with $p=4$ makes
{$\{X_n(\tau)^2:n\ge1\}$ uniformly integrable for each fixed $\tau$.}
Thus it suffices to show that
\[
 \liminf_{n\to\infty}\frac{\EE O_n(a_*n^{-1/3})}{n}>0
 \qquad\text{for some fixed }a_*>0.
\]
Indeed, the monotonicity of expected overlap bounds the integrand
{from below throughout $[0,\tau]$ whenever $0<\tau\le a_*$, and uniform}
integrability passes the resulting variance lower bound to the limit.
The main geometric task is to establish this overlap bound.
Theorem~\ref{thm:main} gives the stronger conclusion that the
overlap fraction tends to one as $a\downarrow0$.

The overlap stability theorem of \cite{GH24} requires $a$ to
decrease with $n$, whereas here it must be small but fixed
independently of $n$. We follow its broad excursion strategy,
but use local weight comparisons and sum expected durations over
all excursion scales. {We first explain this distinction, and then
describe the three main steps in detail.}

An excursion is a part of $\Gamma^t$ that leaves $\Gamma^0$ and
returns without meeting it in between. Its duration is the number
of rows between its endpoints. Fix a small $\beta>0$, and write
$D_{\rm long}$ for the sum of durations of excursions lasting at
least $n^\beta$ rows. The short-excursion argument in
\cite[Proposition 8.4 and its proof]{GH24}, with the horizontal-length
bookkeeping in Lemma~\ref{lem:fullreduction}, reduces our task to
\[
 \lim_{a\downarrow0}\limsup_{n\to\infty}
            \frac{\EE D_{\rm long}(a n^{-1/3})}{n}=0.
\]

In \cite[Sections 3.3.3 and 9.2]{GH24}, one selects a dyadic scale
contributing at least order $n/\log n$ rows when the total excursion
duration is linear in $n$. A time-zero \emph{proxy} is then built:
this is a path between the original endpoints which preserves enough
of the excursion structure of $\Gamma^t$ at the \emph{chosen} scale,
and whose time-zero weight approximates the time-$t$ weight of
$\Gamma^t$. Static near-peak estimates say that well-separated
competing crossings are unlikely both to give nearly optimal weight.
They force a cumulative deficit for the retained excursions, which
is compared with the error in approximating the weight of the whole
path by the proxy.

The distinction between these two quantities matters when missing
overlap is spread over many scales. The deficit then detects only
the small fraction contributed by the selected scale, whereas the
approximation error accumulates along the whole path. To detect
this smaller signal, the latter error must be made correspondingly
smaller. Their theorem takes $a=\exp\{-h(\log\log n)^{68}\}$;
other inputs also contribute to this loss, but the whole-path
comparison is an important place where it enters.
Figure~\ref{fig:proxycomparison} illustrates the distinction.

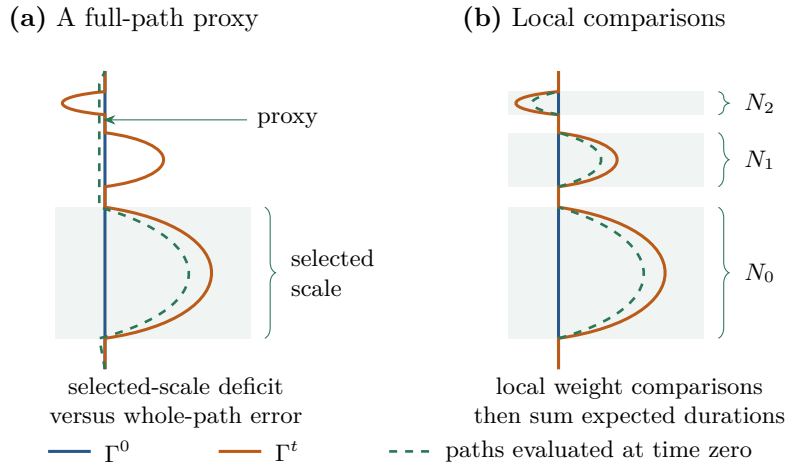
\begin{figure}[tbp]
\centering
\begin{tikzpicture}[x=.94cm,y=.68cm]
 \begin{scope}
  \node[anchor=west] at (0,6.8) {\textbf{(a)} A full-path proxy};
  \fill[pathgreen!7] (.8,.6) rectangle (3.55,3.15);
  \draw[staticpath] (1.5,0)--(1.5,5.8);
  \draw[dynamicpath] (1.5,0)--(1.5,.6)
   .. controls (3.5,1.0) and (3.5,2.75) .. (1.5,3.15)
   --(1.5,3.55)
   .. controls (2.6,3.75) and (2.6,4.4) .. (1.5,4.6)
   --(1.5,4.95)
   .. controls (.7,5.05) and (.7,5.3) .. (1.5,5.4)
   --(1.5,5.8);
  \draw[proxypath] (1.5,0)--(1.42,.6)
   .. controls (3.1,1.2) and (3.1,2.5) .. (1.42,3.15)
   --(1.42,5.55)--(1.5,5.8);
  \draw[decorate,decoration={brace,amplitude=4pt},draw=pathgreen]
    (3.75,3.15)--(3.75,.6)
    node[midway,right=6pt,align=left,font=\footnotesize]
    {selected\\scale};
  \draw[->,draw=pathgreen] (3.5,4.85)--(1.47,4.85);
  \node[anchor=west,align=left,font=\footnotesize] at (3.5,4.85)
    {proxy};
  \node[align=center,font=\footnotesize] at (2.5,-.65)
    {selected-scale deficit\\versus whole-path error};
 \end{scope}
 \begin{scope}[xshift=6cm]
  \node[anchor=west] at (0,6.8) {\textbf{(b)} Local comparisons};
  \foreach \lo/\hi in {.6/3.15,3.55/4.6,4.95/5.4}
   \fill[pathgreen!7] (.8,\lo) rectangle (3.55,\hi);
  \draw[staticpath] (1.5,0)--(1.5,5.8);
  \draw[dynamicpath] (1.5,0)--(1.5,.6)
   .. controls (3.5,1.0) and (3.5,2.75) .. (1.5,3.15)
   --(1.5,3.55)
   .. controls (2.6,3.75) and (2.6,4.4) .. (1.5,4.6)
   --(1.5,4.95)
   .. controls (.7,5.05) and (.7,5.3) .. (1.5,5.4)
   --(1.5,5.8);
  \draw[proxypath] (1.5,.6)
   .. controls (3.1,1.2) and (3.1,2.5) .. (1.5,3.15);
  \draw[proxypath] (1.5,3.55)
   .. controls (2.3,3.85) and (2.3,4.3) .. (1.5,4.6);
  \draw[proxypath] (1.5,4.95)
   .. controls (1.0,5.1) and (1.0,5.25) .. (1.5,5.4);
  \foreach \lo/\hi/\lab in {.6/3.15/N_0,3.55/4.6/N_1,4.95/5.4/N_2}
   \draw[decorate,decoration={brace,amplitude=4pt},draw=pathgreen]
    (3.75,\hi)--(3.75,\lo)
    node[midway,right=6pt,font=\footnotesize] {$\lab$};
  \node[align=center,font=\footnotesize] at (2.5,-.65)
    {local weight comparisons\\then sum expected durations};
 \end{scope}
 \draw[staticpath] (.7,-1.6)--(1.3,-1.6);
 \node[anchor=west,font=\footnotesize] at (1.35,-1.6) {$\Gamma^0$};
 \draw[dynamicpath] (3.1,-1.6)--(3.7,-1.6);
 \node[anchor=west,font=\footnotesize] at (3.75,-1.6) {$\Gamma^t$};
 \draw[proxypath] (5.5,-1.6)--(6.1,-1.6);
 \node[anchor=west,font=\footnotesize] at (6.15,-1.6)
   {paths evaluated at time zero};
\end{tikzpicture}
\caption{Schematic in transverse and row coordinates, with row increasing
upward. In (a), the proxy of \cite{GH24} retains enough excursions of one selected
scale, and its full weight is compared with that of $\Gamma^t$.
Excursions at other scales need not be preserved. In (b), we make a
separate static comparison for each excursion at its own scale;
the dashed pieces represent these local competitors. We do not
assemble them into a full-path proxy or sum their weight errors.
Instead, static rarity bounds the expected duration contributed
at each scale, and these duration bounds are summed over all scales.}
\label{fig:proxycomparison}
\end{figure}

Our local comparisons keep the stability error at the scale of
the excursion being detected. Smaller scales are more subcritical,
so we can impose successively smaller weight tolerances and narrower
tubes relative to their natural fluctuation scales. This improvement
allows us to sum over all scales without a loss depending on $n$.
{The proof has three steps, summarised in Figure~\ref{fig:argumentflow}.}

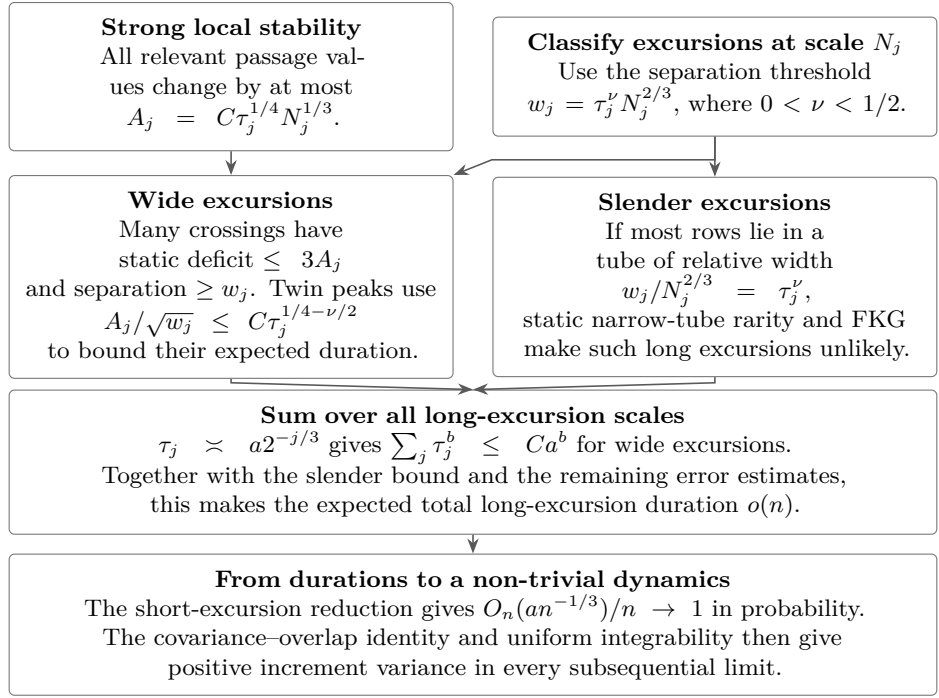
\begin{figure}[tbp]
\centering
\begin{tikzpicture}[
 flowbox/.style={draw=black!55,rounded corners=2pt,align=center,
   text width=5.45cm,inner sep=6pt,font=\footnotesize},
 flowarrow/.style={->,draw=black!65,line width=.6pt}]
 \node[flowbox] (stability) at (-3.2,0)
  {\textbf{Strong local stability}\\
   All relevant passage values change by at most\\
   $A_j=C\tau_j^{1/4}N_j^{1/3}$.};
 \node[flowbox] (width) at (3.2,0)
  {\textbf{Classify excursions at scale $N_j$}\\
   Use the separation threshold\\
   $w_j=\tau_j^\nu N_j^{2/3}$, where $0<\nu<1/2$.};
 \node[flowbox] (wide) at (-3.2,-2.65)
  {\textbf{Wide excursions}\\
   Many crossings have static deficit $\le3A_j$\\
   and separation $\ge w_j$. Twin peaks use\\
   $A_j/\sqrt{w_j}\le C\tau_j^{1/4-\nu/2}$\\
   to bound their expected duration.};
 \node[flowbox] (slender) at (3.2,-2.65)
  {\textbf{Slender excursions}\\
   {If most rows lie in a tube of relative width\\
   $w_j/N_j^{2/3}=\tau_j^\nu$,\\
   static narrow-tube rarity and FKG\\
   make such long excursions unlikely.}};
 \node[flowbox,text width=11.85cm] (sum) at (0,-5.1)
  {\textbf{Sum over all long-excursion scales}\\
   $\tau_j\asymp a2^{-j/3}$ gives $\sum_j\tau_j^b\le Ca^b$ for wide excursions.\\
   Together with the slender bound and the remaining error estimates,\\
   this makes the expected total long-excursion duration $o(n)$.};
 \node[flowbox,text width=11.85cm] (overlap) at (0,-7.25)
  {\textbf{From durations to a non-trivial dynamics}\\
   The short-excursion reduction gives $O_n(a n^{-1/3})/n\to1$ in probability.\\
   The covariance--overlap identity and uniform integrability then give\\
   positive increment variance in every subsequential limit.};
 \draw[flowarrow] (stability.south)--(wide.north);
 \draw[flowarrow] (width.south)--(slender.north);
 \draw[flowarrow] (width.south)--(3.2,-1.1)--(.25,-1.1)--(wide.north east);
 \draw[flowarrow] (wide.south)--(-3.2,-4.05)--(sum.north);
 \draw[flowarrow] (slender.south)--(3.2,-4.05)--(sum.north);
 \draw[flowarrow] (sum.south)--(overlap.north);
\end{tikzpicture}
\caption{{The overlap argument, with limits taken first as $n\to\infty$
and then as $a\downarrow0$. The same shrinking effective time
$\tau_j$ improves both the twin-peak parameter $A_j/\sqrt{w_j}$
and the slenderness constraint $w_j/N_j^{2/3}$. This permits
summation over scales. Tightness, used in the last box, follows
separately from passage-time stability and static endpoint regularity.}}
\label{fig:argumentflow}
\end{figure}

\begin{enumerate}[label=(\arabic*),leftmargin=*]
\item \emph{Prepare stability for the random excursion endpoints
(Proposition~\ref{st:prop:local}).}
The goal is one event on which all the local passage values we
will need change little, wherever the excursions occur. Fix
$t=a n^{-1/3}$. For row duration $N_j\asymp n2^{-j}$ and
comparable horizontal length, Lemma~\ref{lem:rotation} bounds the
$L^p$ norm of a fixed-endpoint increment by $C\sqrt{pN_jt}$.
Relative to its own static fluctuation scale, the natural size is
\[
 \sqrt{N_jt}=\tau_j^{1/2}N_j^{1/3},
 \qquad \tau_j=tN_j^{1/3}\asymp a2^{-j/3}.
\]
Thus the effective time $\tau_j$ decreases geometrically at shorter
scales. We ask only for the larger tolerance
$A_j=C\tau_j^{1/4}N_j^{1/3}$. The ratio
$\sqrt{N_jt}/A_j$ is of order $\tau_j^{1/4}$, leaving room
for uniformity over endpoints.

Geodesic regularity (Lemma~\ref{st:lem:regularityinput}) places all
possible endpoints in deterministic collections of boxes; see
Figure~\ref{fig:endpointboxes}. The strong stability tails then
allow a union bound over these boxes at every scale $N_j\ge n^\beta$.
The decreasing $\tau_j$ improves the failure probability fast
enough to pay for the increasing number of boxes. This event
applies to the endpoints subsequently selected by the geodesics,
without conditioning on those random endpoints.

\item \emph{Bound the expected duration of wide excursions using
static rarity of near-geodesics
(Lemma~\ref{st:lem:crossing} and Corollary~\ref{wd:cor:sum}).}
Consider an excursion of $\Gamma^t$ from $u$ to $v$, so both
endpoints lie on $\Gamma^0$. On an interior row $m$, let
$x=\Gamma^t(m)$ be its departure coordinate. We ask how much
weight a static path from $u$ to $v$ must lose if it is required
to leave row $m$ at this same coordinate $x$. At time $t$ this
requirement costs nothing, since the excursion is part of a geodesic.
The unrestricted passage value and the two passage values forming
the route through $x$ each change by at most $A_j$. Consequently,
the best such static route has deficit at most $3A_j$.
Appending the unchanged prefix and suffix of $\Gamma^0$ gives
a path between the original endpoints with the same deficit;
see Figure~\ref{fig:threevalues}. Thus $x$ is a near-{maximiser}
of the original static routed profile, with a tolerance determined
by the excursion's own length.

An excursion is wide if its departure coordinates are separated
from those of $\Gamma^0$ by more than
$w_j=\tau_j^\nu N_j^{2/3}$ on a fixed positive fraction of
its rows, where $0<\nu<1/2$. Each such interior row witnesses
two nearly optimal crossings separated by at least $w_j$.
We use estimates for the rarity of separated near-geodesics in
static LPP, formulated as twin-peak bounds for the routed weight
profile. Specifically, the parameterised estimate of
\cite[Proposition 82]{Bha25}, imported as
Proposition~\ref{wd:prop:static}, retains the dependence on both
the deficit tolerance and the separation. Its key small parameter is
\[
 \frac{A_j}{\sqrt{w_j}}\le C\tau_j^{1/4-\nu/2}.
\]
This decreases geometrically with $j$. Each wide excursion has
duration at most a constant times the number of its witness rows.
Summing the single-row probabilities therefore bounds the expected
number of witnesses, and hence the expected total duration of wide
excursions. Individual near-peaks may still occur; what matters is
that their expected total count is small.

\item \emph{Exclude long slender excursions
(Proposition~\ref{sl:prop:dynamic}).}
An excursion could remain within distance $w_j$ of $\Gamma^0$
on most rows while sharing no horizontal overlap with it;
Figure~\ref{fig:excursions} contrasts this with the wide case.
We use the mechanism of \cite{GH23}: a path constrained to a
narrow tube around a prescribed curve on most rows is unlikely
to have weight high enough to be a geodesic. The FKG comparison
of \cite[Lemma 9.6]{GH24} permits the prescribed curve to be
$\Gamma^0$, despite its dependence on the evolving environment.

Here the relative tube width is $w_j/N_j^{2/3}=\tau_j^\nu$.
Thus slenderness becomes a stricter condition at smaller scales,
and its rarity improves fast enough to sum over possible locations.
We keep the estimates in terms of this relative width and the
excursion scale, avoiding an additional loss growing with $n$.
\end{enumerate}

The final count includes every scale. For wide excursions,
linearity of expectation and the witness-row bounds give a
scale-dependent contribution of the form
\[
 \frac{\EE D_{\rm wide}}n\lesssim\sum_j\tau_j^b
       +\text{remaining errors},\qquad
 \sum_j\tau_j^b\le Ca^b
\]
for some $b>0$. Here $D_{\rm wide}$ is the total duration of the
long wide excursions. The remaining errors come from stability
failures and boundary rows; Section~\ref{sec:assembly} controls
them and completes the short-excursion reduction. The probability
of a long slender excursion also tends to zero as $a\downarrow0$.
No independence between scales is needed: stability holds
simultaneously, and durations are summed in expectation. Importantly, in this argument, no single
scale has to produce a detectable whole-path weight deficit.
\par\endgroup

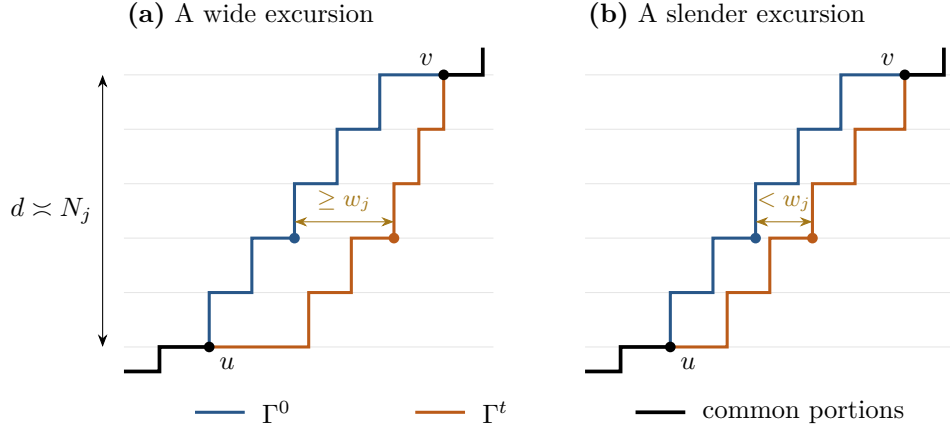
\begin{figure}[tbp]
\centering
\begin{tikzpicture}[x=.94cm,y=.72cm]
 \begin{scope}
  \node[anchor=west] at (0,6.1) {\textbf{(a)} A wide excursion};
  \foreach \y in {0,1,2,3,4,5}
    \draw[black!10] (.1,\y)--(5.3,\y);
  \draw[staticpath] (1.3,0)--(1.3,1)--(1.9,1)
    --(1.9,2)--(2.5,2)--(2.5,3)--(3.1,3)
    --(3.1,4)--(3.7,4)--(3.7,5)--(4.6,5);
  \draw[dynamicpath] (1.3,0)--(2.7,0)--(2.7,1)--(3.3,1)
    --(3.3,2)--(3.9,2)--(3.9,3)--(4.25,3)
    --(4.25,4)--(4.6,4)--(4.6,5);
  \draw[black,line width=1.4pt] (.1,-.45)--(.6,-.45)--(.6,0)--(1.3,0);
  \draw[black,line width=1.4pt] (4.6,5)--(5.15,5)--(5.15,5.5);
  \fill (1.3,0) circle (2pt) node[below right] {$u$};
  \fill (4.6,5) circle (2pt) node[above left] {$v$};
  \fill[pathblue] (2.5,2) circle (2pt);
  \fill[pathorange] (3.9,2) circle (2pt);
  \draw[<->,witnessgold] (2.5,2.3)--node[above,font=\footnotesize] {$\ge w_j$}(3.9,2.3);
  \draw[<->] (-.2,0)--node[left] {$d\asymp N_j$}(-.2,5);
 \end{scope}
 \begin{scope}[xshift=6.1cm]
  \node[anchor=west] at (0,6.1) {\textbf{(b)} A slender excursion};
  \foreach \y in {0,1,2,3,4,5}
    \draw[black!10] (.1,\y)--(5.3,\y);
  \draw[staticpath] (1.3,0)--(1.3,1)--(1.9,1)
    --(1.9,2)--(2.5,2)--(2.5,3)--(3.1,3)
    --(3.1,4)--(3.7,4)--(3.7,5)--(4.6,5);
  \draw[dynamicpath] (1.3,0)--(2.1,0)--(2.1,1)--(2.7,1)
    --(2.7,2)--(3.3,2)--(3.3,3)--(3.9,3)
    --(3.9,4)--(4.6,4)--(4.6,5);
  \draw[black,line width=1.4pt] (.1,-.45)--(.6,-.45)--(.6,0)--(1.3,0);
  \draw[black,line width=1.4pt] (4.6,5)--(5.15,5)--(5.15,5.5);
  \fill (1.3,0) circle (2pt) node[below right] {$u$};
  \fill (4.6,5) circle (2pt) node[above left] {$v$};
  \fill[pathblue] (2.5,2) circle (2pt);
  \fill[pathorange] (3.3,2) circle (2pt);
  \draw[<->,witnessgold] (2.5,2.3)--node[above,font=\footnotesize] {$<w_j$}(3.3,2.3);
 \end{scope}
 \draw[staticpath] (1.1,-1.2)--(1.8,-1.2);
 \node[anchor=west] at (1.9,-1.2) {$\Gamma^0$};
 \draw[dynamicpath] (4.2,-1.2)--(4.9,-1.2);
 \node[anchor=west] at (5,-1.2) {$\Gamma^t$};
 \draw[black,line width=1.4pt] (7.3,-1.2)--(8,-1.2);
 \node[anchor=west] at (8.1,-1.2) {common portions};
\end{tikzpicture}
\caption{Two schematic excursions of the same row duration.
In a wide excursion, a positive fraction of the rows have
departure coordinates separated by at least \(w_j\).
In a slender excursion, this separation is below \(w_j\)
on most rows. Even the latter can have disjoint horizontal
pieces throughout its interior. The drawings show representative
geometries; the definitions count rows and permit exceptions
to the indicated separation.}
\label{fig:excursions}
\end{figure}
\subsection{\texorpdfstring{{Scales and parameters}}{Scales and parameters}}\label{sec:parameters}
\begingroup
We use \(\|X\|_p=(\EE|X|^p)^{1/p}\).
The constants \(c,C\) are positive and may change from
line to line; their dependence on fixed parameters is
indicated when it matters.

There are two independent notions of time that we shall have in the paper.
The row indices describe progress along a staircase; \(t\)
describes the evolution of the Brownian environment. The
parameter \(N_j\) measures row duration, whereas \(\tau_j\)
measures how much the environment has evolved relative to
the critical dynamical time at that duration.

\begingroup
We use \(\beta\) to separate the length scale of short excursions and long excursions. The parameter \(\nu\) determines the tube width that
separates wide from slender excursions, and a smaller \(\nu\)
makes this tube wider. The
wide-excursion estimate requires \(\nu<1/2\), so that the
stability tolerance is small compared with the square root
of the tube width; see \eqref{wd:eq:substitution}.
{The slender-excursion estimate in Proposition~\ref{sl:prop:dynamic}
holds for every fixed \(\nu>0\).}

For the final overlap theorem, the required parameter ranges are
\begin{equation}\label{eq:parameters}
 0<\beta<1/12,\qquad 0<\nu<1/2.
\end{equation}
For example, \(\beta=1/24\) and \(\nu=1/4\) work.
For \(0<a<1\), set
\par\endgroup
\begin{equation}\label{st:eq:scales}
 \begin{gathered}
 t=a n^{-1/3},\qquad N_j=\lfloor n2^{-j}\rfloor,\qquad
 \cI_n=\{j\ge0:N_j\ge n^\beta\},\\
 \tau_j=tN_j^{1/3},\qquad
 \theta_j=\tau_j^\nu,\qquad w_j=\theta_jN_j^{2/3}.
 \end{gathered}
\end{equation}
{We first establish the estimates separately, keeping $\beta$ and $\nu$
free within the ranges stated for each result. The joint restrictions
\eqref{eq:parameters} are imposed at the start of the proof of
Theorem~\ref{thm:main} in Section~\ref{sec:assembly}, where the estimates
are combined using the same wide/slender classification.}
{All estimates below for the BLPP
model hold for sufficiently large \(n\), with \(a\) fixed.}
\par\endgroup

\paragraph{\textbf{Acknowledgements.}}
The author acknowledges the use of GPT-6 Astra at the level of a coauthor in developing the arguments and preparing this manuscript.

\begingroup
\paragraph{\textbf{{Organisation}}}
{{Subsections~\ref{sec:model} and~\ref{sec:parameters} specify
the model, notation and parameters.} {Section~\ref{sec:limit} establishes
the existence and non-triviality of subsequential limiting dynamics,
using the overlap theorem as an input.}
Section~\ref{sec:excursionreduction} then introduces excursions
and reduces the overlap theorem to a bound on long excursions.}
{Section~\ref{sec:stability} proves uniform local
stability, and Section~\ref{sec:crossing} turns it into static
twin-peak events and bounds their probabilities.
Section~\ref{sec:assembly} states the slender-excursion input
and completes the overlap proof. Section~\ref{sec:slender}
proves that input, first for a deterministic reference function
and then for the random geodesic.}
Appendix~\ref{app:parameters} collects the notation, parameter
conditions and external inputs.
\par\endgroup

\begingroup
\section{\texorpdfstring{{Subsequential limits and non-triviality}}{Subsequential limits and non-triviality}}\label{sec:limit}

We first prove Theorem~\ref{cor:nontrivial} assuming the overlap estimate
stated in Theorem~\ref{thm:main}. {This separates the tightness
and non-triviality arguments from the geometric proof of the
overlap estimate.} \begingroup
We use the {normalisation} in \eqref{eq:ourescaling},
with weight multiplier $n^{-1/3}$, throughout the paper.
\par\endgroup
\par\endgroup

\begingroup
\subsection{Tightness and identification of the marginals}\label{sec:outightness}

{In this section, we provide the proof of the tightness and marginal assertions in Theorem~\ref{cor:nontrivial}.} There are five parameters
to control: the four endpoint coordinates of the landscape and
the additional dynamical time. Static BLPP estimates control the
first four, while \cite[Theorem 5]{Bha26} controls the last.
The availability of arbitrarily high moments lets us combine
these bounds into a uniform modulus of continuity.

\par\begingroup
{We use the stationary OU environment and rescaled fields
defined in Section~\ref{sec:oumain}.}
\par\endgroup

\begin{lemma}[OU passage-time increments; {\cite[Theorem 5]{Bha26}}]\label{lem:rotation}
{For deterministic endpoints \(u=(x_0,j_0),v=(x_1,j_1)\in\RR\times\ZZ\)
with \(x_0\le x_1\) and \(j_0\le j_1\), set \(L=x_1-x_0\).
Then, for every \(t\ge0\) and \(p\ge2\),}
\begin{equation}\label{st:eq:rotation}
 \norm{T_u^{v,t}-T_u^{v,0}}_p
 \le C\sqrt{pL(1-e^{-t})}\le C\sqrt{pLt}.
\end{equation}
Moreover, for every \(w\ge2\),
\begin{equation}\label{st:eq:pointtail}
 \PP\!\left(\abs{T_u^{v,t}-T_u^{v,0}}
       >C\sqrt{L(1-e^{-t})w}\right)\le2e^{-w}.
\end{equation}
\end{lemma}
\begin{proof}
For $u=(x_0,j_0)$ and $v=(x_1,j_1)$ with
$D=j_1-j_0\ge1$ and $L=x_1-x_0>0$, put $a=L/D$.
The translated and scaled environment
\[
 \widetilde B_j^r(x)=a^{-1/2}
 \bigl(B_{j_0+j}^r(x_0+ax)-B_{j_0+j}^r(x_0)\bigr)
\]
has the same OU law, with the same dynamical time parameter.
Consequently, writing $\mathbf D=(D,D)$, we have the identity in
law of processes
\[
 (T_u^{v,r})_{r\in\RR}
 \stackrel{d}{=}\bigl(\sqrt{L/D}\,T_{\0}^{\mathbf D,r}\bigr)_{r\in\RR}.
\]
\cite[Theorem 5]{Bha26} and stationarity therefore give
\eqref{st:eq:rotation} and \eqref{st:eq:pointtail}.
For endpoints on the same row these follow directly from the Gaussian
increment law; zero horizontal separation gives zero weight.

\end{proof}

\begingroup
\begingroup
For $z=(x,s;y,t)$ with $s,t\in N^{-1}\ZZ$ and $s<t$, define
the passage field at physical dynamical time $h$ by
\[
 F_N^h(x,s;y,t)
 =N^{-1/3}\left(
 T_{(Ns+2N^{2/3}x,Ns)}^{(Nt+2N^{2/3}y,Nt),h}
 -2N(t-s)-2N^{2/3}(y-x)\right).
\]
Between row-grid points, interpolate linearly in $s$ and $t$
separately, keeping $x,y$ fixed.
\par\endgroup
\begingroup
This is the same field as in \eqref{eq:ourescaling}, with
physical dynamical time in the superscript:
\[
 F_N^h(z)=\widehat{\sL}_N^{\,N^{1/3}h}(z).
\]
The notation $F_N^h$ is convenient when the spatial scale $N$
varies while the physical time $h$ is held fixed, as in the
local comparisons of Section~\ref{sec:stability}.
\par\endgroup
We work on a fixed closed
{coordinate box $K\subset\RR^4_\uparrow$; its time gap
$t-s$ is bounded below}, so the displayed endpoints are
ordered for all sufficiently large $N$.
\par\endgroup

\begin{lemma}[Static endpoint increments; {\cite[Proposition 2.6, Lemma 11.2]{DOV22}}]\label{st:lem:endpointinput}
For the interpolated field above and every fixed closed endpoint
box {\(K\subset\RR^4_\uparrow\)}, there are \(C_K,N_K<\infty\) such
that, for \(N\ge N_K\), \(z,z'\in K\), every deterministic
dynamical time \(h\), and \(p\ge2\),
\begin{equation}\label{st:eq:endpointmoment}
 \norm{F_N^h(z)-F_N^h(z')}_p
 \le C_Kp^{4/3}|z-z'|^{1/6},\qquad p\ge2.
\end{equation}
The constant \(C_K\) is independent of \(p\), \(N\) and \(h\).
\end{lemma}
{The estimate is implicit in the proof of \cite[Theorem 11.1]{DOV22}.
We give the details because this uniform moment formulation is not
stated explicitly there.}
\begin{proof}[Derivation from the static estimates of \cite{DOV22}]
\begingroup
The longitudinal estimate in \cite[Lemma 11.2]{DOV22} has tail
\(C\exp\{-c v^{3/4}\}\) at scale \(|s-s'|^{1/6}\).
Integrating this tail bounds the \(L^p\) norm of such an increment
by \(C_Kp^{4/3}|s-s'|^{1/6}\). Proposition 2.6 gives tail
\(C\exp\{-c v^{3/2}\}\) for spatial increments at scale
\(|x-x'|^{1/2}\), and hence an \(L^p\) bound
\(C_Kp^{2/3}|x-x'|^{1/2}\). On a bounded box,
\(|x-x'|^{1/2}\le C_K|x-x'|^{1/6}\); also
\(p^{2/3}\le p^{4/3}\). Thus the weaker common exponents
in \eqref{st:eq:endpointmoment} bound both types of increment.

Choose a slightly larger closed coordinate box \(K^+\) containing
\(K\) in its interior, still with \(t-s\) bounded away from zero.
For large \(N\), this box contains all neighbouring row-grid points
used to interpolate points of \(K\). Write \(\Delta=N^{-1}\) and
\[
 \mathcal G_N=K^+\cap
 (\RR\times\Delta\ZZ\times\RR\times\Delta\ZZ).
\]
For two points in \(\mathcal G_N\), change their four coordinates
one at a time. All intermediate points remain in the coordinate
box \(K^+\), so the triangle inequality and the preceding
coordinate estimates give
\[
 \|F_N^h(z)-F_N^h(z')\|_p
 \le C_Kp^{4/3}|z-z'|^{1/6},\qquad z,z'\in\mathcal G_N.
\]

We now pass to the separately linear interpolation defining \(F_N^h\).
For \(z\in K\), let \(z_1,\ldots,z_4\in\mathcal G_N\) be the
four points obtained by rounding each of its two row coordinates
to the adjacent mesh values, with its spatial coordinates unchanged.
There are deterministic weights \(w_i(z)\ge0\), summing to one, such that
\[
 F_N^h(z)=\sum_{i=1}^4w_i(z)F_N^h(z_i).
\]
Consequently, for \(z,z'\in K\),
\[
 F_N^h(z)-F_N^h(z')
 =\sum_{i,k=1}^4w_i(z)w_k(z')
       \bigl(F_N^h(z_i)-F_N^h(z'_k)\bigr).
\]
If \(|z-z'|\ge\Delta\), each \(|z_i-z'_k|\) is at most
\(C|z-z'|\), and Minkowski's inequality gives the desired bound.

For smaller increments, we again change one coordinate at a time.
Consider first a row-coordinate change with
\(s_-\le s\le s+\delta\le s_+=s_-+\Delta\), where
\(s_-,s_+\in\Delta\ZZ\), and initially take \(t\in\Delta\ZZ\).
Linear interpolation gives
\[
 F_N^h(x,s+\delta;y,t)-F_N^h(x,s;y,t)
 =\frac{\delta}{\Delta}
   \bigl(F_N^h(x,s_+;y,t)-F_N^h(x,s_-;y,t)\bigr),
\]
and therefore
\[
 \|F_N^h(x,s+\delta;y,t)-F_N^h(x,s;y,t)\|_p
 \le C_Kp^{4/3}\frac{\delta}{\Delta}\Delta^{1/6}
 \le C_Kp^{4/3}\delta^{1/6}.
\]
If \(t\) is between mesh values, the same increment is a convex
combination of the two corresponding increments at adjacent values
of \(t\), so this bound still holds. An increment crossing a cell
boundary is split there into two pieces, and the other row
coordinate is treated in the same way. For a spatial-coordinate
change, the row interpolation weights are unchanged, so the
increment is a convex combination of spatial increments with the
row coordinates on the mesh. Applying the spatial estimate and
then summing the four coordinate changes completes the proof.
Stationarity makes every bound uniform in the deterministic
dynamical time \(h\).
\par\endgroup
\end{proof}

\begin{proof}[Proof of the tightness and marginal assertions in Theorem~\ref{cor:nontrivial}]
\par\begingroup
{Fix $T<\infty$ and a closed coordinate box $K\subset\RR^4_\uparrow$.
It suffices to prove tightness on such boxes, since their
interiors cover $\RR^4_\uparrow$. Choose a fixed larger closed
coordinate box $K^+\subset\RR^4_\uparrow$ containing $K$ in its
interior. This enlargement contains all grid points used to
interpolate points of $K$ once $n$ is sufficiently large.
There is $c_K>0$ such that $t-s\ge c_K$ on $K^+$.
All constants below may depend on these fixed boxes, but are
uniform in $n$ and $\tau,\sigma\in[-T,T]$.}

{For these $n$, the fields $\widehat{\sL}_n$ are continuous on
$[-T,T]\times K$: the Brownian paths are jointly continuous,
passage values {maximise} a continuous weight over the compact
staircase space, and the interpolation formulas agree on common
mesh boundaries.}
\par\endgroup

\textbf{{Dynamical increments.}}
\par\begingroup
{We apply Lemma~\ref{lem:rotation} to the grid endpoint
pairs defined below.} Define the set of grid endpoint pairs by
\[
 \mathcal G_n=K^+\cap
 \bigl(\RR\times n^{-1}\ZZ\times\RR\times n^{-1}\ZZ\bigr).
\]
For $z=(x,s;y,t)\in\mathcal G_n$, the endpoints in
\eqref{eq:ourescaling} have vertical separation
$D=n(t-s)\ge c_Kn$ and horizontal separation
\[
 0<L=n(t-s)+2n^{2/3}(y-x)\le C_Kn.
\]
\par\endgroup
{The deterministic terms $2n(t-s)+2n^{2/3}(y-x)$ in
\eqref{eq:ourescaling} do not depend on the dynamical time
$\tau$. Thus subtracting the values at $\tau$ and $\sigma$
leaves $n^{-1/3}$ times the corresponding passage-time
increment. Applying the general-endpoint bound above at elapsed
OU time $n^{-1/3}|\tau-\sigma|$ gives, for $p\ge2$,}
\begin{equation}\label{eq:oudynamicmoment}
 \|\sL_n^\tau(z)-\sL_n^\sigma(z)\|_p
 \le C_K n^{-1/3}\sqrt{pn\,n^{-1/3}|\tau-\sigma|}
 =C_K\sqrt p\,|\tau-\sigma|^{1/2}.
\end{equation}
\par\begingroup
For $z\in K$, write
$\widehat{\sL}_n^\tau(z)=\sum_{i=1}^4w_i(z)\sL_n^\tau(z_i)$,
where $z_i\in\mathcal G_n$ are its four {neighbouring} grid
endpoint pairs, $w_i(z)\ge0$ and $\sum_iw_i(z)=1$.
The weights do not depend on $\tau$, so Minkowski's inequality
and \eqref{eq:oudynamicmoment} give
\[
 \|\widehat{\sL}_n^\tau(z)-\widehat{\sL}_n^\sigma(z)\|_p
 \le\sum_{i=1}^4w_i(z)
       \|\sL_n^\tau(z_i)-\sL_n^\sigma(z_i)\|_p
 \le C_K\sqrt p\,|\tau-\sigma|^{1/2}.
\]
\par\endgroup

\par\begingroup
\textbf{{Endpoint increments.}}
At each deterministic $\tau$, the environment has the static
Brownian law. Apply Lemma~\ref{st:lem:endpointinput} to
$\widehat{\sL}_n^\tau(z)=F_n^{\tau n^{-1/3}}(z)$ to obtain

\begin{equation}\label{eq:oustaticmoment}
 \|\widehat{\sL}_n^\tau(z)-\widehat{\sL}_n^\tau(z')\|_p
 \le C_{K,p}\|z-z'\|^{1/6},
 \qquad z,z'\in K,\quad p\ge2.
\end{equation}
{Here $\|\cdot\|$ is the Euclidean norm.}
The constant is uniform in sufficiently large $n$ and,
by stationarity, in deterministic $\tau$.
\par\endgroup

\textbf{{Joint tightness and identification of the marginals.}}
{For every $z,z'\in K$, $\tau,\sigma\in[-T,T]$ and
$p\ge2$, the triangle inequality, with intermediate value
$\widehat{\sL}_n^\sigma(z)$, gives}
\begin{equation}\label{eq:oujointmoment}
 \|\widehat{\sL}_n^\tau(z)-\widehat{\sL}_n^\sigma(z')\|_p
 \le C_{K,p}\bigl(
       |\tau-\sigma|^{1/2}+\|z-z'\|^{1/6}\bigr).
\end{equation}
Writing $d=\| (\tau,z)-(\sigma,z')\|$, we obtain
\[
 \EE\left|\widehat{\sL}_n^\tau(z)
              -\widehat{\sL}_n^\sigma(z')\right|^p
 \le C_{K,p}d^{p/6},\qquad d\le1.
\]
Choose $p>30$, so that $p/6>5$, the number of parameters.
The Kolmogorov--Chentsov tightness criterion gives uniform
control of local H\"older moduli of every exponent less than
$1/6-5/p$. Together with tightness at one point in each parameter
box, supplied by the static one-point bounds, this gives
tightness on $[-T,T]\times K$. A countable compact exhaustion
and a diagonal subsequence give locally uniform subsequential
convergence on $\RR\times\RR^4_\uparrow$.

For each fixed $\tau$, evaluation at $\tau$ is continuous in
this topology, and \cite[Theorem 11.1]{DOV22} identifies the
limiting field $\sL^\tau$ as a directed landscape. Finally,
stationarity of the OU environment makes the joint law of
$\widehat{\sL}_n$ invariant under translations of $\tau$.
Such translations are continuous for locally uniform
convergence, so this invariance passes to the limit.
\end{proof}
\par\endgroup

\begingroup
\begin{proof}[Proof of the reversibility assertion in Theorem~\ref{cor:nontrivial}]
The stationary OU environment is invariant in law under
$r\mapsto-r$: this follows directly from its {centered} Gaussian
law and the covariance $e^{-|r-r'|}$ between matching Brownian
increments. Passage values at each time are measurable functions
of that time's environment, and the interpolation in
\eqref{eq:ourescaling} does not mix dynamical times.
Consequently, $\widehat{\sL}_n^\tau$ and
$\widehat{\sL}_n^{-\tau}$ have the same joint law. The map
$f(\tau,z)\mapsto f(-\tau,z)$ is continuous in the local
uniform topology, so this invariance passes to every
subsequential limit.
\end{proof}

\subsection{Increment bounds and non-triviality}\label{sec:nontriviality}
Tightness and the static marginal law by themselves permit a
process that samples one directed landscape and keeps it fixed.
We now use Theorem~\ref{thm:main} to bound the mean-square change
from below at every sufficiently small positive dynamical time.
The remaining sections prove that overlap theorem.

\begin{proof}[Proof of the increment bounds in Theorem~\ref{cor:nontrivial}, assuming Theorem~\ref{thm:main}] 
{The finite-time covariance identity in \cite[Section 4.4]{GH24}
and mean-overlap monotonicity in \cite[Theorem 4.1(2)]{GH24} give}
\begin{align}
 \EE({T_{\mathbf{0}}^{\mathbf{n},t}}-{T_{\mathbf{0}}^{\mathbf{n},0}})^2
    &=2\int_0^t e^{-s}\EE O_n(s)\,ds,
       \label{eq:covidentity}\\
 s&\longmapsto\EE O_n(s)\quad\hbox{is non-increasing}.
       \label{eq:covmonotone}
\end{align}
\begingroup
Taking $\kappa_*=1/2$ in Theorem~\ref{thm:main}, choose a
sufficiently small $a_*>0$ so that, with $b_*=1/4$,
\begin{equation}\label{st:eq:target}
 \liminf_{n\to\infty}
       \frac{\EE O_n(a_*n^{-1/3})}{n}\ge b_*.
\end{equation}
Fix $0<h\le a_*$ and set
\par\endgroup
\[
 X_n(h)=n^{-1/3}({T_{\mathbf{0}}^{\mathbf{n},h n^{-1/3}}}-{T_{\mathbf{0}}^{\mathbf{n},0}})
       =\widehat{\sL}_n^h(0,0;0,1)
           -\widehat{\sL}_n^0(0,0;0,1).
\]
On the integration interval $0\le s\le h n^{-1/3}$,
monotonicity gives $\EE O_n(s)\ge\EE O_n(a_*n^{-1/3})$.
Also $O_n(s)\le n$ deterministically. Therefore
\begin{align*}
 2n^{-2/3}(1-e^{-h n^{-1/3}})\EE O_n(a_*n^{-1/3})
 &\le\EE X_n(h)^2\\
 &\le2n^{1/3}(1-e^{-h n^{-1/3}})\le2h.
\end{align*}
Since $n^{1/3}(1-e^{-h n^{-1/3}})\to h$,
\eqref{st:eq:target} implies
\begin{equation}\label{eq:secondlower}
 \liminf_{n\to\infty}\EE X_n(h)^2\ge2b_*h.
\end{equation}

To pass this bound to a subsequential limit, we use the
case $p=4$ of Lemma~\ref{lem:rotation}:
\begin{equation}\label{eq:fourth}
 \sup_n\EE|X_n(h)|^4\le Ch^2.
\end{equation}
In particular,
\[
 \sup_n\EE\bigl[X_n(h)^2\ind_{\{X_n(h)^2>K\}}\bigr]
 \le Ch^2/K\longrightarrow0\qquad(K\to\infty).
\]
Thus the squares are uniformly integrable. Along any subsequence
whose fields converge in law, $X_n(h)$ converges in law to
\[
 X(h)=\cL^h(0,0;0,1)-\cL^0(0,0;0,1).
\]
For fixed $K$, bounded continuous convergence applies to
$\min(X_n(h)^2,K)$. The uniform-integrability bound then
allows $K\to\infty$, giving
$\EE X_n(h)^2\to\EE X(h)^2$ along the subsequence.
Consequently,
\[
 2b_*h\le\EE X(h)^2\le2h,\qquad 0<h\le a_*.
\]
Stationarity gives the same bounds for increments from any
deterministic time $\tau$. This proves \eqref{eq:limitlower}
with $a_0=a_*$ and $c=2b_*$, and hence non-triviality.
\end{proof}
\par\endgroup

\section{\texorpdfstring{{Excursions and the reduction to long excursions}}{Excursions and the reduction to long excursions}}\label{sec:excursionreduction}

\begingroup
An excursion $E$ of $\Gamma^t$ about $\Gamma^0$ consists of
the portions of the two paths between common points
$u=(u_x,i)$ and $v=(v_x,k)$, with the portions disjoint
between these endpoints. {Here $i,k\in\ZZ$: between consecutive integer rows both paths
are vertical, so they cannot separate or first meet strictly
inside that strip.} Its lifetime is $[i,k]$ and its
row duration is $d_E=k-i$. Necessarily $i<k$: if both
endpoints were on the same row, both paths would traverse
the same horizontal interval between them. The lifetimes
of distinct excursions have disjoint interiors, and thus we must always have
\[
 \sum_E d_E\le n.
\]
Figure~\ref{fig:excursiondurations} illustrates excursions
at several different scales.
\par\endgroup

\begin{figure}[tbp]
\centering
\begin{tikzpicture}[x=.59cm,y=.37cm]
 \draw[->,black!45] (-1,0)--(-1,14.7) node[above] {row};
 \draw[->,black!45] (0,-1)--(15,-1) node[right] {$x$};
 \draw[black!70,line width=1.4pt] (0,0)
  \foreach \r in {1,...,14} {--({\r-1},\r)--(\r,\r)};
 \draw[staticpath] (2,2)--(2,3)--(3,3);
 \draw[dynamicpath] (2,2)--(3,2)--(3,3);
 \draw[staticpath] (5,5)--(5,6)--(6,6)--(6,7)--(7,7)
  --(7,8)--(8,8)--(8,9)--(9,9);
 \draw[dynamicpath] (5,5)--(7,5)--(7,6)--(8,6)--(8,7)--(9,7)--(9,9);
 \draw[staticpath] (11,11)--(11,12)--(12,12)--(12,13)--(13,13);
 \draw[dynamicpath] (11,11)--(13,11)--(13,13);
 \foreach \x/\y in {0/0,2/2,3/3,5/5,9/9,11/11,13/13,14/14}
  \fill (\x,\y) circle (1.8pt);
 \node[below left] at (0,0) {${\mathbf{0}}$};
 \node[above right] at (14,14) {${\mathbf{n}}$};
 \foreach \r in {2,3,5,9,11,13}
  \draw[guide] (\r,\r)--(16,\r);
 \draw[<->] (16.3,2)--node[right] {$d_1$}(16.3,3);
 \draw[<->] (16.3,5)--node[right] {$d_2$}(16.3,9);
 \draw[<->] (16.3,11)--node[right] {$d_3$}(16.3,13);
 \node[anchor=east] at (1.5,2.5) {$E_1$};
 \node[anchor=east] at (5,7) {$E_2$};
 \node[anchor=east] at (10.6,12) {$E_3$};
 \draw[staticpath] (1,16)--(2.2,16);
 \node[anchor=west] at (2.4,16) {$\Gamma^0$};
 \draw[dynamicpath] (6,16)--(7.2,16);
 \node[anchor=west] at (7.4,16) {$\Gamma^t$};
 \draw[black!70,line width=1.4pt] (11,16)--(12.2,16);
 \node[anchor=west] at (12.4,16) {common portions};
\end{tikzpicture}
\begingroup
\caption{Excursions of different row durations.
Between consecutive common portions, the blue and orange paths
separate and then meet again. The three lifetimes have disjoint
interiors and durations $d_1,d_2,d_3$. These durations count rows;
horizontal non-overlap is instead the total length of the blue
horizontal pieces not shared with the orange path.}
\label{fig:excursiondurations}
\endgroup
\end{figure}
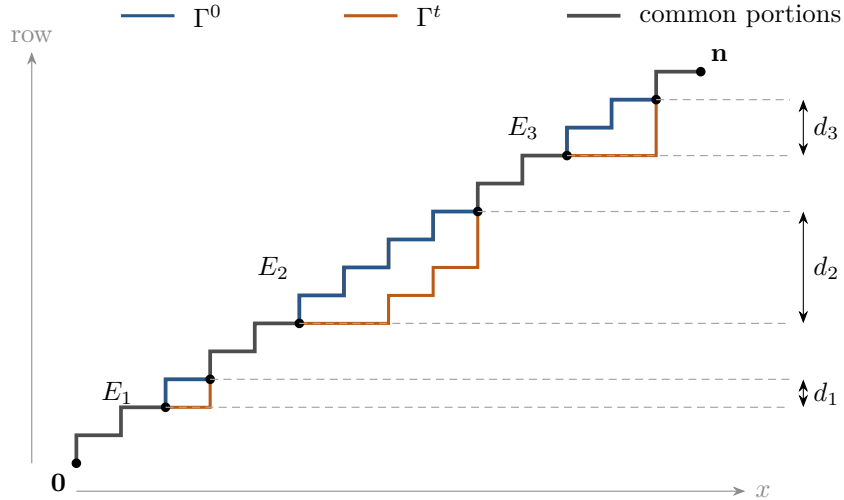
{For a cutoff exponent $\beta>0$, we call an}
excursion long if \(d\ge n^\beta\). Write
\(D_{\mathrm{long}}\) for the sum of the row durations of long
excursions.
{Using the short-excursion estimates of \cite{GH24}, the following
lemma shows that, to control the missing overlap, it suffices to
bound the total duration of long excursions.}
\begingroup
\begin{lemma}[Reduction to long excursions]\label{lem:fullreduction}
\begingroup
Fix $0<\beta<1/12$. There is a constant $C>0$ such that,
for every $\varepsilon>0$,
\begin{equation}\label{st:eq:GHreduction}
 \lim_{n\to\infty}\ \sup_{0\le t\le n^{-1/3}}
 \PP\!\left(n-O_n(t)>C D_{\rm long}(t)+\varepsilon n\right)=0.
\end{equation}
\par\endgroup
\end{lemma}

{Before proving Lemma~\ref{lem:fullreduction}, we record the
two static inputs it uses; both will also be needed in
Section~\ref{sec:stability}.}
\begingroup
\begingroup
We first record a uniform weight estimate that will be useful for us throughout the paper. For a scale $N>0$ and ordered endpoints $p=(x,i)$, $q=(y,k)$
{in $\RR\times\ZZ$}, with $i<k$, write
\[
 {W_N^0}(p,q)=N^{-1/3}\bigl[T_p^{q,0}-(y-x)-(k-i)\bigr].
\]
Thus ${W_N^0}(p,q)$ is the {centered}, {normalised} passage value,
{optimised} over all staircases between the two endpoints.
\par\endgroup

\begingroup
\begin{lemma}[Uniform weights in a compact region; {\cite[Proposition 1.8]{GH23}}]\label{sl:lem:compactweightinput}
{There are absolute constants $c,C>0$ such that
the following holds for every integer $N\ge1$ whenever}
\[
 1\le R\le cN^{1/46},\qquad C\le\zeta\le cN^{1/30}.
\]
{Let \(\mathcal Q\) be the set of ordered endpoint pairs}
\(p=(x,i),q=(y,k)\) satisfying

\begingroup
\[
  \begin{gathered}
 {i,k\in[0,3N]\cap\mathbb Z,\qquad {N/2\le k-i\le N,}}\\
 {\max\{|x-i|,|y-k|\}\le4RN^{2/3}.}
 \end{gathered}
\]
\par\endgroup
 {Then}\footnote{{The term \(R^2\) in the lemma
accounts for the parabolic loss: a transverse displacement of order
\(RN^{2/3}\) over \(N\) rows produces a {normalised} correction of order
\(N^{-1/3}(RN^{2/3})^2/N=R^2\).}}
\begingroup
\[
 \PP\left(\sup_{(p,q)\in\mathcal Q}|{W_N^0}(p,q)|
         >C\bigl(\zeta+R^2\bigr)\right)
 \le CR^2e^{-c\zeta^{3/2}}.
\]
\par\endgroup
\end{lemma}
\begin{proof}
Apply the map
\[
 (x,m)\longmapsto
 {\left(\frac{x-m}{2(4N)^{2/3}},\frac{m}{4N}\right).}
\]
{The endpoint times then lie in \([0,3/4]\),
their gaps lie in \([1/8,1/4]\), and their horizontal coordinates
have magnitude at most \(2\cdot4^{-2/3}R\).
Apply \cite[Proposition 1.8]{GH23} with \(\ell=2\), and also
with \(\ell=3\) to include the boundary gap \(1/8\), taking
the spatial parameters \(M,L\) of order \(R\).}
After the fixed conversion of weights, these estimates
control the parabolically adjusted weight in absolute value at
threshold \(C\zeta\). In original coordinates, the correction
added to \({W_N^0}(p,q)\) satisfies, throughout \(\mathcal Q\),
\begingroup
\[
 0\le\frac{\bigl((y-x)-(k-i)\bigr)^2}{4N^{1/3}(k-i)}
 \le{32R^2}.
\]
\par\endgroup
Absorbing this deterministic term into the weight threshold gives
the stated bound. {The restrictions
\(R\le cN^{1/46}\) and \(\zeta\le cN^{1/30}\) ensure that
the spatial parameters and weight threshold lie in the ranges
of the cited estimates.}
\end{proof}
\par\endgroup

\begingroup
{The following specialization of the static regularity estimates
supplies the horizontal-span bounds needed in the proof of
Lemma~\ref{lem:fullreduction}. It will also control the possible
endpoints of local segments in Section~\ref{sec:stability} and
their weights in Subsection~\ref{sec:slenderapplication}.}
All pairs of points in the statement are allowed simultaneously,
so its applications require no union over individual rows.
\par\endgroup

\begingroup
\begin{lemma}[Simultaneous geodesic and weight regularity;
{\cite[Theorems 1.4(1), 1.6(1), Corollary 1.5 and Proposition 1.8]{GH23}}]
\label{st:lem:regularityinput}
Fix $\beta\in(0,1)$ and $c_0\in(0,1/8)$. There are constants
{$c,C,R_0>0$ such that, for all sufficiently large $n$ depending
on $\beta,c_0$, and every $R$ satisfying
$R_0\le R\le(c_0n^\beta)^{1/64}$, the following holds at any two
deterministic dynamical times $0,t$.}
\begingroup
Let $\cG_{n,R}$ be the event on which, for both $s\in\{0,t\}$,
\begin{equation}\label{st:eq:globalconfinement}
 \sup_{(x,i)\in\Gamma^s}|x-i|\le Rn^{2/3},
\end{equation}
\par\endgroup
and, for every $j\in\cI_n$ and every pair
$(x,i),(y,k)\in\Gamma^s$ with {$i,k\in\ZZ$} and
$c_0N_j\le k-i\le2N_j$, 
\begin{equation}\label{st:eq:georeg}
 |(y-k)-(x-i)|\le H_jN_j^{2/3},
 \qquad H_j=C_{c_0}R(1+j)^{1/3},
\end{equation}

and
\begin{equation}\label{st:eq:weightreg}
 \left|T_{(x,i)}^{(y,k),s}-(y-x)-(k-i)\right|
 \le C_{c_0}R^2(1+j)^{2/3}N_j^{1/3}.
\end{equation}
Then
\begin{equation}\label{st:eq:georegprob}
 \PP(\cG_{n,R}^c)\le Ce^{-cR^3}.
\end{equation}
\end{lemma}
\begin{proof}
\begingroup
After a fixed rescaling of its parameter, Corollary 1.5 of the
cited paper gives
\[
 \PP\!\left(\sup_{(x,i)\in\Gamma^s}|x-i|>Rn^{2/3}\right)
 \le Ce^{-cR^3}
\]
at each deterministic dynamical time $s$. This gives
\eqref{st:eq:globalconfinement}; the rescaling is absorbed into
$c,C$ and $R_0$.
\par\endgroup
For a dyadic row separation in $(n2^{-k-1},n2^{-k}]$, its
Theorems 1.4(1) and 1.6(1) bound transverse increments and
absolute {centered} weights by
$CR(n2^{-k})^{2/3}(1+k)^{1/3}$ and
$CR^2(n2^{-k})^{1/3}(1+k)^{2/3}$, respectively, with failure
probability $Ce^{-cR^3k}$ for $k\ge1$.
The change to raw coordinates only changes constants.
{For gaps at most $n/2$, each range
$[c_0N_j,2N_j]$ meets a bounded number of these dyadic classes,
depending only on $c_0$, with $1+k\asymp_{c_0}1+j$.
Summing their failure probabilities over $j$ costs at most
$Ce^{-cR^3}$.}

\begingroup
It remains to treat row gaps in $(n/2,n]$. Global confinement
{\eqref{st:eq:globalconfinement}} puts their endpoints in the strip {$|x-i|\le Rn^{2/3}$.}
Apply Lemma~\ref{sl:lem:compactweightinput} with $N=n$,
spatial parameter of order $R$, and $\zeta$ of order $R^2$.
Its failure probability is $CR^2e^{-cR^3}\le C'e^{-c'R^3}$,
and both terms in its weight threshold are of order $R^2$.
This gives \eqref{st:eq:weightreg} for these gaps, while global confinement gives \eqref{st:eq:georeg}.
\par\endgroup

\begingroup
We check the parameter ranges uniformly in the allowed values of $R$.
The weight estimate \cite[Theorem 1.6(1)]{GH23} explicitly requires
$r\le(nh_{1,2})^{1/64}$, where $nh_{1,2}$ is the row separation.
Since our separations are at least $c_0n^\beta$, this is ensured by
$R\le(c_0n^\beta)^{1/64}$. This bound also implies $R\le n^{1/10}$,
the restriction in the geometric estimates of Theorem 1.4(1) and
Corollary 1.5. The relevant dyadic indices satisfy
$2^k\le c_0^{-1}n^{1-\beta}$, so $2^k\le hn$ for large $n$.
Finally, $R=O(n^{\beta/64})$ and $R^2=O(n^{\beta/32})$ satisfy
the bounds on the spatial parameter and weight threshold in
Lemma~\ref{sl:lem:compactweightinput}, since
$\beta/64<1/46$ and $\beta/32<1/30$. Apply these static bounds
at both deterministic times and take their union.
\par\endgroup
\end{proof}
\par\endgroup
\par\endgroup
{We are now ready to prove Lemma~\ref{lem:fullreduction}.}
\begin{proof}[{Proof of Lemma~\ref{lem:fullreduction}}]
\begingroup
{Fix $\varepsilon>0$ and choose fixed $\delta,g\in(0,1/4)$
so that $2\delta+2g<\varepsilon/2$. These parameters will control
the losses from short thin excursions and from the boundary strips,
respectively. Let $t_n\in[0,n^{-1/3}]$ be any deterministic
sequence, and put $L_n=\lceil n^\beta\rceil$.}
The overlap $O_n(t_n)$ is a horizontal length, whereas
$D_{\rm long}(t_n)$ counts row durations. We first record the
horizontal-span bound used to relate these quantities.
\par\endgroup

\begingroup
\paragraph{\textbf{Horizontal spans and missing overlap.}}
\begingroup
Let $\mathcal R_n$ be the event on which, for both $s\in\{0,t_n\}$
and all ordered points $(x,i),(y,k)\in\Gamma^s$ {with $i,k\in\ZZ$},
\begingroup
\begin{equation}\label{eq:horizontalspan}
 y-x\le 2\bigl((k-i)\vee L_n\bigr),
 \qquad |x-i|\le n^{3/4}.
\end{equation}
\par\endgroup
\begingroup
{Lemma~\ref{st:lem:regularityinput}, applied with $R=\log n$ and
$c_0=1/16$, gives $\PP(\mathcal R_n)=1-o(1)$ uniformly in the
deterministic choice of $t_n\le n^{-1/3}$.}\footnote{{For a row gap
$d\ge L_n$, choose a retained scale with $d\in[N_j/2,N_j]$.
By \eqref{st:eq:georeg}, the transverse error divided by $d$
is at most $C(\log n)^{4/3}L_n^{-1/3}=o(1)$, so $y-x\le2d$.
For $d<L_n$, extend the segment along the same geodesic to row
duration $L_n$ and use monotonicity. Finally,
\eqref{st:eq:globalconfinement} gives
$|x-i|\le n^{2/3}\log n\le n^{3/4}$.}}
{The first bound in \eqref{eq:horizontalspan} converts row duration
into horizontal length; the second controls the loss near the two
ends of the full path. We first treat long excursions, then handle
short excursions according to thinness and proximity to the endpoint rows.}
\par\endgroup
\par\endgroup

For an excursion $E$ with endpoints $u=(u_x,i)$ and
$v=(v_x,k)$, denote the portions of the two paths between
these endpoints by $\Gamma_E^0$ and $\Gamma_E^{t_n}$.
Their total horizontal lengths are both $v_x-u_x$, since
their horizontal coordinates are nondecreasing. Hence, on
$\mathcal R_n$,
\[
 \sum_{E\,\mathrm{long}}
   |\Gamma_E^0\setminus\Gamma_E^{t_n}|_{\hor}
 \le \sum_{E\,\mathrm{long}}(v_x-u_x)
 \le {2}\sum_{E\,\mathrm{long}}(k-i)
 ={2}D_{\rm long}(t_n).
\]
This is the contribution of long excursions to the missing
horizontal overlap $n-O_n(t_n)$.
\par\endgroup

\begingroup
\paragraph{\textbf{Short non-thin excursions.}}
{Following \cite[Section 8.2]{GH24}, and using the fixed $\delta$
chosen at the start of the proof, we call an excursion with
lifetime $[i,k]$ $\delta$-thin}\footnote{{Thinness uses a fixed microscopic threshold at every row.
In contrast, slenderness, introduced in Section~\ref{sec:assembly},
requires separation below the scale-dependent width
$w_j=\theta_jN_j^{2/3}$ on most rows. Thinness controls missing
horizontal overlap directly, whereas slenderness describes
excursions that are narrow relative to their natural transversal scale.}} when
\[
 \max_{m\in\{i,\ldots,k-1\}}
       |\Gamma^{t_n}(m)-\Gamma^0(m)|<\delta.
\]
In \cite[Proposition 8.4]{GH24}, the fixed duration threshold
$\beta_2/8$ may be replaced by any fixed $d>0$, and the microscopic
thinness threshold may be any fixed positive constant. The same
proof applies, with constants depending on these choices. Thus the
{total row duration of short non-thin excursions, divided by
$n$, converges to zero in probability: for every fixed $\varepsilon>0$,
the probability that this duration exceeds $\varepsilon n$ tends
to zero.}

{To control horizontal length, we use a counting consequence of the
same proof. Let $M_n$ count the short excursions that are not
$\delta$-thin and whose lifetimes lie in $[gn,(1-g)n]$. We will
justify the following convergence in the next paragraph:}
\begin{equation}\label{eq:shortunlucky}
 {\frac{M_n}{n^{1-\beta}}\longrightarrow0\qquad\text{in probability}.}
\end{equation}

{By \cite[Lemma 10.2]{GH24}, with probability $1-o(1)$,
uniformly in $t_n$, every short excursion has transverse width
at most $Cn^{2\beta/3}(\log n)^{1/3}$ in original coordinates.
On this event, use witness separations between $\delta$ and this
upper bound, and set $\lambda=1/12$. Every excursion counted by
$M_n$ is then unlucky or normal. Roughly, the unlucky case admits
a separated static competitor of unusually small deficit, whereas
the normal case forces a definite deficit. For the precise definitions,
see \cite[Section 10, before Lemmas 10.3 and 10.4]{GH24}.
With these fixed choices of separation constants,
\cite[Lemma 10.3]{GH24} bounds the unlucky count}
by $n^{1-\lambda+o(1)}$ with probability $1-o(1)$. The result
{\cite[Lemma 10.5]{GH24}, with its fixed count coefficient
$2^{-5}\beta_2$ replaced by any $\varepsilon>0$ (the proof works
verbatim), shows that the number of normal excursions, divided by
$n^{1-\beta}$, converges to zero in probability.} Since $\beta<\lambda$ and
$\beta+\lambda<1/6$, this proves \eqref{eq:shortunlucky}.
{By \eqref{eq:horizontalspan}, on $\mathcal R_n$ each counted
excursion has horizontal span at most $2L_n$. Their total contribution
to $n-O_n(t_n)$ is therefore at most $2L_nM_n$. Since
$L_n=\lceil n^\beta\rceil$, \eqref{eq:shortunlucky} gives}
\[
 {\frac{2L_nM_n}{n}\longrightarrow0\qquad\text{in probability}.}
\]
\par\endgroup

\begingroup
\paragraph{\textbf{{Excursions that are short and thin.}}}
Fix a $\delta$-thin excursion $E$ from $u=(u_x,i)$ to
$v=(v_x,k)$. To describe the portions of the two geodesics
between $u$ and $v$, set, for $s\in\{0,t_n\}$,
\[
 z_{i-1}^s=u_x,\qquad
 z_m^s=\Gamma^s(m)\quad(i\le m<k),\qquad
 z_k^s=v_x.
\]
Their horizontal intervals on row $m$ are
$I_m^s=[z_{m-1}^s,z_m^s]$, for $i\le m\le k$.
Writing $\Delta_m=|z_m^0-z_m^{t_n}|$, the common endpoints
and thinness give
\[
 \Delta_{i-1}=\Delta_k=0,\qquad
 \Delta_m<\delta\quad(i\le m<k).
\]
The elementary interval inequality
$\Leb([a,b]\setminus[a',b'])\le |a-a'|+|b-b'|$ now gives
\[
 \begin{aligned}
 |\Gamma_E^0\setminus\Gamma_E^{t_n}|_{\hor}
 &=\sum_{m=i}^k\Leb(I_m^0\setminus I_m^{t_n})\\
 &\le\sum_{m=i}^k(\Delta_{m-1}+\Delta_m)
 =2\sum_{m=i}^{k-1}\Delta_m
 \le2\delta(k-i).
 \end{aligned}
\]
Consequently, all short thin excursions together contribute at most
\[
 \sum_{E\,\mathrm{short\ and\ thin}}
    |\Gamma_E^0\setminus\Gamma_E^{t_n}|_{\hor}
 \le2\delta\sum_{E\,\mathrm{short\ and\ thin}}d_E
 \le2\delta n
\]
to $n-O_n(t_n)$. This estimate is deterministic.
\par\endgroup

\paragraph{\textbf{{Boundary strips.}}}
The only short excursions still unaccounted for have lifetimes
\begingroup
not contained in $[gn,(1-g)n]$. Put
$m_-=\lceil gn\rceil$ and $m_+=\lfloor(1-g)n\rfloor$.
The total horizontal lengths of $\Gamma^0$ on rows $0,\ldots,m_-$
and $m_+,\ldots,n$ are respectively $\Gamma^0(m_-)$ and
$n-\Gamma^0(m_+-1)$, since the horizontal increments telescope.
The confinement bound included in $\mathcal R_n$ therefore
bounds their sum by $2gn+O(n^{3/4})$.
\par\endgroup

This bounds the loss from all excursions contained in these strips,
regardless of their number. Any remaining short excursion crosses
one of the two strip edges. Since lifetimes have disjoint interiors,
at most one crosses each edge, and $\mathcal R_n$ bounds their
combined horizontal span by $O(L_n)$.

\begingroup
{On $\mathcal R_n$, the preceding bounds give}
\[
 {1-\frac{O_n(t_n)}n
 \le 2\frac{D_{\rm long}(t_n)}n+2\delta+2g
       +\frac{2L_nM_n}{n}+O(n^{-1/4}+L_n/n).}
\]
{By \eqref{eq:shortunlucky}, the last two terms exceed
$\varepsilon/2$ with probability tending to zero. Since
{$2\delta+2g<\varepsilon/2$ and $\PP(\mathcal R_n^c)=o(1)$,
this proves \eqref{st:eq:GHreduction} with $C=2$.} All the estimates are
uniform over deterministic $t_n\in[0,n^{-1/3}]$, which gives
the supremum in the statement.}
\par\endgroup

\end{proof}

{In view of Lemma~\ref{lem:fullreduction}, to prove
Theorem~\ref{thm:main} it remains to show that}
\begin{equation}\label{st:eq:longtarget}
 \lim_{a\downarrow0}\limsup_{n\to\infty}
       \frac{\EE D_{\rm long}(a n^{-1/3})}{n}=0.
\end{equation}

\par\endgroup

\section{Uniform local passage-time stability}\label{sec:stability}
\begingroup
{The goal of this section is to make the passage-time stability
estimate of Lemma~\ref{lem:rotation} simultaneous over the random
endpoint pairs selected by excursions,
providing the local stability input in Figure~\ref{fig:argumentflow}.}
{For the physical dynamical time $t=an^{-1/3}$, the effective
elapsed time at scale $N_j$ is}
\begin{equation}\label{st:eq:tau}
 \tau_j=a(N_j/n)^{1/3}\asymp a2^{-j/3}.
\end{equation}
The same perturbation is therefore increasingly subcritical for
shorter segments. We retain this improvement while obtaining
uniformity over all their possible endpoints.
\par\endgroup

\begin{proposition}[Uniform local stability]\label{st:prop:local}
Fix \(0<\beta<1\) and \(c_0\in(0,1/8)\).
There are constants \(C,c,a_0>0\) such that, for each
\(0<a\le a_0\) and all sufficiently large \(n\), there is an
event \(\cE_{n,a}\) with
\begin{equation}\label{st:eq:goodprob}
 \PP(\cE_{n,a}^c)\le C\exp\{-c a^{-3/16}\}
\end{equation}
on which the following holds simultaneously: for every
\(j\in\cI_n\) and all \(u=(x,i),v=(y,k)\) on \(\Gamma^t\)
with {\(i,k\in\ZZ\)} and \(c_0N_j\le k-i\le2N_j\), 
\begin{equation}\label{st:eq:localstability}
 \abs{T_u^{v,t}-T_u^{v,0}}
 \le A_j:=C\tau_j^{1/4}N_j^{1/3}.
\end{equation}
The lower bound on \(n\) may depend on \(a,\beta,c_0\).
\end{proposition}

\begingroup
\begingroup
The parameter $c_0$ determines how many rows we discard near
each end of an excursion before choosing a witness row. In the
proof of Theorem~\ref{thm:main}, we take $c_0<\chi/32$, where
{$\chi$ is a constant appearing later in
Proposition~\ref{sl:prop:dynamic}.}
This ensures that the discarded rows leave enough witnesses to
obtain \eqref{eq:widepathwise}. Thus $c_0$ is fixed independently
of $n$ and $a$, but its choice later must allow for the value of $\chi$.
Constants in this section may depend on this fixed $c_0$.
\par\endgroup
\par\endgroup

\begingroup
Note that the fixed-endpoint scale in \eqref{st:eq:rotation} is
$\tau_j^{1/2}N_j^{1/3}$; here we allow the larger tolerance
$\tau_j^{1/4}N_j^{1/3}$. This leaves room for endpoint uniformity.
Since $\tau_j$ decreases geometrically, the resulting failure
probability improves rapidly enough to sum over all endpoint boxes
and all scales, as checked in \eqref{st:eq:scaleunion}. {This
gives the bound in terms of $a$ alone in \eqref{st:eq:goodprob},
uniformly for all sufficiently large $n$.}
\par\endgroup

\begingroup
{We first use Lemma~\ref{lem:rotation} on a fine mesh of
endpoint pairs covering one compact family. The mesh is chosen
so that the static oscillation between nearby endpoints is
small compared with the dynamical tolerance. {Keeping track of
how the moment bounds grow with $p$ allows us to choose $p$ as a
function of the elapsed time and obtain a stretched-exponential
bound for the maximum. We apply this comparison to deterministic
endpoint boxes and sum over all retained boxes and scales.
The resulting simultaneous event also covers the boxes containing
the endpoints selected by the geodesic.}}
\par\endgroup

\subsection{Uniform stability in a compact endpoint family}

\begingroup
{Fix a closed coordinate box
{$K\subset\RR^4_\uparrow$.} We use the field
$F_N^h(z)$ from Section~\ref{sec:outightness}, including its
separately linear interpolation in the two row coordinates.
Lemma~\ref{lem:rotation} and the deterministic interpolation
weights give, for $z\in K$ and all sufficiently large $N$,}
\begin{equation}\label{st:eq:pointmoment}
 \norm{F_N^{\tau N^{-1/3}}(z)-F_N^0(z)}_p
 \le C_K\sqrt p\,\tau^{1/2},\qquad p\ge2,\quad 0<\tau\le1.
\end{equation}
At a fixed endpoint pair, this difference becomes small with
$\tau$. Further, the static estimate \eqref{st:eq:endpointmoment} controls
its oscillation when the endpoints vary. We combine the two by
starting with an endpoint mesh of width $\delta\asymp\tau^3$,
so that $\delta^{1/6}\asymp\tau^{1/2}$, and then controlling
the fluctuations between mesh points by successively finer meshes.

The lemma includes the translations and slopes needed for the
endpoint boxes below. For a deterministic origin $o=(x_o,i_o)$,
where $i_o\in\ZZ$, and a slope $\lambda>0$, set
\[
 \Phi_{N;o,\lambda}(x,s)
 =\bigl(x_o+\lambda(Ns+2N^{2/3}x),\ i_o+Ns\bigr).
\]
\begingroup
On the row grid $s,t\in N^{-1}\ZZ$, define
\[
 D_{N;o,\lambda}^{\tau}(x,s;y,t)
 =\frac{
 T_{\Phi_{N;o,\lambda}(x,s)}^{\Phi_{N;o,\lambda}(y,t),\,\tau N^{-1/3}}
 -T_{\Phi_{N;o,\lambda}(x,s)}^{\Phi_{N;o,\lambda}(y,t),\,0}
 }{\sqrt\lambda\,N^{1/3}},
\]
\endgroup
and interpolate linearly in the two row coordinates. For
$o=(0,0)$ and $\lambda=1$, this is
$F_N^{\tau N^{-1/3}}-F_N^0$; the deterministic centering cancels.

\begin{lemma}[Compact endpoint comparison]\label{st:lem:compact}
{Fix a closed coordinate box $K\subset\RR^4_\uparrow$.}
There are constants $C_K,c_K,N_K>0$
such that, for every deterministic $o\in\RR\times\ZZ$,
$\lambda>0$, $N\ge N_K$, $0<\tau\le1$ and $p\ge48$,
\begin{equation}\label{st:eq:compactmoment}
 \norm{\sup_{z\in K}|D_{N;o,\lambda}^{\tau}(z)|}_p
 \le C_Kp^{4/3}\tau^{1/2-12/p}.
\end{equation}
Consequently,
\begin{equation}\label{st:eq:compacttail}
 \PP\left(\sup_{z\in K}|D_{N;o,\lambda}^{\tau}(z)|
              >C_K\tau^{1/4}\right)
 \le C_K\exp\{-c_K\tau^{-3/16}\}.
\end{equation}
The constants are uniform in the deterministic origin and slope;
the probability bound is for each such choice.
\end{lemma}
\par\endgroup

\begin{proof}
\begingroup
The translated and scaled Brownian motions
\[
 \widetilde B_\ell^h(x)=\lambda^{-1/2}
 \bigl(B_{i_o+\ell}^h(x_o+\lambda x)-B_{i_o+\ell}^h(x_o)\bigr)
\]
have the same stationary OU law, with unchanged dynamical time.
Passage values in the original environment are $\sqrt\lambda$
times those in this environment. Hence it suffices to take
$o=(0,0)$ and $\lambda=1$; write $D_N=D_{N;(0,0),1}^{\tau}$.
\par\endgroup

\begingroup
Write $K=\prod_{r=1}^4[a_r,b_r]$. For each integer $k\ge0$,
let $\mathcal D_k$ be the grid obtained by dividing each
coordinate interval into $2^k$ equal pieces, including its
endpoints. These grids are nested, have at most $C2^{4k}$
points, and have Euclidean mesh size at most $C_K2^{-k}$.
Choose
\[
 k_* =\left\lceil\log_2(\tau^{-3})\right\rceil,
 \qquad \delta=2^{-k_*},\qquad
 \tfrac12\tau^3<\delta\le\tau^3.
\]
The symbol $k_*$ only records the first mesh level, chosen so
that its endpoint-increment scale $\delta^{1/6}$ is comparable
to the fixed-endpoint dynamical scale $\tau^{1/2}$.
The maximum over this first grid satisfies
\[
 \begin{aligned}
 \left\|\max_{z\in\mathcal D_{k_*}}|D_N(z)|\right\|_p
 &\le |\mathcal D_{k_*}|^{1/p}
       \max_{z\in\mathcal D_{k_*}}\|D_N(z)\|_p\\
 &\le C_K\sqrt p\,\delta^{-4/p}\tau^{1/2}
 \le C_K\sqrt p\,\tau^{1/2-12/p},
 \end{aligned}
\]
by \eqref{st:eq:pointmoment}.

This grid alone does not control points between its vertices.
We use finer grids to bound that remaining oscillation.
For $z\in\mathcal D_k$, let $\pi_{k-1}z\in\mathcal D_{k-1}$
be obtained by rounding each coordinate down to the preceding
coarser grid point. Then $|z-\pi_{k-1}z|\le C_K2^{-k}$.
Applying \eqref{st:eq:endpointmoment} {from Lemma~\ref{st:lem:endpointinput}} at both dynamical times gives
\[
 \|D_N(z)-D_N(\pi_{k-1}z)\|_p
 \le C_Kp^{4/3}2^{-k/6}.
\]
Since there are at most $C2^{4k}$ such differences,
\[
 \left\|\max_{z\in\mathcal D_k}
      |D_N(z)-D_N(\pi_{k-1}z)|\right\|_p
 \le C_Kp^{4/3}2^{-k(1/6-4/p)}.
\]
For any $z\in K$, let $z_k\in\mathcal D_k$ be its coordinatewise
lower grid approximation. These choices satisfy
$\pi_{k-1}z_k=z_{k-1}$ and $z_k\to z$. Continuity therefore
gives the telescoping bound
\[
 \sup_{z\in K}|D_N(z)-D_N(z_{k_*})|
 \le\sum_{k>k_*}\max_{w\in\mathcal D_k}
              |D_N(w)-D_N(\pi_{k-1}w)|.
\]
Since $p\ge48$ implies $1/6-4/p\ge1/12$, Minkowski's
inequality and the convergent geometric sum yield
\[
 \begin{aligned}
 \left\|\sup_{z\in K}|D_N(z)-D_N(z_{k_*})|\right\|_p
 &\le C_Kp^{4/3}\sum_{k>k_*}2^{-k(1/6-4/p)}\\
 &\le C_Kp^{4/3}\delta^{1/6-4/p}
 \le C_Kp^{4/3}\tau^{1/2-12/p}.
 \end{aligned}
\]
All constants are uniform in $p\ge48$. Adding the first-grid
bound proves \eqref{st:eq:compactmoment}. This refinement of
the mesh controls a supremum over a continuum; the interpolation
in the row coordinates has already been included in the
endpoint estimate \eqref{st:eq:endpointmoment}.
\par\endgroup

For the tail bound, take
\(p=\max\{48,\lceil\tau^{-3/16}\rceil\}\). Then
\(p^{4/3}\le C\tau^{-1/4}\), and
\[
 \tau^{-12/p}
 \le \exp\{12\tau^{3/16}\log(1/\tau)\}\le C.
\]
Thus the \(L^p\) norm in \eqref{st:eq:compactmoment}
is at most \(C_K\tau^{1/4}\). \begingroup
Markov's inequality gives
\[
 \begin{aligned}
 \PP\!\left(\sup_{z\in K}|D_N(z)|>eC_K\tau^{1/4}\right)
 &\le \frac{\EE\!\left[\bigl(\sup_{z\in K}|D_N(z)|\bigr)^p\right]}
              {(eC_K\tau^{1/4})^p}\\
 &\le e^{-p}\le\exp\{-\tau^{-3/16}\}.
 \end{aligned}
\]
Relabelling the constants proves \eqref{st:eq:compacttail}.
\par\endgroup
\end{proof}

The stretched-exponential bound uses the explicit growth of
the moment constants, rather than only their finiteness for
each fixed order. The exponent \(3/16\) comes from balancing
\(p^{4/3}\tau^{1/2}\) with the allowed error
\(\tau^{1/4}\); we do not attempt to {optimise} it here. The resulting
failure probability decreases rapidly enough to accommodate
the growing number of endpoint boxes at smaller scales.

\subsection{The endpoint boxes and the union over scales}

\begingroup
{To prove Proposition~\ref{st:prop:local}, we will obtain a
probability bound of the form}
\[
 \PP(\cE_{n,a}^c)
 \le Ce^{-cR^3}
       +C\sum_{j\in\cI_n}M_j e^{-c\tau_j^{-3/16}},
\]
where the first term is the cost of geodesic regularity and $M_j$
will count a deterministic collection of endpoint-box pairs at
scale $N_j$. The geometry below has two purposes: to bound $M_j$
by a polynomial in $2^j$, and to place every retained box pair
in the same compact endpoint family after translation and scaling.
Lemma~\ref{st:lem:compact} will then give the exponential term
for each pair, with uniform constants. The calculation is completed
in \eqref{st:eq:scaleunion}, where the increasingly strong decay
at smaller scales pays for the number of possible locations.
\par\endgroup

{Lemma~\ref{st:lem:compact} gives uniform passage-time stability
between deterministic endpoint boxes whose rescaled coordinates
lie in a fixed compact family. We will apply it to boxes at
longitudinal scale \(N\) and transverse scale \(N^{2/3}\).} A long geodesic may enter many
such boxes, and its local chord need not have slope exactly
one. {We use Lemma~\ref{st:lem:regularityinput} to limit the
possible box pairs.} {We then translate and rescale each retained pair, using the line
joining its box centers as the reference direction.} Figure~\ref{fig:endpointboxes} illustrates
the geometry. The estimates are applied to every retained
pair, before the geodesic selects any of them.

\paragraph{\textbf{A deterministic cover.}}
\begingroup
On the event $\cG_{n,R}$ of Lemma~\ref{st:lem:regularityinput},
whose probability is at least $1-Ce^{-cR^3}$, we will cover every
possible pair of endpoints selected from the geodesic {on integer rows} by a finite
deterministic list of box pairs. Lemma~\ref{st:lem:compact} applies
to fixed regions in KPZ coordinates, whereas
Lemma~\ref{st:lem:regularityinput} allows a transverse increment
of size $H_jN_j^{2/3}$, with $H_j$ growing with $j$.
{We choose the boxes so that, for each retained pair, translation
and scaling relative to the chord between their centers puts every
endpoint pair in one fixed compact set $K$, independent of the
scale and the box pair. Lemma~\ref{st:lem:compact} then applies
with the same constants to all retained pairs.}

Each box has transverse width $N_j^{2/3}$ and a smaller row
width, denoted by $\ell_j$, of order $N_j/H_j$.
The smaller row width ensures that a box still has transverse
width $O(N_j^{2/3})$ when measured relative to the chord joining
the two box centers. {Figure~\ref{fig:endpointboxes} illustrates
these boxes and the change of reference chord, which is verified
in Lemma~\ref{st:lem:commoncompact}.}

Recall from \eqref{st:eq:scales} that
$N_j=\lfloor n2^{-j}\rfloor$ and
$\cI_n=\{j\ge0:N_j\ge n^\beta\}$. Fix $j\in\cI_n$.
For a point $(x,i)$, write $X=x-i$ for its horizontal
displacement from the diagonal. Define
\[
 \ell_j=\left\lfloor\frac{c_0N_j}{16H_j}\right\rfloor,
 \qquad I_{j,r}=[r\ell_j,(r+1)\ell_j)\cap\ZZ,
 \qquad J_{j,s}=[sN_j^{2/3},(s+1)N_j^{2/3}),
\]
for $r,s\in\ZZ$, and
\[
 Q_{j,r,s}=\{(x,i):i\in I_{j,r},\ x-i\in J_{j,s}\}.
\]
Thus $I_{j,r}$ specifies a row bin and $J_{j,s}$ a bin for
the transverse coordinate $X$. For fixed $R$, the bounds
$N_j\ge n^\beta$ and $H_j\le C_{c_0}R(1+\log n)^{1/3}$
ensure that $c_0N_j/(16H_j)\ge2$ at every retained scale
for all sufficiently large $n$. In particular, $\ell_j\ge1$
and $\ell_j\asymp_{c_0}N_j/H_j$.
{For such $n$, the sets $I_{j,r}$ partition the rows $\ZZ$,
and the half-open intervals $J_{j,s}$ partition the transverse
coordinate line $\RR$. Consequently, the boxes $Q_{j,r,s}$,
$r,s\in\ZZ$, partition $\RR\times\ZZ$: each point $(x,i)$
belongs to exactly one box.}

Let $\mathfrak B_j$ be the set of ordered pairs
$b=(Q_b^-,Q_b^+)$ of these boxes that contain at least one
pair $(x,i)\in Q_b^-$, $(y,k)\in Q_b^+$ satisfying
\begingroup
\begin{equation}\label{st:eq:boxcompatibility}
 \begin{gathered}
 0\le i,k\le n,\qquad c_0N_j\le k-i\le2N_j,\\
 \max\{|x-i|,|y-k|\}\le Rn^{2/3},\qquad
 |(y-k)-(x-i)|\le H_jN_j^{2/3}.
 \end{gathered}
\end{equation}
The transverse bounds are those in \eqref{st:eq:globalconfinement}
and \eqref{st:eq:georeg} of Lemma~\ref{st:lem:regularityinput}.
\par\endgroup
The boxes and this compatibility test are deterministic.
In particular, $b$ denotes a \emph{pair of boxes}, not a pair
of points. Put $M_j=|\mathfrak B_j|$.

\begin{lemma}[Covering and counting endpoint-box pairs]\label{st:lem:boxcover}
For fixed $R$ and sufficiently large $n$, on the event
$\cG_{n,R}$ of Lemma~\ref{st:lem:regularityinput}, every
$u=(x,i),v=(y,k)\in\Gamma^t$ with {$i,k\in\ZZ$} and
$c_0N_j\le k-i\le2N_j$ belongs to a product
$Q_b^-\times Q_b^+$ with $b\in\mathfrak B_j$.
Moreover,
\begin{equation}\label{st:eq:boxcount}
 {M_j\le C_{c_0}R^2H_j^2(n/N_j)^{10/3}
     \le C_{c_0}R^4(1+j)^{2/3}2^{10j/3}.}
\end{equation}
\end{lemma}
\begin{proof}
{Since the boxes $Q_{j,r,s}$ partition $\RR\times\ZZ$, there
are unique boxes $Q_b^-$ and $Q_b^+$ containing $u=(x,i)$ and
$v=(y,k)$, respectively. The hypotheses on $u,v$ give the row
bounds in \eqref{st:eq:boxcompatibility}. On $\cG_{n,R}$,
\eqref{st:eq:globalconfinement} and \eqref{st:eq:georeg} give
the two transverse bounds. Thus $(u,v)$ satisfies
\eqref{st:eq:boxcompatibility}, so $b=(Q_b^-,Q_b^+)$ belongs
to $\mathfrak B_j$.}

\begingroup
For either endpoint, there are at most $Cn/\ell_j$
choices of a row bin meeting $[0,n]$, and at most
$CR(n/N_j)^{2/3}$ choices of a transverse bin meeting
{$[-Rn^{2/3},Rn^{2/3}]$.} Counting all ordered pairs of
these boxes therefore gives
\[
 M_j\le
 \left(C_{c_0}\frac{nH_j}{N_j}\,
             R\left(\frac n{N_j}\right)^{2/3}\right)^2
 \le C_{c_0}R^2H_j^2(n/N_j)^{10/3}.
\]
Using $H_j=C_{c_0}R(1+j)^{1/3}$ and
$n/N_j\asymp2^j$ proves \eqref{st:eq:boxcount}.
{This is an overcount: it includes box pairs containing no
endpoint pair that satisfies \eqref{st:eq:boxcompatibility}.
Only box pairs that do contain such a pair, namely those in
$\mathfrak B_j$, enter the stability event.}
\par\endgroup
\end{proof}
One application of Lemma~\ref{st:lem:compact} will control
all point pairs in
$Q_b^-\times Q_b^+$, so the union bound pays for $M_j$ box
pairs, with no further count of the points inside them.
\par\endgroup

\begin{figure}[tbp]

\centering
\begin{tikzpicture}[x=1cm,y=.95cm]
 \node[anchor=west] at (-.2,4.65) {\textbf{(a)} A retained pair of boxes};
 \filldraw[fill=pathblue!8,draw=pathblue!55]
   (.4,0)--(1.8,0)--(2.3,.8)--(.9,.8)--cycle;
 \filldraw[fill=pathblue!8,draw=pathblue!55]
   (2.9,3)--(4.3,3)--(4.8,3.8)--(3.4,3.8)--cycle;
 
 \node[anchor=east] at (.25,.15) {$Q_b^-$};
 \node[anchor=east] at (2.8,3.6) {$Q_b^+$};
 \draw[guide] (1.35,.4)--(3.85,3.4);
 \fill (1.35,.4) circle (1.7pt);
 \fill (3.85,3.4) circle (1.7pt);
 \node[anchor=east,text=black] (originlabel) at (.3,1.35) {$o_b$};
 \draw[->,black!65,line width=.4pt] (originlabel.east)--(1.3,.45);

 \node[rotate=50,fill=white,inner sep=1pt,font=\footnotesize]
   at (2.65,2) {chord slope $\lambda_b$};
 \draw[dynamicpath] (1.85,.65)--(2.7,.65)--(2.7,1.5)
   --(3.6,1.5)--(3.6,2.6)--(4.1,2.6)--(4.1,3.6);
 \fill[pathorange] (1.85,.65) circle (2pt) node[above] {$u$};
 \fill[pathorange] (4.1,3.6) circle (2pt) node[above] {$v$};
 \draw[<->] (.4,-.25)--node[below,font=\footnotesize] {$N_j^{2/3}$}(1.8,-.25);
 \draw[<->] (4.95,3)--node[right,font=\footnotesize] {$N_j/H_j$}(4.95,3.8);
 \node[align=center,text width=4.8cm,font=\footnotesize] at (2.4,-1.35)
   {$b=(Q_b^-,Q_b^+)$ deterministic;\\$u,v$ selected by the path.};
 \draw[->,line width=.8pt] (5.4,1.9)--(6.4,1.9);
 \node[align=center,font=\footnotesize] at (5.9,2.65)
   {translate\\and scale};
 \begin{scope}[xshift=7cm]
  \node[anchor=west] at (-.2,4.65) {\textbf{(b)} One compact endpoint family};
  \filldraw[fill=pathblue!8,draw=pathblue!55] (.6,0) rectangle (3,.8);
  \filldraw[fill=pathblue!8,draw=pathblue!55] (.6,3) rectangle (3,3.8);
  \draw[guide] (1.8,0)--(1.8,3.8);
  \fill[pathorange] (1,.35) circle (2pt);
  \fill[pathorange] (2.5,3.6) circle (2pt);
  \draw[<->] (.6,-.25)--node[below,font=\footnotesize] {$O(1)$}(3,-.25);
  \draw[<->] (3.35,.8)--node[right,align=left,font=\footnotesize]
    {time gap\\bounded below}(3.35,3);
  \node[align=center,text width=4.8cm,font=\footnotesize] at (2,-1.35)
    {The comparison holds for\\every pair in the boxes.};
 \end{scope}
\end{tikzpicture}
\caption{The deterministic endpoint cover used for local stability.
On the left, row bins of width comparable to \(N_j/H_j\) and
bins of width \(N_j^{2/3}\) in the coordinate \(x-i\) give
slanted boxes in the original plane. Their representative
chord determines \(\lambda_b\). On the right, endpoints are
displayed in coordinates relative to this chord, after
{normalisation} at scale \(N_j\). The calculation below places
them in a fixed compact family, uniformly over all retained
box pairs. {Black dots mark the deterministic
box centers (with integer rounding in the row coordinate),
the first of which is the origin $o_b$.} No conditioning
on the depicted geodesic is used.}
\label{fig:endpointboxes}
\end{figure}
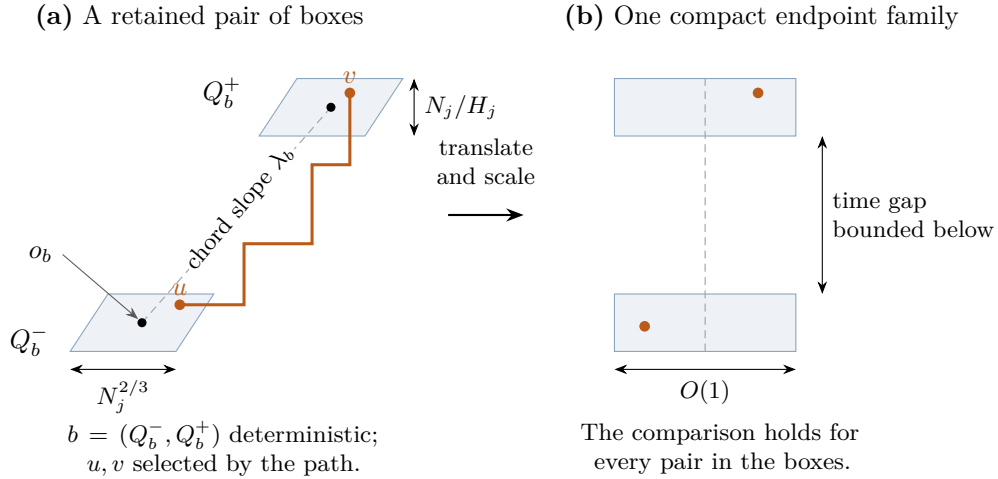

\paragraph{\textbf{Compact coordinates for each pair of boxes.}}
\begingroup
{{The endpoint boxes $Q_b^-$ and $Q_b^+$} are defined relative to the diagonal $x=i$, but
the chord joining their centers can have a different slope.
For a compatible pair, the difference from slope one is at most
of order $H_jN_j^{-1/3}$. We {normalise} relative to this chord
so that both centers have horizontal coordinate zero. Moving
$\ell_j$ rows within either box can then create a horizontal
displacement of order $H_jN_j^{-1/3}\ell_j$ from the chord,
or $H_j\ell_j/N_j$ after division by the transverse scale
$N_j^{2/3}$. A row height of order $N_j$ would allow this
quantity to grow with $H_j$. Our choice
$\ell_j\asymp N_j/H_j$ keeps it bounded, as verified in
Lemma~\ref{st:lem:commoncompact}. {Thus, in the new coordinates, the transverse widths of both
boxes remain bounded independently of $j$. Note that the geodesic
may visit several boxes within a row bin.}}

For $b\in\mathfrak B_j$, take the centers $(i_b,X_b)$ and
$(k_b,Y_b)$ of its two boxes in row and transverse coordinates,
rounding each row coordinate down to an integer. This rounding
keeps the origin on a Brownian row. In the original coordinates,
the representatives are $(x_b,i_b)=(i_b+X_b,i_b)$ and
$(y_b,k_b)=(k_b+Y_b,k_b)$. Define
\[
 o_b=(x_b,i_b),\qquad
 \lambda_b=\frac{y_b-x_b}{k_b-i_b}
           =1+\frac{Y_b-X_b}{k_b-i_b}.
\]
The origin $o_b$ and representative chord are shown in
Figure~\ref{fig:endpointboxes}.

\begin{lemma}[A common compact set for all box pairs]
\label{st:lem:commoncompact}
There is $L=L(c_0)$ such that, for every fixed $R$ and all
sufficiently large $n$, the following holds for every
$j\in\cI_n$ and $b\in\mathfrak B_j$.
We have $\lambda_b\in[1/2,2]$. Moreover, if
$(x,i)\in Q_b^-$ and $(y,k)\in Q_b^+$, and
\[
 (x',s')=\Phi_{N_j;o_b,\lambda_b}^{-1}(x,i),\qquad
 (y',u')=\Phi_{N_j;o_b,\lambda_b}^{-1}(y,k),
\]
then
\[
 |x'|,|y'|\le L,\qquad
 -c_0/8\le s'\le c_0/8,\qquad c_0/2\le u'\le3.
\]
\end{lemma}
\begin{proof}
By the definition of $\mathfrak B_j$, there is a compatible
pair of points in the two boxes with row separation in
$[c_0N_j,2N_j]$ and transverse displacement at most
$H_jN_j^{2/3}$. Moving either point to its representative
changes its row coordinate by at most $\ell_j$ and its
transverse coordinate by at most $N_j^{2/3}/2$. Thus
\[
 c_0N_j/2\le k_b-i_b\le3N_j,\qquad
 |Y_b-X_b|\le(H_j+1)N_j^{2/3}.
\]
It follows that $|\lambda_b-1|\le C_{c_0}H_jN_j^{-1/3}$.
Since $H_j\le C_{c_0}R(1+\log n)^{1/3}$ and $N_j\ge n^\beta$,
this tends to zero uniformly in $j$ for fixed $R$.
Hence $\lambda_b\in[1/2,2]$ for all sufficiently large $n$.

For $(x,i)\in Q_b^-$, the {normalised} coordinates are
\[
 (x',s')=
 \left(\frac{x-x_b-\lambda_b(i-i_b)}{2\lambda_bN_j^{2/3}},
       \frac{i-i_b}{N_j}\right).
\]
Writing $X=x-i$, the numerator of the first coordinate is
\[
 x-x_b-\lambda_b(i-i_b)
 =X-X_b+(1-\lambda_b)(i-i_b).
\]
Here $|X-X_b|\le N_j^{2/3}/2$, and
\[
 \underbrace{|1-\lambda_b|}_{O_{c_0}(H_jN_j^{-1/3})}
 \underbrace{|i-i_b|}_{\le\ell_j}
 \le C_{c_0}H_jN_j^{-1/3}\ell_j
 \le C_{c_0}N_j^{2/3}.
\]
{The last inequality is where the smaller row width is used:
the factor $H_j$ from the slope difference is cancelled by
$\ell_j\le c_0N_j/(16H_j)$. Thus this contribution stays bounded
after division by $N_j^{2/3}$.}
The same bound holds for a target endpoint relative to
$(y_b,k_b)$. Since the representative chord passes through
both representatives, its {normalised} horizontal coordinate
$y'$ has the same form. Dividing by $2\lambda_bN_j^{2/3}$
therefore gives $|x'|,|y'|\le L$ for some $L=L(c_0)$.

Finally, $|i-i_b|\le\ell_j\le c_0N_j/16$.
Comparing any target row $k$ with the compatible pair used
above gives
\[
 c_0N_j-2\ell_j\le k-i_b\le2N_j+2\ell_j.
\]
Dividing by $N_j$ proves the claimed row-coordinate bounds.
\end{proof}

For the remainder of the proof, fix $L$ supplied by
Lemma~\ref{st:lem:commoncompact} and set
\begin{equation}\label{st:eq:commoncompact}
 K=[-L,L]\times[-c_0/8,c_0/8]\times[-L,L]\times[c_0/2,3].
\end{equation}
Thus every pair in $Q_b^-\times Q_b^+$ has {normalised}
coordinates in this same $K$, using its own origin $o_b$
and slope $\lambda_b$. The set $K$ is independent of
$j,b,R,n$ and has time gap at least $3c_0/8$.
Lemma~\ref{st:lem:compact} therefore applies with the same
constants to every retained pair.
\par\endgroup

\begin{proof}[Proof of Proposition~\ref{st:prop:local}]
{Choose a fixed constant $R_1\ge R_0$, with $R_0$
the threshold in Lemma~\ref{st:lem:regularityinput}, and set
$R=R_1a^{-1/16}$.}
\begingroup
{With $K$ from \eqref{st:eq:commoncompact} and the finite collections
$\mathfrak B_j$ defined above, each $b\in\mathfrak B_j$
has its deterministic origin $o_b$ and slope $\lambda_b$.
Define the event that all these box-pair comparisons hold:}
\[
 \mathcal S_{n,a}=
 \bigcap_{j\in\cI_n}\ \bigcap_{b\in\mathfrak B_j}
 \left\{
  \sup_{z\in K}|D_{N_j;o_b,\lambda_b}^{\tau_j}(z)|
                  \le C_K\tau_j^{1/4}
 \right\}.
\]
{Recall from \eqref{st:eq:tau} that
$\tau_j=tN_j^{1/3}=a(N_j/n)^{1/3}$. Hence
$\tau_jN_j^{-1/3}=t$, so every comparison uses the same
physical dynamical times zero and $t$. The supremum over $K$
includes every endpoint pair in the boxes indexed by $b$.
The union over $b$ already counts all the origins and slopes
used here, since $(o_b,\lambda_b)$ is determined by $b$.}
\par\endgroup

Lemma~\ref{st:lem:compact} and \eqref{st:eq:boxcount} give
\begin{align}
 \PP(\mathcal S_{n,a}^c)
 &\le C\sum_{j\in\cI_n}M_j
                  \exp\{-c\tau_j^{-3/16}\}\notag\\
 &\le CR^4\sum_{j\ge0}{(1+j)^{2/3}2^{10j/3}}
             \exp\{-c a^{-3/16}2^{j/16}\}\notag\\
 &\le CR^4\exp\{-c'a^{-3/16}\}
 \le C\exp\{-c''a^{-3/16}\}.
 \label{st:eq:scaleunion}
\end{align}
For the third inequality, use
\(AB\ge(A+B)/2\) with \(A=a^{-3/16}\ge1\) and
\(B=2^{j/16}\ge1\); the remaining series
{\(\sum_j(1+j)^{2/3}2^{10j/3}\exp\{-c2^{j/16}/2\}\)}
is finite. In the final inequality the factor
\(R^4=R_1^4a^{-1/4}\) is absorbed by decreasing the
exponential constant.

Define \(\cE_{n,a}=\cG_{n,R}\cap\mathcal S_{n,a}\).
The regularity bound \eqref{st:eq:georegprob} gives
\(\PP(\cG_{n,R}^c)\le C\exp\{-cR_1^3a^{-3/16}\}\),
so \eqref{st:eq:scaleunion} proves \eqref{st:eq:goodprob}.
For fixed \(a\), \(R\) is independent of \(n\); hence
all the regularity and compact-coordinate conditions hold for
sufficiently large \(n\).

\begingroup
{Write $u_2,v_2$ for the row coordinates of $u,v$, respectively.
On $\cE_{n,a}\subset\cG_{n,R}$, Lemma~\ref{st:lem:boxcover}
gives the explicit containment}
\[
 \bigl\{(u,v)\in\Gamma^t\times\Gamma^t:
            {u_2,v_2\in\ZZ,\quad}
            c_0N_j\le v_2-u_2\le2N_j\bigr\}
 \subseteq\bigcup_{b\in\mathfrak B_j}Q_b^-\times Q_b^+.
\]
For any such pair,
choose its box pair $b$. Lemma~\ref{st:lem:commoncompact} gives
\[
 z=\bigl(\Phi_{N_j;o_b,\lambda_b}^{-1}(u),
         \Phi_{N_j;o_b,\lambda_b}^{-1}(v)\bigr)\in K.
\]
The event $\mathcal S_{n,a}$ controls the comparison for this
$b$ and every $z\in K$. Undoing its {normalisation} yields
\[
 |T_u^{v,t}-T_u^{v,0}|
 \le C_K\sqrt{\lambda_b}\,\tau_j^{1/4}N_j^{1/3}
 \le\sqrt2 C_K\tau_j^{1/4}N_j^{1/3},
\]
which is \eqref{st:eq:localstability}.
\par\endgroup

\end{proof}

{The exponential in $2^{j/16}$ beats even the
overcount $(1+j)^{2/3}2^{10j/3}$ of endpoint-box pairs.} This is the
precise role of increasing subcriticality at smaller scales
in the union bound. The resulting event controls every
relevant endpoint pair, including the three pairs used for
each crossing in the next section. Consequently, the later
count treats each actual excursion directly, with the common
exceptional event charged only once in \eqref{eq:widemean}.

\section{\texorpdfstring{{From local stability to static near-{maximisers}}}{From local stability to static near-maximisers}}\label{sec:crossing}

\begingroup
{Having obtained local stability, we consider an excursion
of $\Gamma^t$ away from $\Gamma^0$ whose row duration lies in
$[N_j/2,N_j]$. At each integer row $m$ at least $2c_0N_j$ rows
from both endpoints of the excursion, we will show that the exit
coordinate $\Gamma^t(m)$ is a near-maximiser of the static routed
profile, with deficit at most $3A_j$. Here
$A_j=C\tau_j^{1/4}N_j^{1/3}$ is the local stability tolerance
in \eqref{st:eq:localstability}. If
$|\Gamma^t(m)-\Gamma^0(m)|\ge w_j$, this gives the static
twin-peak event $\cP_m(3A_j,w_j)$ defined in
\eqref{wd:eq:peak}. We first prove this implication and then
bound the twin-peak probability, retaining its dependence on
both the tolerance and the separation.}
\par\endgroup

\subsection{\texorpdfstring{{{From dynamical crossings to static near-{maximisers}}}}{From dynamical crossings to static near-maximisers}}

\begingroup
{The endpoints of an excursion lie on both geodesics.
In the time-zero environment, take the best path between these
endpoints that exits row $m$ at $x=\Gamma^t(m)$. This path
consists of two static geodesics, one ending at $(x,m)$ and
the other beginning at $(x,m+1)$. Comparing the passage values
for these two legs and for the full segment bounds its deficit
by $3A_j$. We then attach the portions of $\Gamma^0$ before
and after the excursion to obtain a path from $\mathbf{0}$ to
$\mathbf{n}$ with the same deficit. Figure~\ref{fig:threevalues}
shows these paths.}

{The same three-value comparison appears in
\cite[Section 3.3.2, equation (25)]{GH24}. Here
Proposition~\ref{st:prop:local} supplies one stability event for
every admissible choice of excursion endpoints and row $m$,
at the critical time $t=a n^{-1/3}$. In their treatment of many
excursions, \cite[Section 3.3.3]{GH24} instead compares the total
weight of a proxy path.}
\par\endgroup

Let \(u=(u_x,i)\), \(v=(v_x,k)\) lie on both
\(\Gamma^0\) and \(\Gamma^t\), with {\(i,k\in\ZZ\)} and
\(N_j/2\le k-i\le N_j\). They may be the endpoints of
an excursion. Recall that $c_0$ is the small fixed bulk fraction
in Proposition~\ref{st:prop:local}. We choose {an integer row $m$} in the interior
of the excursion, a distance at least $2c_0N_j$ from either end:
\begingroup
\begin{equation}\label{st:eq:interior}
 m-i\ge2c_0N_j,\qquad k-m\ge2c_0N_j.
\end{equation}
\par\endgroup
Let \(x=\Gamma^t(m)\) denote the exit coordinate and put
\[
 z=(x,m),\qquad z^+=(x,m+1),\qquad
 Z_m^0(x)={T_{\mathbf{0}}^{z,0}}+{T_{z^+}^{\mathbf{n},0}}.
\]
The two routed passage values use disjoint row sets; the
vertical jump from \(z\) to \(z^+\) has weight zero.

\begin{figure}[tbp]
\centering
\begin{tikzpicture}[x=.9cm,y=.75cm]
 \draw[guide] (.4,2) node[left] {row $m$}--(8,2);
 \draw[guide] (.4,3) node[left] {row $m+1$}--(8,3);
 \draw[black,line width=1.3pt] (.6,-.5)--(2,-.5)--(2,0)--(2.6,0);
 \draw[staticpath] (2.6,0)--(2.6,1)--(3.2,1)
   --(3.2,2)--(3.8,2)--(3.8,3)--(4.4,3)
   --(4.4,4)--(5.1,4)--(5.1,5)--(7,5);
 \draw[dynamicpath] (2.6,0)--(4,0)--(4,1)--(4.8,1)
   --(4.8,2)--(5.6,2)--(5.6,3)--(6,3)
   --(6,4)--(7,4)--(7,5);
 \draw[black,line width=1.3pt] (7,5)--(8,5)--(8,6);
 \draw[proxypath] (2.6,0)--(3.6,0)--(3.6,1)--(4.3,1)
   --(4.3,2)--(5.6,2);
 \draw[proxypath] (5.6,3)--(6.6,3)--(6.6,5)--(7,5);
 \foreach \x/\y in {2.6/0,7/5,.6/-.5,8/6}
    \fill (\x,\y) circle (2pt);
 \node[below left] at (2.6,0) {$u$};
 \node[above] at (7,5) {$v$};
 \node[below left] at (.6,-.5) {${\mathbf{0}}$};
 \node[above] at (8,6) {${\mathbf{n}}$};
 \fill[pathorange] (5.6,2) circle (2.4pt) node[right,fill=white,inner sep=1pt] {$z=(x,m)$};
 \fill[pathorange] (5.6,3) circle (2.4pt) node[above left,fill=white,inner sep=1pt] {$z^+$};
 \node[anchor=north west,align=left,text width=4cm] at (9,5.65)
   {\textbf{Three local values}\\[3pt]
    $T_u^v$, $T_u^z$, $T_{z^+}^v$\\[6pt]
    Each changes by at most $A_j$.};
 \node[anchor=north west,align=left,text width=4cm] at (9,2.45)
   {\textbf{A global witness}\\[3pt]
    Keep the static prefix and suffix.\\[3pt]
    The deficit at $x$ is at most $3A_j$.};
 \draw[staticpath] (1,-1.5)--(1.65,-1.5);
 \node[anchor=west,font=\footnotesize] at (1.75,-1.5) {time-zero geodesic};
 \draw[dynamicpath] (5,-1.5)--(5.65,-1.5);
 \node[anchor=west,font=\footnotesize] at (5.75,-1.5) {time-$t$ geodesic};
 \draw[proxypath] (9,-1.5)--(9.65,-1.5);
 \node[anchor=west,font=\footnotesize] at (9.75,-1.5) {static {maximising} legs};
\end{tikzpicture}
\caption{{The static deficit estimate.} The time-$t$ segment from
\(u\) to \(v\) passes through the vertical jump from
\(z\) to \(z^+\). The two dashed legs are separately
{optimised} in the static environment. With the common
static prefix and suffix, they give a candidate routed
path from \({\mathbf{0}}\) to \({\mathbf{n}}\). The argument compares three
{optimised} values; it does not estimate the static weight
of the orange excursion. The drawn legs are schematic
and may have additional intersections with the static geodesic.}
\label{fig:threevalues}
\end{figure}
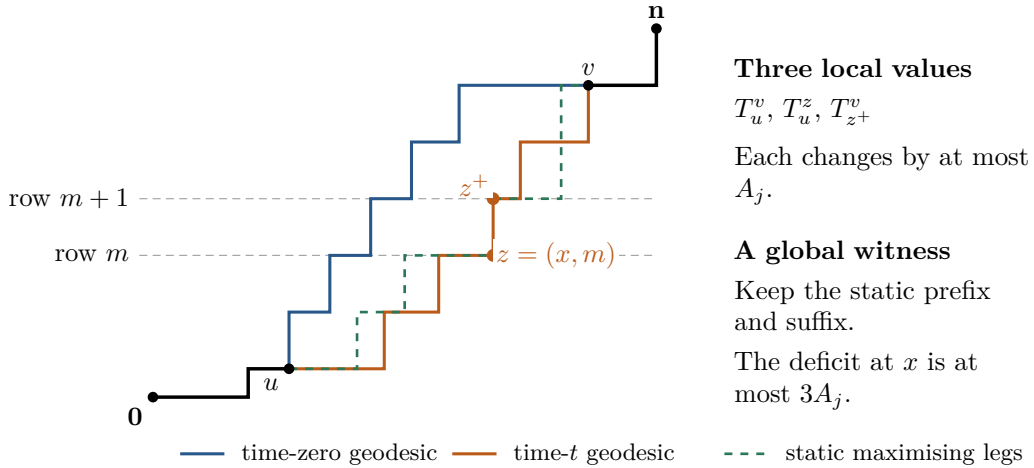

\begin{lemma}[{Static deficit estimate}]\label{st:lem:crossing}
{For all sufficiently large $n$, on the event
$\cE_{n,a}$ of Proposition~\ref{st:prop:local},}
simultaneously for all the above choices,
\begin{equation}\label{st:eq:crossing}
 0\le {T_{\mathbf{0}}^{\mathbf{n},0}}-Z_m^0\bigl(\Gamma^t(m)\bigr)\le3A_j.
\end{equation}
\end{lemma}

\begin{proof}
The static geodesic passes through \(u,v\). Its prefix and
suffix, concatenated with {maximising} static paths from \(u\)
to \(z\) and from \(z^+\) to \(v\), show that
\begin{equation}\label{st:eq:concatenate}
 0\le {T_{\mathbf{0}}^{\mathbf{n},0}}-Z_m^0(x)
 \le T_u^{v,0}-T_u^{z,0}-T_{z^+}^{v,0}.
\end{equation}
At dynamical time \(t\), all four points occur in order on
\(\Gamma^t\), so
\[
 T_u^{v,t}=T_u^{z,t}+T_{z^+}^{v,t}.
\]
Therefore the final expression in \eqref{st:eq:concatenate} equals
\begin{equation}\label{st:eq:threevalues}
 \begin{aligned}
 &(T_u^{v,0}-T_u^{v,t})
 +(T_u^{z,t}-T_u^{z,0})\\
 &\hspace{35mm}+(T_{z^+}^{v,t}-T_{z^+}^{v,0}).
 \end{aligned}
\end{equation}
\begingroup
For all sufficiently large $n$, the upper leg has row duration
$k-m-1\ge2c_0N_j-1\ge c_0N_j$, since $N_j\ge n^\beta$.
Together with \eqref{st:eq:interior}, this shows that all three
pairs are covered by \eqref{st:eq:localstability}. Each absolute
difference is at most $A_j$, which proves \eqref{st:eq:crossing}.
\par\endgroup
\end{proof}

The conclusion uses the \emph{global} static profile with
the original deterministic endpoints \({\mathbf{0},\mathbf{n}}\), but its
tolerance is determined by the excursion's local length.
{The next subsection applies Brownianity to this
fixed-endpoint profile at the separation $w_j$;} it does not require a
Brownian comparison for profiles with the random endpoints
\(u,v\).

\subsection{\texorpdfstring{{A static twin-peak estimate}}{A static twin-peak estimate}}\label{sec:nearpeaks}
\begingroup
Lemma~\ref{st:lem:crossing} associates a static near-{maximiser}
to each suitably separated crossing of a dynamical excursion.
We now bound the probability of such a crossing for a deterministic
row. Keeping both the deficit tolerance and the spatial separation
in the estimate will allow its expected contribution to be summed
over all excursion scales.
\par\endgroup

Work in static Brownian last passage percolation, with
{endpoints \(\mathbf{0},\mathbf{n}\)}. Fix a bulk fraction \(\xi\in(0,1/2)\)
and an integer row \(\xi n\le m\le(1-\xi)n\). Set
\[
 Z_m(x)={T_{\mathbf{0}}^{(x,m)}}+{T_{(x,m+1)}^{\mathbf{n}}},\qquad
 M_m={\Gamma_{\mathbf{0}}^{\mathbf{n}}}(m),\qquad 0\le x\le n.
\]
Here \({\Gamma_{\mathbf{0}}^{\mathbf{n}}}(m)\) is the exit coordinate from row \(m\).
The two passage values defining \(Z_m\) use disjoint row sets.
Almost surely \(M_m\) is its unique {maximiser} and
\(\max Z_m={T_{\mathbf{0}}^{\mathbf{n}}}\).

For a deficit \(A>0\) and separation \(w>0\), define
\begin{equation}\label{wd:eq:peak}
 \cP_m(A,w)=
 \{\exists x\in[0,n]:
       |x-M_m|\ge w,\quad {T_{\mathbf{0}}^{\mathbf{n}}}-Z_m(x)\le A\}.
\end{equation}
The event retains two pieces of information: the competitor
loses at most \(A\) in weight, and it is at least \(w\)
away from the {maximising} departure. Keeping both parameters
is essential. At smaller excursion scales both will shrink,
but \(A\) shrinks faster than the Brownian fluctuation scale
\(\sqrt w\). This is what makes the near-peak probability
summable over excursion scales.

\begingroup
The relevant ratio is the deficit tolerance divided by the
Brownian fluctuation scale at separation $w$:
\begin{equation}\label{wd:eq:parameters}
 \sigma=\frac{A}{\sqrt w}.
\end{equation}
\par\endgroup

\begin{proposition}[{\cite[Proposition 82]{Bha25}}]\label{wd:prop:static}
Fix \(\xi\in(0,1/2)\) and \(B>0\). There are constants
\(C,c,\sigma_0>0\) and \(n_0\), depending only on \(\xi,B\),
such that, for \(n\ge n_0\), every integer row $\xi n\le m\le(1-\xi)n$, and every
\(A,w>0\) satisfying
\begin{equation}\label{wd:eq:range}
 {n^{-B}\le\frac{w}{n^{2/3}}\le1},
 \qquad n^{-B}\le\sigma\le\sigma_0,
\end{equation}
we have
\begin{equation}\label{wd:eq:main}
 \PP(\cP_m(A,w))
 \le C\sigma^{1/2}+Ce^{-cn^{1/12}}.
\end{equation}
\end{proposition}
\begingroup
\begin{proof}[Identification with the cited proposition]
Take \(\beta'=\xi\), \(D=B\), and \(\varepsilon=1/2\) in
\cite[Proposition 82]{Bha25}, and choose \(\sigma_0=1\).
The row-disjoint routed profile there is
\[
 Z_{\mathbf0}^{\mathbf n,\bullet}(x,m)
 ={T_{\mathbf{0}}^{(x,m)}}+{T_{(x,m+1)}^{\mathbf{n}}}=Z_m(x).
\]
Since \(Z_m(x)\le {T_{\mathbf{0}}^{\mathbf{n}}}\), its absolute weight deficit is
\({T_{\mathbf{0}}^{\mathbf{n}}}-Z_m(x)\). Its event \(\mathrm{TP}_{A,w,m}\) is
therefore exactly \(\cP_m(A,w)\), and its two lower bounds
are \(n^{-B}\le w/n^{2/3}\) and \(n^{-B}\le A/\sqrt w\).
The cited bound \(C\sigma^{1-\varepsilon}+Ce^{-cn^{1/12}}\)
then gives \eqref{wd:eq:main}, uniformly over the stated rows
and parameters.
\end{proof}
\endgroup
\begingroup
Proposition~\ref{wd:prop:static} uses the locally Brownian
{behaviour} of passage-time profiles and hence of routed profiles
such as $Z_m$, which is a sum of profiles from disjoint row sets.
The Brownian Gibbs description of the Airy line ensemble
\cite{CH14} is the basis for quantitative Brownian comparisons
in \cite{Ham22,CHH23} and the sharper Wiener-density estimates
in \cite{Dau24}. For BLPP, Ganguly and Hammond
\cite[Theorem 1.3]{GH23} use Brownian comparison to bound
separated near maxima. The estimate imported here from
\cite[Proposition 82]{Bha25} adapts the Wiener-density argument
to BLPP and retains the tolerance and separation without a
loss growing with the ambient window. Such estimates are called
\emph{twin-peak estimates}: they control the occurrence of a
second point whose height is close to the maximum but whose
location is separated from the {maximiser}.
\par\endgroup

\subsection{\texorpdfstring{{Substitution into the excursion argument}}{Substitution into the excursion argument}}

\begingroup
\begingroup
With the parameters in \eqref{eq:parameters} and
\eqref{st:eq:scales}, {the static deficit estimate
in Lemma~\ref{st:lem:crossing}} gives the tolerance
$3A_j$. Writing it in terms of the effective time, we have
\begin{equation}\label{wd:eq:dynparameters}
 3A_j=C_0\tau_j^{1/4}N_j^{1/3},\qquad
 w_j=\tau_j^\nu N_j^{2/3},
\end{equation}
where $C_0>0$ is fixed. The ratio entering the twin-peak bound is
therefore
\begin{equation}\label{wd:eq:substitution}
 \sigma_j=\frac{3A_j}{\sqrt{w_j}}
          =C_0\tau_j^{1/4-\nu/2}.
\end{equation}
\par\endgroup

The powers of $N_j$ cancel. Since $\nu<1/2$, this ratio
decreases with the effective time $\tau_j$ and therefore
with the excursion scale.

\begingroup
Combining this ratio with the static twin-peak estimate gives
the following bound at a fixed row and, on summing the resulting
geometric series, over every retained scale.
\par\endgroup

\begin{corollary}\label{wd:cor:sum}
Fix $\xi\in(0,1/2)$, $0<\nu<1/2$ and $C_0>0$.
For every sufficiently small fixed $a>0$ and all sufficiently
large $n$, uniformly over $j\in\cI_n$ and integers
$\xi n\le m\le(1-\xi)n$,
\begin{equation}\label{wd:eq:rowbound}
 \PP(\cP_m({3A_j},w_j))
 \le C_{\xi,\nu,C_0}\tau_j^{\,1/8-\nu/4}
                  +Ce^{-cn^{1/12}}.
\end{equation}
Consequently,
\begin{equation}\label{wd:eq:sum}
 \frac1n\sum_{\xi n\le m\le(1-\xi)n}
       \sum_{j\in\mathcal I_n}\PP(\cP_m({3A_j},w_j))
 \le C_{\xi,\nu,C_0}a^{\,1/8-\nu/4}
                  +Ce^{-c'n^{1/12}}.
\end{equation}
\end{corollary}
\begin{proof}
\begingroup
To apply Proposition~\ref{wd:prop:static}, we must check
its range \eqref{wd:eq:range}. We take $B=2$. For a fixed
sufficiently small $a>0$ and all sufficiently large $n$,
uniformly over $j\in\cI_n$,
\[
 \begin{aligned}
 \frac{w_j}{n^{2/3}}
 &=a^\nu(N_j/n)^{(2+\nu)/3}
   \ge a^\nu n^{-(2+\nu)/3}\ge n^{-2},\\
 \sigma_j
 &=C_0a^{1/4-\nu/2}(N_j/n)^{(1/4-\nu/2)/3}
   \ge C_0a^{1/4-\nu/2}n^{-(1/4-\nu/2)/3}\ge n^{-2}.
 \end{aligned}
\]
Here we only used $N_j\ge1$; both powers of $n$ are smaller
than $2$ because $0<\nu<1/2$. The upper bounds
$w_j/n^{2/3}\le1$ and $\sigma_j\le\sigma_0$ follow from
$N_j\le n$ and a sufficiently small choice of $a$.
These are precisely the four inequalities in
\eqref{wd:eq:range}.

Proposition~\ref{wd:prop:static} now applies with $B=2$,
and \eqref{wd:eq:substitution} gives \eqref{wd:eq:rowbound}.
\par\endgroup
Since $1/8-\nu/4>0$ and $\tau_j\le a2^{-j/3}$,
\[
 \sum_{j\in\cI_n}\tau_j^{1/8-\nu/4}
 \le a^{1/8-\nu/4}\sum_{j\ge0}2^{-j(1/8-\nu/4)/3}
 \le C_\nu a^{1/8-\nu/4}.
\]
\begingroup
For each pair $(m,j)$, the bound \eqref{wd:eq:rowbound} also
contains the error term $Ce^{-cn^{1/12}}$. There are at most
$n+1$ choices of $m$ and $C\log n$ choices of $j$. After the
division by $n$ in \eqref{wd:eq:sum}, the sum of these errors
is therefore at most
\[
 \frac{n+1}{n}\,C\log n\,e^{-cn^{1/12}}
 \le C'e^{-c'n^{1/12}}.
\]
Combining this with the geometric sum above proves
\eqref{wd:eq:sum}.
\par\endgroup
\end{proof}

\begingroup
Thus summability of the twin-peak probabilities requires
$0<\nu<1/2$. The slender-excursion estimate used next holds
for every fixed $\nu>0$, so it imposes no additional restriction.
\par\endgroup
\par\endgroup

\begingroup
\section{Proof of the overlap theorem}\label{sec:assembly}
\begingroup
The previous section gives a static twin-peak event at every
suitably interior row where a dynamical excursion is separated
from the static geodesic. We can therefore control long
excursions that have such a separation on a positive fraction
of their rows. Following the terminology of \cite{GH23,GH24},
the complementary excursions are called \emph{slender}: they
remain close to the static geodesic for almost all of their
duration. Their rarity requires a separate argument, tightening
the corresponding slender-path estimate in \cite{GH23}.
We state the resulting bound here, use it to complete
Theorem~\ref{thm:main}, and give its proof in
Section~\ref{sec:slender}.
\par\endgroup

\subsection{\texorpdfstring{{The slender-excursion input}}{The slender-excursion input}}

\begingroup
For $\chi\in(0,1)$, an excursion $\pi$ of $\Gamma^t$ from
$\Gamma^0$ is \emph{slender at scale $j$} if its row duration
$d$ belongs to $[N_j/2,N_j]$ and
\[
 \#\{r\text{ a row of }\pi:
       |\pi(r)-\Gamma^0(r)|>w_j\}
 \le{\chi d}.
\]

\begingroup
Here $\pi(r)$ is the departure coordinate of the excursion's
time-$t$ segment, and the count includes both endpoint rows.
\par\endgroup
\begingroup
{In the following proposition,} $a$ determines the tube widths through
\[
 \theta_j=\bigl[a(N_j/n)^{1/3}\bigr]^\nu,
 \qquad w_j=\theta_jN_j^{2/3}.
\]
These are exactly the widths in
\eqref{st:eq:scales} and \eqref{wd:eq:dynparameters} at dynamical
time $a n^{-1/3}$. We now keep these widths fixed while allowing the
comparison time $t$ to be any deterministic nonnegative number.
Thus $a$ is a width parameter in the proposition below; it need
not equal $tn^{1/3}$. {This freedom will allow us to
enlarge the tubes later without changing the two geodesics being
compared.} In the proof of the overlap theorem, we shall return to the choice $t=a n^{-1/3}$.
\par\endgroup

\begingroup
\begin{proposition}[Rarity of slender excursions]\label{sl:prop:dynamic}
There is an absolute constant $\chi\in(0,1)$ such that the
following holds. Fix $0<\beta<1$ and $\nu>0$.
There are $a_*>0,c,C>0$, depending on $\beta,\nu$, such that,
for $0<a\le a_*$
and all sufficiently large $n$ depending on $a$, the
following holds. For any fixed OU dynamical time $t\ge0$,
\begin{equation}\label{sl:eq:noslender}
 \PP\!\left(
  \begin{array}{c}
   \text{there is a slender excursion of }\Gamma^t
          \text{ from }\Gamma^0\\
   \text{at some scale }N_j\ge n^\beta
  \end{array}\right)
 \le C e^{-c a^{-3\nu/8}}.
\end{equation}
The constants and the lower bound on $n$ are uniform in the
fixed time $t$. In particular, this applies at
$t=a n^{-1/3}$.
\end{proposition}
\par\endgroup

\begingroup
\begingroup
Importantly, note here that the tolerance $\theta_j$ decreases with $j$:
an excursion at a smaller scale must remain closer to $\Gamma^0$
relative to its natural transversal scale to be called slender.
In \cite[Proposition 9.2 and its proof]{GH24}, the corresponding
tolerance is the same at every scale and decreases with $n$;
that choice makes the error small enough to absorb the counts
over possible locations. Here the tolerance decreases with the
excursion scale itself. Together with the local estimate proved
below, this allows us to sum over locations and scales while
keeping $a$ fixed as $n\to\infty$.

For the overlap argument, narrowing the tubes means that more
excursions must be treated as wide. Corollary~\ref{wd:cor:sum}
allows precisely these scale-dependent separation thresholds
when $0<\nu<1/2$: passage-time stability improves at smaller
scales, and the ratio in \eqref{wd:eq:substitution} still tends
to zero. The two estimates therefore work with the same widths.
The bound \eqref{sl:eq:noslender} depends only on $a$ once
$n$ is sufficiently large; no decay of $a$ with $n$ is required.
\par\endgroup

As in \cite[Section 4]{GH23} and
\cite[Section 9.3]{GH24}, the proof uses static regularity
at the two times and an FKG
comparison after conditioning on $\Gamma^0$. The latter bounds
the conditional probability of any increasing event outside this
geodesic by its probability in a fresh environment. We can therefore
apply the deterministic-path estimate with $\Gamma^0$ as its
reference path. {To the best of the author's knowledge, this use of FKG after
conditioning on a geodesic first appeared in \cite{BSS19}; subsequent
applications include \cite{GH23,GH24,BB24,BB23}.} We record the precise
conditional form needed here in Lemma~\ref{sl:lem:exteriorinput}.
Since this argument does not require passage-time stability,
the conclusion is uniform in the fixed dynamical time $t$.
\par\endgroup
\par\endgroup

\begingroup
\subsection{\texorpdfstring{{Completion of the proof}}{Completion of the proof}}
We now complete the proof of Theorem~\ref{thm:main}, assuming
Proposition~\ref{sl:prop:dynamic}. The main step is to bound the
{expected total duration of long excursions by combining the two
branches of Figure~\ref{fig:argumentflow}: the static twin-peak estimate
for wide excursions and Proposition~\ref{sl:prop:dynamic} for the
remaining slender ones.}

\begin{proof}[Proof of Theorem~\ref{thm:main} assuming Proposition~\ref{sl:prop:dynamic}]
Fix $\beta,\nu$ satisfying \eqref{eq:parameters}. Throughout this
proof we take $t=a n^{-1/3}$, with $a>0$ fixed while $n$ tends
to infinity. By Lemma~\ref{lem:fullreduction}, it is enough to prove
\begin{equation}\label{eq:longdone}
 \lim_{a\downarrow0}\limsup_{n\to\infty}
       \frac{\EE D_{\rm long}(a n^{-1/3})}{n}=0.
\end{equation}
We will obtain this by counting witness rows at every scale.
The local stability event {$\cE_{n,a}$ from Proposition~\ref{st:prop:local}} holds simultaneously for all possible
subsegments, whereas near-peak rarity is used in expectation,
one row and scale at a time. This distinction allows us to sum
the expected contributions without selecting a dominant scale.

\begingroup
\paragraph{\textbf{The classification and the slender contribution.}}
Let $\chi$ be the constant in Proposition~\ref{sl:prop:dynamic},
and choose $0<c_0<\chi/32$ in Proposition~\ref{st:prop:local}.
Fix $\xi\in(0,1/2)$ and write
\[
 \cB_{n,\xi}=\{m\in\ZZ:\xi n\le m\le(1-\xi)n\}.
\]
All these constants are fixed before $n$ tends to infinity.

\begingroup
For a long excursion $E$, let $i_E,k_E$ be its endpoint
rows, let $d_E=k_E-i_E$, and let $\pi_E$ denote its time-$t$
segment. Assign $E$ to the largest $j\ge0$ for which $N_j\ge d_E$.
Since $d_E\in[\lceil n^\beta\rceil,n]\cap\mathbb Z$, this
choice gives $j\in\cI_n$ and $N_j/2<d_E\le N_j$.
At this scale, we use exactly the definition preceding
Proposition~\ref{sl:prop:dynamic}: $E$ is slender if
\par\endgroup
\[
 \#\bigl\{m\in[i_E,k_E]\cap\ZZ:
        |\pi_E(m)-\Gamma^0(m)|>w_j\bigr\}\le\chi d_E,
\]
and we call it wide otherwise. Denoting the respective sums of
row durations by $D_{\rm slender}$ and $D_{\rm wide}$, we have
$D_{\rm long}=D_{\rm slender}+D_{\rm wide}\le n$.

Proposition~\ref{sl:prop:dynamic} now applies directly to every
excursion classified as slender. Since $D_{\rm slender}=0$
unless such an excursion exists, and always $D_{\rm slender}\le n$,
it gives
\begin{equation}\label{eq:slendertotal}
 \frac{\EE D_{\rm slender}}n\le Ce^{-ca^{-3\nu/8}}.
\end{equation}

\paragraph{\textbf{Charging the wide durations to witness rows.}}
For each deterministic scale $j$ and bulk row $m$, let
\begin{equation}\label{eq:rowevent}
 P_{j,m}=\cP_m(3A_j,w_j),
\end{equation}
where $\cP_m$ is the static near-peak event in \eqref{wd:eq:peak}.
The key counting bound is that, for all sufficiently large $n$,
on the local stability event $\cE_{n,a}$,
\begin{equation}\label{eq:widepathwise}
 \frac{\chi}{2}D_{\rm wide}
 \le\sum_{j\in\cI_n}\sum_{m\in\cB_{n,\xi}}\ind_{P_{j,m}}
       +2\xi n+2.
\end{equation}
We prove this by assigning sufficiently many distinct witness
rows to each wide excursion, then summing over the excursions.

Let $\mathfrak X_j$ be the collection of wide excursions assigned
to scale $j$. For $E\in\mathfrak X_j$, define
\[
 S_E=\bigl\{m\in\ZZ:i_E<m<k_E,\quad
               |\Gamma^t(m)-\Gamma^0(m)|>w_j\bigr\}.
\]
The inequalities $i_E<m<k_E$ mean that we use the rows
$i_E+1,\ldots,k_E-1$, excluding the two endpoint rows.
On these rows, the excursion departure $\pi_E(m)$ equals
the full-path departure $\Gamma^t(m)$. Since $E$ is wide,
more than $\chi d_E$ rows satisfy the separation test in its
definition; excluding the two endpoint rows therefore gives
\[
 |S_E|\ge\chi d_E-2.
\]

\begingroup
The static deficit estimate in Lemma~\ref{st:lem:crossing}
requires a little more distance from the ends, and the static
near-peak bound in Corollary~\ref{wd:cor:sum} requires a bulk
row of the original geodesic. Accordingly, define
\par\endgroup
\[
 Q_E=\bigl\{m\in S_E\cap\cB_{n,\xi}:
       m-i_E\ge2c_0N_j,\quad {k_E-m\ge2c_0N_j}\bigr\},
\]
and put
\[
 b_E=\#\bigl(((i_E,k_E)\cap\ZZ)\setminus\cB_{n,\xi}\bigr).
\]
At each end, at most $2c_0N_j+2$ integer rows fail the local
distance condition. It follows that
\[
 \begin{aligned}
 |Q_E|
 &\ge |S_E|-(4c_0N_j+4)-b_E\\
 &\ge\chi d_E-4c_0N_j-6-b_E\\
 &\ge\frac{\chi}{2}d_E-b_E.
 \end{aligned}
\]
For the last inequality, $N_j\le2d_E$ and $c_0<\chi/32$
give $4c_0N_j\le\chi d_E/4$, while $d_E\ge n^\beta$
ensures $6\le\chi d_E/4$ for all sufficiently large $n$.

We next check that every row retained in $Q_E$ is counted on
the right of \eqref{eq:widepathwise}. If $m\in Q_E$, its local
distance conditions are precisely \eqref{st:eq:interior}, so
Lemma~\ref{st:lem:crossing} gives, on $\cE_{n,a}$,
\[
 {T_{\mathbf{0}}^{\mathbf{n},0}}-Z_m^0\bigl(\Gamma^t(m)\bigr)\le3A_j.
\]
Moreover, $m\in S_E$ gives
$|\Gamma^t(m)-\Gamma^0(m)|>w_j$. Thus
$x=\Gamma^t(m)$ witnesses $P_{j,m}$ by its definition.

The open row intervals $(i_E,k_E)$ of distinct excursions
are disjoint. In particular, the sets $Q_E$ are disjoint,
and for each fixed $j,m$ at most one $E\in\mathfrak X_j$
has $m\in Q_E$. Hence, on $\cE_{n,a}$,
\[
 \sum_{j\in\cI_n}\sum_{E\in\mathfrak X_j}|Q_E|
 =\sum_{j\in\cI_n}\sum_{m\in\cB_{n,\xi}}
                  \sum_{E\in\mathfrak X_j}\ind_{\{m\in Q_E\}}
 \le\sum_{j\in\cI_n}\sum_{m\in\cB_{n,\xi}}\ind_{P_{j,m}}.
\]
The same disjointness, now applied to the rows outside the
bulk interval, gives
\[
 \sum_{j\in\cI_n}\sum_{E\in\mathfrak X_j}b_E
 \le\#\bigl((\{1,\ldots,n-1\})\setminus\cB_{n,\xi}\bigr)
 \le2\xi n+2.
\]
Finally, summing $\chi d_E/2\le |Q_E|+b_E$ over all wide
excursions proves \eqref{eq:widepathwise}.
\par\endgroup

\paragraph{\textbf{Expected long-excursion duration.}}
\begingroup
Taking expectations in \eqref{eq:widepathwise} on \(\cE_{n,a}\),
and using \(D_{\rm wide}\le n\) on its complement, gives the
first inequality below. Proposition~\ref{st:prop:local} and
Corollary~\ref{wd:cor:sum} give the second.
{In the next two displays, $c,C$ may depend on
the fixed bulk parameter $\xi$; we keep the boundary contribution
$4\xi/\chi$ explicit.}
\begingroup
\begin{equation}\label{eq:widemean}
 \begin{aligned}
 \frac{\EE D_{\rm wide}}n
 &\le \PP(\cE_{n,a}^c)
   +\frac{2}{\chi n}\sum_{j\in\cI_n}\sum_{m\in\cB_{n,\xi}}
                         \PP(P_{j,m})
   +\frac{4\xi}{\chi}+\frac{4}{\chi n}\\
 &\le Ce^{-c a^{-3/16}}+C a^{\,1/8-\nu/4}
       +\frac{4\xi}{\chi}+\frac{4}{\chi n}+Ce^{-c n^{1/12}}.
 \end{aligned}
\end{equation}
\par\endgroup
\par\endgroup
Together with \eqref{eq:slendertotal}, this proves
\begingroup
\[
 \limsup_{n\to\infty}\frac{\EE D_{\rm long}}n
 \le Ce^{-c a^{-3/16}}+C a^{\,1/8-\nu/4}
       +Ce^{-ca^{-3\nu/8}}+\frac{4\xi}{\chi}.
\]
\par\endgroup
First let $a\downarrow0$, with $\xi$ fixed, and then let
$\xi\downarrow0$. Since $\nu<1/2$, all terms on the right
vanish in this order, establishing \eqref{eq:longdone}.
No independence between rows, excursions, or scales has
been used in this assembly.

\begingroup
\begingroup
By the reduction at the start of the proof, \eqref{eq:longdone}
establishes \eqref{eq:mainprob}, completing the proof.
\par\endgroup
\par\endgroup
\end{proof}
\par\endgroup
\par\endgroup

\begingroup
\section{\texorpdfstring{{Rarity of slender excursions}}{Rarity of slender excursions}}\label{sec:slender}

\begingroup
\begingroup
It remains to prove Proposition~\ref{sl:prop:dynamic}, which excludes
long excursions that remain close to $\Gamma^0$ on almost all their
rows. The argument compares the weight penalty for following a
narrow tube with the weight of an actual geodesic segment. We first
explain this comparison, then establish the estimate for a deterministic
reference curve, and finally apply it to the random curve $\Gamma^0$.
We refine the arguments of \cite[Section 4]{GH23} and
\cite[Section 9]{GH24}, retaining their probabilistic inputs but
changing the counting of sampling grids and endpoint locations so
that a fixed tube parameter remains permitted as the scale grows.
\par\endgroup
\par\endgroup

\begingroup
\begingroup
Throughout this section, paths and their endpoints are in the
original BLPP coordinates. Thus \(\pi(m)\) is the departure
location on integer row \(m\), as defined in
{Section~\ref{sec:model}}.
For any scale \(N>0\) and a staircase \(\pi\) from
\((x,i)\) to \((y,k)\), we define {its centered weight at dynamical time \(t\) by}
\begin{equation}\label{sl:eq:chart}
 {\W_N^t(\pi)
 =N^{-1/3}\bigl[\wgt^t(\pi)-(y-x)-(k-i)\bigr].}
\end{equation}
\begingroup
This is the weight of a specified path. The {optimised} value introduced
before Lemma~\ref{sl:lem:compactweightinput} is therefore
\[
 {W_N^0}(p,q)=\sup_{\pi:p\to q}{\W_N^0}(\pi)
\]
in the static environment.
\par\endgroup
Note that for the same path at another scale \(D\), we have
\begin{equation}\label{sl:eq:weightscale}
 {\W_N^t(\pi)=(D/N)^{1/3}\W_D^t(\pi).}
\end{equation}
{The static estimates below are stated at time zero, with the
superscript \(0\) retained. By stationarity, they also hold at any
fixed dynamical time. We now discuss the weight comparison that
will rule out slender excursions.}

At scale \(N_j\), weight regularity gives every actual excursion
\(\pi\) of \(\Gamma^t\) a lower bound \(-B_j\) for
\(\W_{N_j}^t(\pi)\), because the excursion is itself a geodesic
between its endpoints; the allowance \(B_j\) will be chosen in
\eqref{sl:eq:allowances}. On the other hand,
{Proposition~\ref{sl:prop:local} below shows that, for a deterministic
reference curve, with high probability every path of row duration
in \([N_j/2,N_j]\) that stays within a tube of width of order
\(\theta_jN_j^{2/3}\) around that curve on most rows has centered
weight strictly below \(-c_*\theta_j^{-1}\).} We will ensure that
\(B_j<c_*\theta_j^{-1}\), so an actual slender excursion would satisfy
\[
 \W_{N_j}^t(\pi)\ge-B_j>-c_*\theta_j^{-1}.
\]
It would therefore be a slender path with unusually high weight.

However, the reference curve in this application is \(\Gamma^0\), which is
random and correlated with the time-\(t\) environment. Uniformity
over deterministic reference curves {(Proposition~\ref{sl:prop:local})} does not by itself justify this
substitution. The additional fact we use is that an excursion has
interior disjoint from \(\Gamma^0\). After conditioning on \(\Gamma^0\),
the exterior FKG comparison in Lemma~\ref{sl:lem:exteriorinput}
bounds the probability of a high-weight exterior path by the
corresponding probability in fresh Brownian noise, with the reference
curve held fixed. Consequently, this makes the deterministic-curve estimate
applicable even to the random path \(\Gamma^0\), thereby yielding the rarity of slender excursions.

\subsection{\texorpdfstring{{Slender paths near a deterministic reference curve}}{Slender paths near a deterministic reference curve}}\label{sec:deterministicslender}

\begingroup
We begin with paths that stay close to a deterministic reference
function. Horizontal translations and integer row translations
preserve the law of Brownian increments and the {centered}-weight
{normalisation} in \eqref{sl:eq:chart}. We therefore state the local
estimates with the lower row at zero and the central line $x=m$;
the same estimates hold in any translated window, with unchanged
constants.

Fix $N\in\mathbb N$ and $R\ge1$. Let
$\phi:[0,2N]\cap\mathbb Z\to\mathbb R$ be a deterministic
reference function satisfying
\[
 |\phi(m)-m|\le2RN^{2/3}
 \quad\text{for every }m\in[0,2N]\cap\mathbb Z.
\]
It prescribes a horizontal location on each row, within a strip
around the $\{x=m\}$ passing through the origin. No path structure or modulus
of continuity is assumed for $\phi$.
\par\endgroup

\begingroup
\begingroup
Given $\theta>0$ and $\chi\in(0,1)$, we call a staircase
$\pi$ from $(x,i)$ to $(y,k)$ \emph{admissible} if its
lifetime lies in $[0,2N]$, its row duration $k-i$ belongs
to $[N/2,N]$, and both endpoints lie in the prescribed strip:
\[
 \begin{gathered}
 0\le i\le k\le2N,\qquad N/2\le k-i\le N,\\
 |x-i|\le4RN^{2/3},\qquad |y-k|\le4RN^{2/3}.
 \end{gathered}
\]
\par\endgroup
In addition, the path must follow the much narrower tube around
\(\phi\) on all but a small fraction of its rows:
\[
 \#\{m\in[i,k]\cap\mathbb Z:
             |\pi(m)-\phi(m)|>2\theta N^{2/3}\}
 \le\chi(k-i).
\]
Thus \(R\) controls the broad strip containing the reference
curve and the two endpoints, while \(\theta\) controls how
closely the path follows the reference curve on most rows.
Figure~\ref{fig:admissiblepath} depicts these separate requirements.
\par\endgroup
\par\endgroup

\begin{figure}[tbp]
\centering
\begingroup
\begin{tikzpicture}[x=.87cm,y=1.02cm,font=\small,text=black]
 \fill[black!7] (.5,0)--(5.5,0)--(8.36,5.2)--(3.36,5.2)--cycle;
 \filldraw[fill=yellow!24,draw=witnessgold!65,line width=.35pt] (1.75,0)--(4.25,0)--(7.11,5.2)--(4.61,5.2)--cycle;
 \draw[black!35,dashed] (.5,0)--(3.36,5.2);
 \draw[black!35,dashed] (5.5,0)--(8.36,5.2);
 \draw[guide] (3,0)--(5.86,5.2);
 \filldraw[fill=pathgreen!35,draw=pathgreen!70,line width=.4pt]
 (2.76,0)--(3.16,.8)--(3.46,1.2)--(3.66,1.6)--(3.96,2)
 --(4.06,2.4)--(4.26,2.8)--(4.56,3.2)--(4.76,3.6)
 --(4.96,4)--(5.16,4.4)--(5.46,5.2)
 --(6.14,5.2)--(5.84,4.4)--(5.64,4)--(5.44,3.6)
 --(5.24,3.2)--(4.94,2.8)--(4.74,2.4)--(4.64,2)
 --(4.34,1.6)--(4.14,1.2)--(3.84,.8)--(3.44,0)--cycle;
 \draw[proxypath]
 (3.1,0)--(3.5,.8)--(3.8,1.2)--(4,1.6)--(4.3,2)
 --(4.4,2.4)--(4.6,2.8)--(4.9,3.2)--(5.1,3.6)
 --(5.3,4)--(5.5,4.4)--(5.8,5.2);
 \draw[staticpath]
 (3.2,1.6)--(3.45,1.6)
 --(3.45,2)--(4.42,2)--(4.42,2.4)--(5.16,2.4)
 --(5.16,2.8)--(5.20,2.8)--(5.20,3.2)--(5.22,3.2)
 --(5.22,3.6)--(5.3,3.6)--(5.3,4)--(5.5,4);
 \foreach \x/\y in {4.42/2,5.22/3.2,5.3/3.6,5.5/4}
   \fill[pathblue] (\x,\y) circle (1.4pt);
 \foreach \x/\y in {3.45/1.6,5.16/2.4,5.20/2.8} {
   \draw[pathorange,line width=1.2pt]
     (\x-.085,\y-.085)--(\x+.085,\y+.085)
     (\x-.085,\y+.085)--(\x+.085,\y-.085);
 }
 \fill[pathblue] (3.2,1.6) circle (1.8pt);
 \node[anchor=south east,text=black] at (3.2,1.6) {$(x,i)$};
 \node[anchor=south west,text=black] at (5.5,4) {$(y,k)$};
 \draw[guide] (-.1,1.6)--(3.2,1.6);
 \draw[guide] (-.1,4)--(5.5,4);
 \draw[<->,black!60] (.05,1.6)--(.05,4)
   node[midway,left,align=right] {$d=k-i$\\$N/2\le d\le N$};
 \node[anchor=east,text=black] at (0,0) {$0$};
 \node[anchor=east,text=black] at (2.85,5.2) {$2N$};
 \draw[black!40] (.35,0)--(5.65,0);
 \draw[black!40] (3.21,5.2)--(8.51,5.2);
 \node[anchor=north,text=black] at (3.0,-.10) {line $x=m$};
  \draw[->,black!55] (9.0,5.0)--(6.9,4.9);
  \node[anchor=west,align=left,font=\footnotesize] at (9.05,5.0)
    {{reference strip:}\\{$|\phi(m)-m|$}\\{$\le2RN^{2/3}$}};
  \draw[->,black!55] (9.0,3.4)--(5.35,3.6);
  \node[anchor=west,align=left,font=\footnotesize] at (9.05,3.4)
    {narrow tube:\\$|\pi(m)-\phi(m)|$\\$\le2\theta N^{2/3}$};
  \draw[->,black!55] (9.0,1.8)--(5.25,2.4);
  \node[anchor=west,align=left,font=\footnotesize] at (9.05,1.8)
    {{bad departures:}\\at most $\chi d$\\bad rows};
  \draw[->,black!55] (9.0,.2)--(5.9,.9);
  \node[anchor=west,align=left,font=\footnotesize] at (9.05,.2)
    {{endpoint strip:}\\{$|x-i|,\ |y-k|$}\\{$\le4RN^{2/3}$}};
\end{tikzpicture}
\caption{An admissible path in original BLPP coordinates.
The blue staircase has duration between $N/2$ and $N$ and its
endpoints lie in the {broad light-gray strip}. The dashed reference function
$\phi$ lies in the {smaller pale-yellow strip}. Most departure points
of the staircase lie in the {narrow green tube} around $\phi$;
{the initial point lies outside the tube, and the orange
crosses mark bad departures.} The strip and tube widths,
and the proportion of bad rows, are schematic. The reference
function specifies a location on each integer row; its dashed
interpolation is only a visual guide.}
\label{fig:admissiblepath}
\par\endgroup
\end{figure}
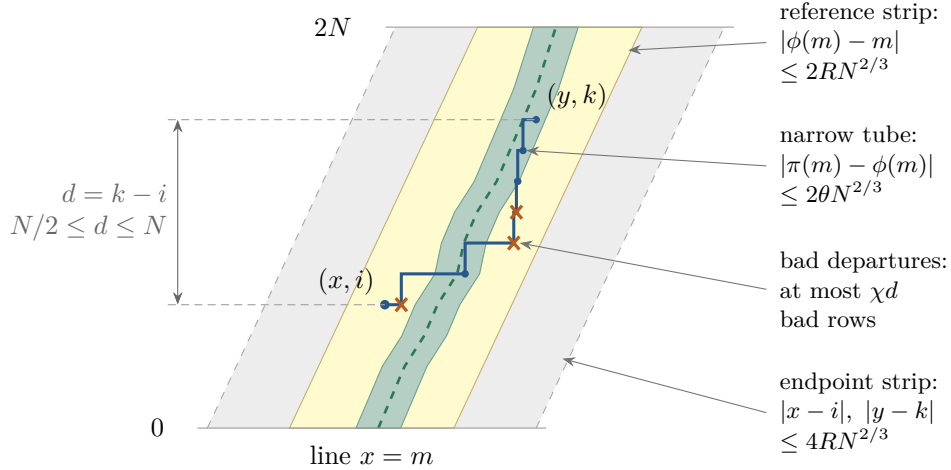

\begingroup
The next proposition is a local strengthening of
\cite[Theorem 1.10]{GH23}. It removes that statement's
restriction $\theta^{-1/4}>C\log N$, allowing a fixed
$\theta>0$ as $N\to\infty$, provided
{$N\ge C\theta^{-300}$}.
\par\endgroup
\begin{proposition}\label{sl:prop:local}
{There are constants \(\chi>0,\theta_*>0,c_*,c,C>0\)
such that the following holds.}
If
\begin{equation}\label{sl:eq:localrange}
 0<\theta\le\theta_*,
 \qquad 1\le R\le\theta^{-1/4},
 \qquad {N\ge C\theta^{-300}},
\end{equation}
then, {uniformly in \(N\) and the deterministic reference
function \(\phi\)},
\begin{equation}\label{sl:eq:localclaim}
 \PP\!\left(\exists\text{ admissible }\pi:
                    {\W_N^0}(\pi)\ge-c_*\theta^{-1}\right)
 \le C e^{-c\theta^{-{3/2}}}.
\end{equation}
In particular, \(\theta\) can be fixed while \(N\to\infty\).
\end{proposition}
To heuristically see why the penalty has order \(\theta^{-1}\), first match
the width of the tube to the transverse scale of a shorter path.
A row duration of order \(\theta^{3/2}N\) has transverse scale
\[
 (\theta^{3/2}N)^{2/3}=\theta N^{2/3}.
\]
Thus a duration-\(N\) path contains order \(\theta^{-3/2}\) pieces
at this shorter scale. The mechanism behind
\cite[Proposition 4.5]{GH23}, imported as
Lemma~\ref{sl:lem:sparseinput}, is that the {optimised} {centered}
weight of a crossing between sufficiently short prescribed intervals
{has negative mean. This originates from the negative mean of the
GUE Tracy--Widom distribution; see \cite[Lemma A.4]{BGHH22}.
Its magnitude, measured in scale-\(N\) units, is}
of order
\[
 N^{-1/3}(\theta^{3/2}N)^{1/3}=\theta^{1/2}.
\]
\begingroup
{The negative mean is first established for fixed endpoints.
Allowing the endpoints to vary within prescribed intervals increases
the maximal weight. To keep its mean negative, those intervals are
chosen to be a sufficiently small fraction of the transverse scale
of the shorter piece. The constant \(b\) in
Lemma~\ref{sl:lem:sparseinput} makes this choice precise.}
\par\endgroup
Summing over order \(\theta^{-3/2}\) constrained pieces therefore
gives a penalty of order
\(\theta^{-3/2}\theta^{1/2}=\theta^{-1}\).
{Lemma~\ref{sl:lem:sparseinput} below gives a probability bound
for this total weight when the constraint intervals are prescribed
deterministically, even if a small fraction of the constraint rows
are omitted. To apply it to an admissible path, we must select
constraint intervals from the rows on which the path is close to
the reference curve. The sampling argument below does this without
introducing a counting factor that grows with \(N\).}
\par\endgroup

\begingroup
{We will use three ingredients in the proof of
Proposition~\ref{sl:prop:local}.} Lemma~\ref{sl:lem:sparseinput},
imported from \cite{GH23}, supplies the weight penalty for a
deterministic collection of sparse constraints. Lemma~\ref{sl:lem:gridcuts}
shows that each admissible path gives such a collection for at
least half the translations of a sampling grid. Finally,
Lemma~\ref{sl:lem:endpieces} bounds the weights of the two pieces
discarded when cutting to the first and last {good grid rows}.
We state these ingredients below and then prove
Proposition~\ref{sl:prop:local}, postponing the proofs of
Lemmas~\ref{sl:lem:gridcuts} and~\ref{sl:lem:endpieces} to the
next two subsections.
\par\endgroup

\begingroup
\begin{lemma}[Prescribed sparse constraints; adapted from
{\cite[Proposition 4.5]{GH23}}]\label{sl:lem:sparseinput}
There are constants \(\kappa_0,\chi_0\in(0,1/2]\) and
\(b,d_1,d_2,C>0\) with the following property. Let
\(N,D\in\mathbb N\) and \(0<\kappa<\kappa_0\) satisfy\footnotemark
\[
 \begin{gathered}
 N/3\le D\le N,\qquad N\kappa^{3/2}\in\mathbb N,\qquad
 \frac{D}{N\kappa^{3/2}}\in\mathbb N,\\
 N\ge C\kappa^{-195}.
 \end{gathered}
\]
\footnotetext{{Although not listed in
\cite[Proposition 4.5]{GH23}, the assumption
\(r\ge C_0\kappa^{-195}\) appears in Lemmas 4.9 and 4.14
there, whose endpoint-oscillation estimates are used in its proof.
{The bound ensures that consecutive constraint intervals
are not too far apart horizontally, relative to the diagonal and
their row separation, to apply \cite[Proposition 4.4]{GH23}.}
After the change of coordinates below, the displayed bound on
\(N\) is sufficient for both lemmas.}}
\begingroup
Let $\mathcal C$ be a deterministic collection of horizontal
segments of length $2b\kappa N^{2/3}$, on distinct rows of
\[
 (N\kappa^{3/2}\mathbb Z)\cap[0,D],
\]
\par\endgroup
and we assume that the collection $\mathcal C$ includes the two endpoint rows and omits at most
\(\chi_0D/(N\kappa^{3/2})\) of the intermediate grid rows.
Suppose also that every point \((x,m)\) of every segment satisfies
{\(|x-m|\le2D^{2/3+1/20}\)}. Let \(\Pi(\mathcal C)\)
comprise the staircases starting in the lowest segment, ending
in the highest segment, and visiting every segment in \(\mathcal C\).
Then, we have
\[
 \PP\!\left(\sup_{\pi\in\Pi(\mathcal C)}{\W_N^0}(\pi)
                    \ge-d_1\kappa^{-1}\right)
 \le e^{-d_2\kappa^{-3/2}}.
\]
\end{lemma}

\begin{proof}
We express \cite[Proposition 4.5]{GH23} in original BLPP
coordinates, absorbing its fixed weight-{normalisation} factor
into the constants. To apply it at the actual duration \(D\),
use the parameter \(\widehat\kappa=\kappa(N/D)^{2/3}\).
It describes exactly the prescribed row spacing and segment length:
\[
 \widehat\kappa^{3/2}D=N\kappa^{3/2},\qquad
 \widehat\kappa D^{2/3}=\kappa N^{2/3}.
\]
\begingroup
The assumptions \(N\kappa^{3/2}\in\mathbb N\) and
\(D/(N\kappa^{3/2})\in\mathbb N\) give
\[
 \widehat\kappa^{3/2}\in D^{-1}\mathbb N,
 \qquad \widehat\kappa^{-3/2}\in\mathbb N,
\]
as required in the cited proposition. Its horizontal coordinate
is \((x-m)/(2D^{2/3})\), so our intervals of length
\(2b\kappa N^{2/3}\) become intervals of length
\(b\widehat\kappa\), explaining the factor \(2\) in our statement.
\par\endgroup
Since
\(\kappa\le\widehat\kappa\le3^{2/3}\kappa\), decreasing
\(\kappa_0\) ensures its upper bound on \(\widehat\kappa\).
Increasing \(C\) also ensures the sufficient condition
\(D\ge C'\widehat\kappa^{-195}\) used in its proof.

\begingroup
Let \({\chi_{\rm src}}>0\) denote the upper bound on the parameter
\(\chi\) in \cite[Proposition 4.5]{GH23}. Choose
\(3\chi_0<{\chi_{\rm src}}\), and then decrease \(\kappa_0\)
so that \(3\kappa_0^{3/2}\le\chi_0\). Since
\(\widehat\kappa^{3/2}\le3\kappa_0^{3/2}\le\chi_0\),
the interval \([2\chi_0,3\chi_0]\) contains a multiple
\(\widehat\chi\) of \(\widehat\kappa^{3/2}/2\). This choice satisfies
\[
 0<\widehat\chi<{\chi_{\rm src}},\qquad
 2\widehat\chi\widehat\kappa^{-3/2}\in\mathbb N,\qquad
 \widehat\chi\ge2\chi_0\ge2\widehat\kappa^{3/2}.
\]
The collection omits at most \(\chi_0\widehat\kappa^{-3/2}\)
rows and hence obeys the required count with allowance
\(\widehat\chi\). Finally, its confinement bound becomes
\(|(x-m)/(2D^{2/3})|\le D^{1/20}\) in {the coordinates of \cite{GH23}}.
\par\endgroup

\begingroup
Let \(d_1',d_2'>0\) be the weight and tail constants supplied
by the cited estimate. Since \(D\ge N/3\) and
\(\widehat\kappa\le3^{2/3}\kappa\), the conversion
\eqref{sl:eq:weightscale} gives
\[
 d_1'(D/N)^{1/3}\widehat\kappa^{-1}
       \ge\frac{d_1'}3\kappa^{-1},\qquad
 d_2'\widehat\kappa^{-3/2}
       \ge\frac{d_2'}3\kappa^{-3/2}.
\]
Thus \(d_1=d_1'/3\) and \(d_2=d_2'/3\) give the asserted bound.
\par\endgroup
\end{proof}

{To apply Lemma~\ref{sl:lem:sparseinput} in the proof
of Proposition~\ref{sl:prop:local}, we must extract a collection
of prescribed intervals visited by an admissible path. Such a path follows} \(\phi\) on most integer rows, but
its bad rows could be concentrated on one fixed grid of spacing
\(N\kappa^{3/2}\). We therefore consider every integer translation
of that grid: few bad rows overall will ensure that many
translations sample mostly good rows. {Whereas the proof of
\cite[Lemma 4.15]{GH23} selects one suitable translation, we use
the fact that at least half of the translations work. This allows
us to average over translations rather than take a union bound.}
For each such translation,
we keep the portion between the first and last {good grid rows}.
Its duration \(D\) may differ from \(N\), which is why the
preceding lemma allows the whole range \([N/3,N]\). No change
of sampling parameter is needed when we apply it.
\par\endgroup
\begingroup
We choose the constant \(\chi\) in Proposition~\ref{sl:prop:local}
once and for all so that
\begin{equation}\label{sl:eq:chichoice}
 {15\chi\le\chi_0,\qquad \chi\le1/50,\qquad
 C\chi\log(1/\chi)<d_2/2,}
\end{equation}
where \(\chi_0,d_2\) come from Lemma~\ref{sl:lem:sparseinput}
and \(C\) is the absolute constant in the counting bound
\eqref{sl:eq:entropy} below. We then choose \(\theta_*\)
sufficiently small in terms of these fixed constants. These
choices are independent of \(N,R,\theta,\phi\).
Fix \(N,R,\theta,\phi\) as in Proposition~\ref{sl:prop:local}.
With \(b\) as in Lemma~\ref{sl:lem:sparseinput}, choose \(\kappa\) so that
\par\endgroup
\begin{equation}\label{sl:eq:gridspacing}
 N\kappa^{3/2}\in\mathbb N,\qquad
 4\theta\le b\kappa\le8\theta,\qquad
 \kappa^{3/2}\le\chi/100.
\end{equation}
\begingroup
{To make this choice, take the integer spacing \(N\kappa^{3/2}\)
in the interval}
\[
 {\left[N(4\theta/b)^{3/2},\,N(8\theta/b)^{3/2}\right].}
\]
{Its length is a fixed positive constant times \(N\theta^{3/2}\),
which is at least one under \eqref{sl:eq:localrange}, so such an
integer exists. The resulting \(\kappa\) satisfies
\(4\theta\le b\kappa\le8\theta\), and hence \(\kappa\asymp\theta\).
Decreasing \(\theta_*\) ensures \(\kappa^{3/2}\le\chi/100\).
This last condition controls the loss from aligning the path's
endpoint rows with a sampling grid: for a path of row duration
\(d\ge N/2\), one
grid spacing satisfies \(N\kappa^{3/2}\le\chi d/50\).
It is absorbed into the discarded-duration bound
\eqref{sl:eq:cutdistances} in the proof of
Lemma~\ref{sl:lem:gridcuts}.}
\par\endgroup For each offset
\(r\in\{0,\ldots,N\kappa^{3/2}-1\}\), set
\[
 \begin{aligned}
 \mathcal R_r&=(r+N\kappa^{3/2}\mathbb Z)\cap[0,2N],\\
 I_v&=[\phi(v)-b\kappa N^{2/3},\phi(v)+b\kappa N^{2/3}]
       \times\{v\},\qquad v\in\mathcal R_r.
 \end{aligned}
\]
\begingroup
For a path \(\pi\), call a row \(v\) good if
\(|\pi(v)-\phi(v)|\le2\theta N^{2/3}\); at a {good grid row},
its departure point lies in \(I_v\). Each bad row belongs to exactly
one translated grid, so few bad rows overall force a small
bad-row fraction for many offsets. The following lemma makes
this precise. For each such offset, we keep the portion of
\(\pi\) between its departures on the first and last {good grid rows}, discarding the two end pieces. The lemma controls
both the duration of this central portion and the constraints
it visits.
\par\endgroup

\begin{lemma}[Many usable sampling grids]\label{sl:lem:gridcuts}
Let \(\pi\) be an admissible path {in the setting of
Proposition~\ref{sl:prop:local}}, with endpoint rows \(i,k\)
and duration \(d=k-i\in[N/2,N]\). {For at least half the offsets
\(r\in\{0,\ldots,N\kappa^{3/2}-1\}\), the set of good rows in
\(\mathcal R_r\cap[i,k]\) has at least two elements.
Let \(v_-\) and \(v_+\) be its smallest and largest elements.
The portion \(\pi_{\rm c}\) between the departures at these rows satisfies}
\[
 D:=v_+-v_-\in[N/3,N],\qquad
 \max\{v_--i,k-v_+\}\le5\chi d.
\]
This path visits both endpoint segments and all but at most
{\(5\chi\kappa^{-3/2}\)} of the intermediate segments \(I_v\),
\(v\in\mathcal R_r\cap(v_-,v_+)\).
\end{lemma}

\begingroup
{We prove Lemma~\ref{sl:lem:gridcuts} in
Section~\ref{sec:samplingproof}. It controls the central portion
of the path, as illustrated in Figure~\ref{fig:gridcuts}.
The next input bounds both discarded end pieces, uniformly over
their endpoints, even when their row duration is very short.
Its proof is given in Section~\ref{sec:endpieceproof}.}

\begin{lemma}[Uniform control of the end pieces]\label{sl:lem:endpieces}
{There is an absolute constant \(C_1>0\) such that,
for every fixed \(\delta\in(0,1]\), the conditions
\eqref{sl:eq:localrange} of Proposition~\ref{sl:prop:local}, with
\(\theta\) sufficiently small depending on \(\delta\), imply the
existence of an event \(\cU\) with
\(\PP(\cU^c)\le C_\delta e^{-c_\delta\theta^{-3/2}}\), where
\(c_\delta,C_\delta>0\) depend only on \(\delta\), on which}
\[
 {{\W_N^0}(\sigma)\le2C_1\delta\theta^{-1}}
\]
simultaneously for every path segment \(\sigma\) whose endpoints
\((x,i),(y,k)\) satisfy
\[
 i,k\in[0,2N]\cap\mathbb Z,\qquad 0\le k-i\le N,\qquad
 \max\{|x-i|,|y-k|\}\le4RN^{2/3}.
\]
\end{lemma}
\par\endgroup
\begingroup
We now complete the proof of Proposition~\ref{sl:prop:local},
postponing the proofs of Lemmas~\ref{sl:lem:gridcuts}
and~\ref{sl:lem:endpieces} to the next two subsections.
\par\endgroup
\begingroup
\begin{proof}[Proof of Proposition~\ref{sl:prop:local}
{assuming Lemmas~\ref{sl:lem:gridcuts} and~\ref{sl:lem:endpieces}}]
We first bound the probability of a high-weight central path
for one sampling grid. We then use Lemmas~\ref{sl:lem:gridcuts}
and~\ref{sl:lem:endpieces} to show that an admissible high-weight
path forces this event for at least half the offsets.

\paragraph{{\textbf{Step 1. A prescribed collection of constraints.}}} 
For a fixed offset \(r\), we first list all possible central cut
rows and collections of visited segments allowed by
Lemma~\ref{sl:lem:gridcuts}. Let
\begin{equation}\label{sl:eq:cutpairs}
 \mathcal P_r=\{(v_-,v_+)\in\mathcal R_r^2:
                    N/3\le v_+-v_-\le N\}.
\end{equation}
\begingroup
For each pair, write \(D=v_+-v_-\), and let
\(\mathcal C_r(v_-,v_+)\) comprise the subcollections of
\(\{I_v:v\in\mathcal R_r\cap[v_-,v_+]\}\) that contain
both endpoint segments and omit at most
{\(5\chi\kappa^{-3/2}\)} intermediate segments.
These families are deterministic; an element specifies segments
to be visited, rather than crossing points inside them.

Fix a pair \((v_-,v_+)\in\mathcal P_r\) and a collection
\(\mathfrak c\in\mathcal C_r(v_-,v_+)\). Let
\(\Pi(\mathfrak c)\) be the staircases starting in its lowest
segment, ending in its highest segment, and visiting every segment
of \(\mathfrak c\).
\begingroup
{We apply Lemma~\ref{sl:lem:sparseinput} after translating by
\((-v_-,-v_-)\), so the lower cut row becomes zero. A point
\((x,v)\) becomes \((x-v_-,v-v_-)\); its horizontal displacement
from the diagonal is unchanged because
\((x-v_-)-(v-v_-)=x-v\). This translation also preserves
centered weights and the law of the Brownian increments.} The duration, grid and interval-length
conditions hold by construction. Since \(D\ge N/3\) and
\(15\chi\le\chi_0\), the number of omitted rows is at most
\(5\chi\kappa^{-3/2}\le\chi_0D/(N\kappa^{3/2})\).
For every \((x,v)\in I_v\),
\[
 |x-v|\le4RN^{2/3}\le2D^{2/3+1/20}.
\]
Here we use the reference-function bound, \(b\kappa\le8\theta\),
\(R\le\theta^{-1/4}\), \(N\ge C\theta^{-300}\), and
\(D\ge N/3\). These bounds and \(\kappa\asymp\theta\)
also ensure \(N\ge C\kappa^{-195}\) and \(\kappa<\kappa_0\)
for sufficiently small \(\theta_*\).
\par\endgroup

Since \(b\kappa\le8\theta\), the lemma's threshold magnitude
satisfies \(d_1\kappa^{-1}\ge(d_1b/8)\theta^{-1}\).
We therefore fix \(c_1=d_1b/8\) and obtain directly
\begin{equation}\label{sl:eq:GHcore}
 \PP\!\left(\sup_{\pi\in\Pi(\mathfrak c)}{\W_N^0}(\pi)
                         \ge-c_1\theta^{-1}\right)
 \le e^{-d_2\kappa^{-3/2}}.
\end{equation}
\par\endgroup

\paragraph{{\textbf{Step 2. Counting collections within one grid.}}}
Define
\begin{equation}\label{sl:eq:Hrdefinition}
 \cH_r=
 \bigcup_{(v_-,v_+)\in\mathcal P_r}
 \ \bigcup_{\mathfrak c\in\mathcal C_r(v_-,v_+)}
 \left\{\sup_{\pi\in\Pi(\mathfrak c)}{\W_N^0}(\pi)
                              \ge-c_1\theta^{-1}\right\}.
\end{equation}
This is a fixed-offset version of the
\(\operatorname{HighSlenderWeight}\) event preceding
\cite[Lemma 4.15]{GH23}. It already allows every possible
pair of central cut rows in our local window.

\begingroup
There are at most \(C\kappa^{-3}\le C\theta^{-3}\) pairs
in \(\mathcal P_r\). For each pair, the list has at most
\(\kappa^{-3/2}\) intermediate grid rows, and a collection
is specified by omitting at most \(5\chi\kappa^{-3/2}\) of
them. The elementary binomial bound therefore gives
\begin{equation}\label{sl:eq:entropy}
 |\mathcal C_r(v_-,v_+)|
 \le\exp\{C\chi\log(1/\chi)\kappa^{-3/2}\}.
\end{equation}
\par\endgroup
{By \eqref{sl:eq:chichoice},
\(C\chi\log(1/\chi)<d_2/2\).} Equations
\eqref{sl:eq:GHcore}--\eqref{sl:eq:entropy} then imply
\begin{equation}\label{sl:eq:Hrprob}
 \PP(\cH_r)\le C\theta^{-3}e^{-c\theta^{-3/2}}
             \le Ce^{-c'\theta^{-3/2}},\qquad 0\le r<N\kappa^{3/2}.
\end{equation}
The choices counted here are the {grid cut rows} and the
omitted constraints. Their numbers depend on \(\theta\),
not on the number \(N\) of microscopic rows. The uniform
end-piece event in Lemma~\ref{sl:lem:endpieces} is what permits us to count these
coarse cut rows instead of the original endpoints.

\paragraph{{\textbf{Step 3. Averaging over the usable offsets.}}}
{Set \(c_*=c_1/4\) and fix \(\delta\in(0,1]\) so small that}
\begin{equation}\label{sl:eq:deltachoice}
 {4C_1\delta\le c_1/4.}
\end{equation}
Apply Lemma~\ref{sl:lem:endpieces} with this choice, decreasing
\(\theta_*\) if necessary. Denote the target event by
\[
 E=\{\exists\text{ admissible }\pi:
                       {\W_N^0}(\pi)\ge-c_*\theta^{-1}\}.
\]
On \(E\cap\cU\), fix a witnessing path \(\pi\).
By Lemma~\ref{sl:lem:gridcuts}, for at least half the offsets \(r\),
the central path visits a collection in \(\mathcal C_r(v_-,v_+)\).
\begingroup
At each good cut row \(v\),
\[
 |\pi(v)-v|\le|\phi(v)-v|+2\theta N^{2/3}
             \le4RN^{2/3},
\]
since \(R\ge1\) and \(\theta\le1\).
{The original endpoints satisfy the same bound
by admissibility. Thus each discarded piece has endpoints in
the strip required by Lemma~\ref{sl:lem:endpieces} and row
duration at most \(N\). On \(\cU\), that lemma bounds the
{centered} weight of each piece by \({2C_1\delta\theta^{-1}}\).}
\par\endgroup
{By weight additivity, subtracting these two
weights from \({\W_N^0}(\pi)\) gives}
\[
 {\W_N^0}(\pi_{\rm c})
 \ge-\frac{c_1}{4}\theta^{-1}-{4C_1\delta\theta^{-1}}
 \ge-\frac{c_1}{2}\theta^{-1},
\]
{by \eqref{sl:eq:deltachoice}. For each usable offset \(r\),
the cut pair \((v_-,v_+)\) lies in \(\mathcal P_r\), and
\(\pi_{\rm c}\) visits a collection
\(\mathfrak c\in\mathcal C_r(v_-,v_+)\).
Its weight is at least \(-c_1\theta^{-1}/2\), which exceeds
the threshold \(-c_1\theta^{-1}\) in
\eqref{sl:eq:Hrdefinition}. Thus this central path witnesses
\(\cH_r\). The end-piece bounds on \(\cU\) hold for all cuts
simultaneously, so the argument applies to every usable offset.
There are at least half as many usable offsets as total offsets,
giving the pathwise implication and indicator bound}
\begin{equation}\label{sl:eq:keyaverage}
 \begin{gathered}
 E\cap\cU\subseteq
 \left\{\sum_{r=0}^{N\kappa^{3/2}-1}\ind_{\cH_r}\ge N\kappa^{3/2}/2\right\},\\
 \ind_{E\cap\cU}\le\frac{2}{N\kappa^{3/2}}\sum_{r=0}^{N\kappa^{3/2}-1}\ind_{\cH_r}.
 \end{gathered}
\end{equation}
Taking expectations and using Lemma~\ref{sl:lem:endpieces}
and \eqref{sl:eq:Hrprob} gives
\[
 \begin{aligned}
 \PP(E)&\le\PP(\cU^c)+\frac{2}{N\kappa^{3/2}}\sum_{r=0}^{N\kappa^{3/2}-1}\PP(\cH_r)\\
 &{\le Ce^{-c\theta^{-3/2}}.}
 \end{aligned}
\]
{This proves \eqref{sl:eq:localclaim}.}
Averaging over the usable offsets
has removed the factor \(N\kappa^{3/2}\asymp\theta^{3/2}N\) that
a union over offsets would introduce. No independence between
offsets is used.

\end{proof}
\par\endgroup

\subsection{Proof of the sampling lemma}\label{sec:samplingproof}
\begingroup
The proof of \cite[Lemma 4.15]{GH23} selects one compatible
offset. Here we retain at least half of all offsets. A single
admissible high-weight path will therefore force many of the
fixed-offset events \(\cH_r\), allowing us to average their
indicators instead of taking a union over offsets. This avoids
the factor \(N\kappa^{3/2}\) and is one of the changes that
removes the logarithmic restriction on \(\theta\).

We now prove the deterministic sampling statement used above.
The key observation is that each bad row belongs to exactly one
translated grid, so the total number of bad rows controls the
fraction of unusable offsets.

\begin{proof}[Proof of Lemma~\ref{sl:lem:gridcuts}]
For each offset \(r\), let \(B_r(\pi)\) count the {bad grid rows} in \(\mathcal R_r\cap[i,k]\). The residue classes partition
the rows, so admissibility gives
\begin{equation}\label{sl:eq:shiftaverage}
 \sum_{r=0}^{N\kappa^{3/2}-1}B_r(\pi)\le\chi d.
\end{equation}
Consequently, at least half the offsets satisfy
\begin{equation}\label{sl:eq:goodshift}
 B_r(\pi)\le\frac{4\chi d}{N\kappa^{3/2}}.
\end{equation}
Indeed, otherwise the sum in \eqref{sl:eq:shiftaverage} would
exceed \(2\chi d\). These offsets may depend on the path,
but the count holds deterministically for every admissible path.

{Fix an offset satisfying \eqref{sl:eq:goodshift}.
The number of good grid rows in \([i,k]\) is at least
\((1-4\chi)d/(N\kappa^{3/2})-1\), which is at least two by
\eqref{sl:eq:chichoice} and \eqref{sl:eq:gridspacing}.
Let \(v_-,v_+\) be the first and last such rows.}

{Starting at \(i\), reaching the first grid row costs less
than one spacing \(N\kappa^{3/2}\), even if that row is good.
If it is bad, each move to the next grid row costs one further
spacing. There are at most \(B_r(\pi)\) bad grid rows before
\(v_-\). The same argument, read backwards from \(k\), applies
to \(v_+\). Hence}
\begin{equation}\label{sl:eq:cutdistances}
 {\begin{aligned}
 \max\{v_--i,k-v_+\}
 &\le\bigl(B_r(\pi)+1\bigr)N\kappa^{3/2}\\
 &\le4\chi d+N\kappa^{3/2}\le5\chi d.
 \end{aligned}}
\end{equation}
{The \(1\) accounts for the initial distance to a grid row;
the factor \(N\kappa^{3/2}\) converts the number of grid steps
into a row duration. The second inequality uses
\eqref{sl:eq:goodshift}, and the last uses
\(N\kappa^{3/2}\le\chi d/50\), which follows from
\eqref{sl:eq:gridspacing} and \(d\ge N/2\).
Since \(\chi\le1/50\) by \eqref{sl:eq:chichoice}, this implies}
\begin{equation}\label{sl:eq:central}
 {D=v_+-v_-\in[N/3,N].}
\end{equation}
\begingroup
For the number of omitted constraints, it is enough to use
\(d\le N\) directly:
\[
 B_r(\pi)\le\frac{4\chi d}{N\kappa^{3/2}}
          \le4\chi\kappa^{-3/2}\le5\chi\kappa^{-3/2}.
\]
\par\endgroup
The central path starts and ends in the segments on its good
cut rows, and visits the prescribed segment at every intervening
{good grid row}. The bound on \(B_r(\pi)\) therefore gives the
claimed number of omitted segments.
Figure~\ref{fig:gridoffsets} illustrates this deterministic count.

\end{proof}
\par\endgroup

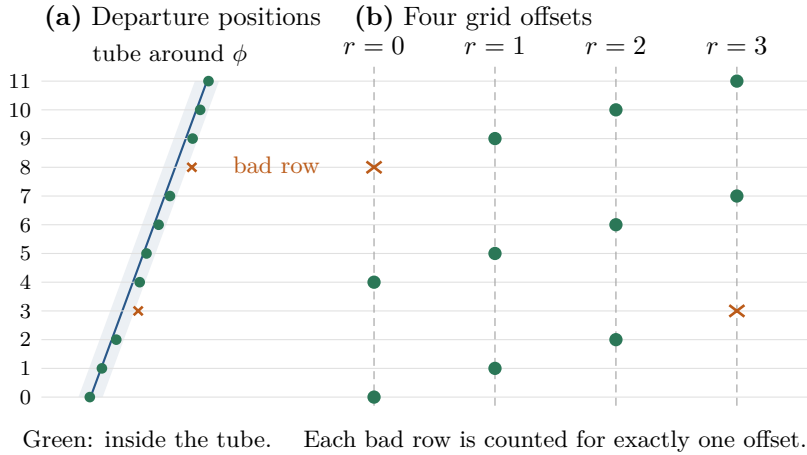
\begin{figure}[tbp]
\centering
\begin{tikzpicture}[x=1cm,y=.38cm]
 \node[anchor=west] at (-.1,13.2) {\textbf{(a)} Departure positions};
 \node[anchor=west] at (4,13.2) {\textbf{(b)} Four grid offsets};
 \fill[pathblue!9] (.49,0)--(.77,2)--(1.05,4)--(1.33,6)--(1.61,8)
   --(2.03,11)--(2.35,11)--(1.93,8)--(1.65,6)--(1.37,4)
   --(1.09,2)--(.81,0)--cycle;
 \draw[pathblue,line width=.8pt] (.65,0)--(.93,2)--(1.21,4)
   --(1.49,6)--(1.77,8)--(2.19,11);
 \node[above,font=\footnotesize] at (1.7,11.2) {tube around $\phi$};
 \foreach \y in {0,...,11} {
   \draw[black!12] (0,\y)--(10.2,\y);
   \node[left,font=\scriptsize] at (0,\y) {$\y$};
 }
 \foreach \x/\y in {.64/0,.80/1,.99/2,1.30/4,1.39/5,1.55/6,1.70/7,2.00/9,2.10/10,2.21/11}
   \fill[pathgreen] (\x,\y) circle (2pt);
 \foreach \x/\y in {1.28/3,1.99/8} {
   \draw[pathorange,line width=1pt] (\x-.06,\y-.15)--(\x+.06,\y+.15)
     (\x-.06,\y+.15)--(\x+.06,\y-.15);
 }
 \node[anchor=west,font=\footnotesize,pathorange] at (2.4,8.1) {bad row};
 \foreach \x/\r in {4.4/0,6/1,7.6/2,9.2/3} {
   \draw[guide] (\x,-.3)--(\x,11.5);
   \node[above] at (\x,11.6) {$r=\r$};
 }
 \foreach \x/\y in {4.4/0,4.4/4,6/1,6/5,6/9,7.6/2,7.6/6,7.6/10,9.2/7,9.2/11}
   \fill[pathgreen] (\x,\y) circle (2.5pt);
 \foreach \x/\y in {4.4/8,9.2/3}
   \draw[pathorange,line width=1pt] (\x-.1,\y-.2)--(\x+.1,\y+.2)
     (\x-.1,\y+.2)--(\x+.1,\y-.2);
 \node[below,align=center,font=\footnotesize] at (6.8,-.8)
   {Each bad row is counted for exactly one offset.};
 \node[below,align=center,font=\footnotesize] at (1.4,-.8)
   {Green: inside the tube.};
\end{tikzpicture}
\caption{The deterministic grid-offset count, shown with
\(N\kappa^{3/2}=4\) for illustration. The dots on the left are departure
positions in original coordinates on successive
integer rows; their horizontal locations are nondecreasing.
They do not depict the horizontal pieces of a staircase.
Each column on the right samples one residue
class modulo \(N\kappa^{3/2}\). Since the bad rows are partitioned
among the offsets, few bad rows in total imply a small
bad-row fraction for at least half the offsets, with the
allowance in \eqref{sl:eq:goodshift}. For each such offset
the path is cut at its first and last {good grid rows}.
One high-weight path therefore forces many events
\(\cH_r\), which gives the averaged indicator bound
\eqref{sl:eq:keyaverage}.}
\label{fig:gridoffsets}
\end{figure}

\begin{figure}[tbp]

\centering
\begin{tikzpicture}[x=.9cm,y=.72cm]
 \node[anchor=west] at (-.2,7) {\textbf{(a)} One translated sampling grid};
 \fill[pathblue!8] (1,0)--(1.4,2)--(1.8,4)--(2.2,6)
       --(3.2,6)--(2.8,4)--(2.4,2)--(2,0)--cycle;
 \draw[pathblue] (1.5,0)--(1.9,2)--(2.3,4)--(2.7,6);
 \node[pathblue,above] at (2.7,6.1) {$\phi$};
 \foreach \y in {1,...,5} {
   \draw[guide] (0,\y)--(4.4,\y);
   \draw[pathblue,line width=2pt] ({1.1+.2*\y},\y)--({1.9+.2*\y},\y);
 }
 \draw[dynamicpath,densely dashed] (.85,0)--(1.0,1)--(2.0,2);
 \draw[dynamicpath] (2.0,2)--(2.3,3)--(2.35,4);
 \draw[dynamicpath,densely dashed] (2.35,4)--(3.3,5)--(3.5,6);
 \foreach \x/\y in {2/2,2.3/3,2.35/4} \fill[pathgreen] (\x,\y) circle (2.5pt);
 \foreach \x/\y in {1.0/1,3.3/5} {
  \draw[pathorange,line width=1pt] (\x-.1,\y-.1)--(\x+.1,\y+.1)
                          (\x-.1,\y+.1)--(\x+.1,\y-.1);
 }
 \node[left] at (0,0) {$i$}; \node[left] at (0,6) {$k$};
 \node[left] at (0,2) {$v_-$}; \node[left] at (0,4) {$v_+$};
 \draw[<->] (4.8,2)--node[right,font=\footnotesize] {$N\kappa^{3/2}$ rows}(4.8,3);
 \draw[<->] (1.7,3.3)--(2.5,3.3);
 \node[anchor=east,font=\footnotesize,black,fill=white,inner sep=1pt] at (1.35,3.4) {$2b\kappa N^{2/3}$};
 \draw[guide] (1.38,3.3)--(1.65,3.3);
 \begin{scope}[xshift=7.3cm]
  \node[anchor=west] at (-.3,7) {\textbf{(b)} The three pieces};
  \draw[dynamicpath,densely dashed] (1,0)--(1,2);
  \draw[dynamicpath,line width=2pt] (1,2)--(1,4);
  \draw[dynamicpath,densely dashed] (1,4)--(1,6);
  \foreach \y/\lab in {0/i,2/v_-,4/v_+,6/k} {
    \fill (1,\y) circle (2pt); \node[left] at (.85,\y) {$\lab$};
  }
  \node[anchor=west,align=left,text width=3.8cm] at (1.6,1)
       {Discard the initial piece.\\Its weight is at most ${2C_1\delta\theta^{-1}}$.};
  \node[anchor=west,align=left,text width=3.8cm] at (1.6,3)
       {Retain $\pi_{\rm c}$.\\{It visits the segments at all good grid rows.}};
  \node[anchor=west,align=left,text width=3.8cm] at (1.6,5)
       {Discard the final piece.\\The same weight bound applies.};
 \end{scope}
\end{tikzpicture}
\caption{Sampling and cutting for a fixed offset $r$. The short blue
segments are the prescribed intervals $I_v$, separated by $N\kappa^{3/2}$ rows.
Green points lie in the tube and orange crosses mark {bad grid rows}.
The cuts are at the first and last {good grid rows}, $v_-$ and $v_+$.
The central part visits all good prescribed intervals between them.
Lemma~\ref{sl:lem:endpieces} bounds both discarded weights uniformly,
even though their endpoints depend on the path.}
\label{fig:gridcuts}
\end{figure}
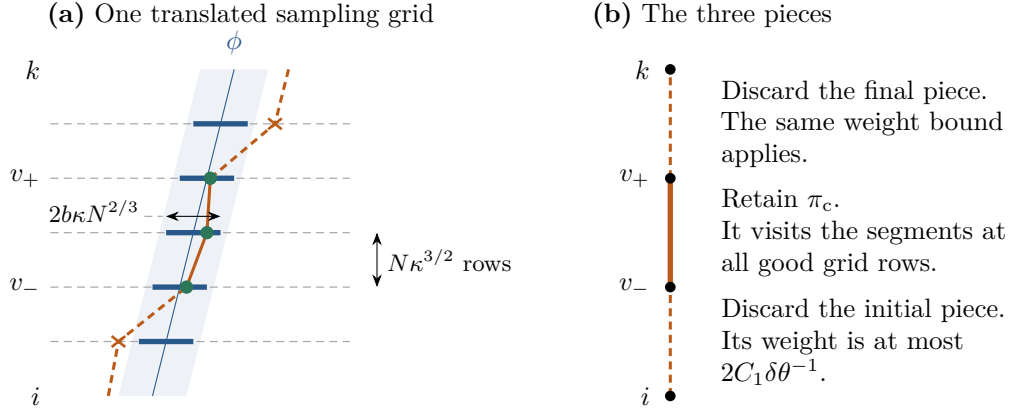

\subsection{Proof of the end-piece bound}\label{sec:endpieceproof}
\begingroup
{We now prove Lemma~\ref{sl:lem:endpieces} using the
compact-region weight estimate of Lemma~\ref{sl:lem:compactweightinput}.
That estimate applies when the row duration is comparable to $N$.
We extend each short end piece to put it within this range,
as in the proof of \cite[Lemma 4.15]{GH23}.}
\par\endgroup

\begin{proof}[Proof of Lemma~\ref{sl:lem:endpieces}]
\begingroup
We use the extension argument from \cite[Lemma 4.15]{GH23}
with a uniform weight event on the local window.
{Fix \(\delta\in(0,1]\) and set
\(\zeta=\delta\theta^{-1}\)}. Let \(\mathcal Q_{N,R}\) be the
deterministic set of ordered pairs \(p=(x,i),q'=(y,k)\) with
\begingroup
\[
 \begin{gathered}
 i,k\in[0,3N]\cap\mathbb Z,\qquad N\le k-i\le2N,\\
 \max\{|x-i|,|y-k|\}\le4RN^{2/3}.
 \end{gathered}
\]
\par\endgroup
Define
\begin{equation}\label{sl:eq:U}
 \cU=\cU_{N,R}(\zeta)
 =\left\{\sup_{(p,q')\in\mathcal Q_{N,R}}
                   |{W_N^0}(p,q')|\le C_1\zeta\right\}.
\end{equation}
{For \(\theta\le\delta^2\), the assumption
\(R\le\theta^{-1/4}\) gives
\(R^2\le\theta^{-1/2}\le\delta\theta^{-1}=\zeta\).}
{Apply Lemma~\ref{sl:lem:compactweightinput}
at scale \(2N\), so that its row-gap range is exactly \([N,2N]\).
Since \({W_N^0}(p,q')=2^{1/3}{W_{2N}^0}(p,q')\), a sufficiently
large choice of \(C_1\) gives}
\begingroup
\begin{equation}\label{sl:eq:Uprob}
 \PP(\cU^c)\le CR^2e^{-c\zeta^{3/2}}
             \le{C_\delta e^{-c_\delta\theta^{-3/2}}}.
\end{equation}
\par\endgroup
{The requirements \(R\le cN^{1/46}\) and
\(C\le\zeta\le cN^{1/30}\), as well as the ordering of endpoints
with the prescribed row gaps, follow from} \(R\le\theta^{-1/4}\),
{\(\zeta=\delta\theta^{-1}\) and
\(N\ge C\theta^{-300}\), for sufficiently small \(\theta\)
depending on \(\delta\)}.

{Let \(\sigma\) be a segment as in the statement, and note that its row duration
may be arbitrarily short. Write its endpoints
as \(p=(x,i)\), \(q'=(y,k)\), and extend the latter}
along slope one to \(q''=(y+N,k+N)\). Both pairs
\((p,q'')\) and \((q',q'')\) belong to \(\mathcal Q_{N,R}\),
so superadditivity gives, on \(\cU\),
\begin{equation}\label{sl:eq:endcap}
 {\W_N^0}(\sigma)
 \le {W_N^0}(p,q'')-{W_N^0}(q',q'')
 \le2C_1\zeta{=2C_1\delta\theta^{-1}}.
\end{equation}
This controls even zero-row-duration end pieces, without a union
over their microscopic endpoints.
\par\endgroup
\end{proof}
\par\endgroup

\begingroup
\subsection{\texorpdfstring{{Slender excursions in the dynamical environment}}{Slender excursions in the dynamical environment}}\label{sec:slenderapplication}

We now implement the weight comparison described at the start of
the section. Proposition~\ref{sl:prop:local} provides the slender-path
penalty. The two remaining tasks are to apply it with \(\Gamma^0\)
as the reference curve, using the exterior FKG comparison, and to
handle the unknown excursion location and scale. We state the
conditional comparison below. For the locations, at scale \(N_j\)
we cover the row interval by \(O(n/N_j)=O(2^j)\) windows.
Recall from \eqref{st:eq:scales} that
\[
 \theta_j=\bigl(a(N_j/n)^{1/3}\bigr)^\nu
          \asymp a^\nu2^{-j\nu/3}.
\]
The resulting error \(e^{-c\theta_j^{-{3/2}}}\) decreases fast
enough to sum over these windows and all scales. Regularity
only costs powers of \(1+j\), which are also dominated by
the increasing penalty \(\theta_j^{-1}\).

Here is the conditional form of the {FKG comparison} that
we use. For a deterministic staircase \(\gamma\), write
\(\operatorname{Ext}(\gamma)=(\mathbb R\times\mathbb Z)\setminus\gamma\).
A functional is increasing in the exterior increments if
increasing every Brownian increment along a horizontal
interval in \(\operatorname{Ext}(\gamma)\) cannot decrease it.

\begin{lemma}[{FKG comparison outside the geodesic}; {\cite[Lemma 9.6 and its proof]{GH24}}]\label{sl:lem:exteriorinput}
Let \(\widetilde B\) be a fresh Brownian environment independent
of the stationary OU process \((B^s)_{s\ge0}\). Fix \(t\ge0\).
For any measurable family of bounded nonnegative functionals
\(H_\gamma\), each depending only on Brownian
increments along horizontal intervals in \(\operatorname{Ext}(\gamma)\),
and increasing in those increments,
\[
 \EE\bigl[H_{\Gamma^0}(B^t)\mid\Gamma^0=\gamma\bigr]
 \le \EE\bigl[H_\gamma(\widetilde B)\bigr]
 \quad\text{for almost every }\gamma.{\footnotemark}
\]
\footnotetext{{Here ``almost every'' is with respect to the law of
the random staircase \(\Gamma^0\). The conditional expectation is
interpreted through a regular conditional distribution given
\(\Gamma^0=\gamma\); the exceptional set of reference staircases has
probability zero under this law.}}
In particular, this applies to the event that a
\(\gamma\)-dependent prescribed family of paths, with interiors
disjoint from \(\gamma\), contains a path exceeding a given
{centered}-weight threshold.
\end{lemma}
{
\begin{proof}[Derivation of the conditional formulation]
\begingroup
This is the conditional form of the exterior comparison in
\cite[Lemma 9.6 and its proof]{GH24}. Their construction fixes
the reference staircase and its on-curve noise before applying
FKG to the exterior increments. Averaging over the on-curve
noise therefore gives the comparison conditional only on
\(\Gamma^0=\gamma\), with \(H_\gamma\) allowed to depend on
that fixed curve. The same proof already extends the comparison
to every \(t\ge0\), using the monotonicity of the OU transition
and stationarity of fresh Brownian noise. Applying this
conditional domination to \(H_\gamma\) gives the displayed
inequality. No weight-regularity or stability event enters
the conditioning.
\par\endgroup
\end{proof}
}
\begingroup
We finally prove {Proposition~\ref{sl:prop:dynamic}}, whose definitions and
statement were given in Section~\ref{sec:assembly}.
\par\endgroup
\begin{proof}[Proof of {Proposition~\ref{sl:prop:dynamic}}]
\begingroup
{\textbf{Step 1. Reduction to wider tubes.}}
We begin with a reduction that lets us keep the tube exponent small
throughout the proof. Write \(\mathcal S_n^t(a,\nu)\) for the event
in \eqref{sl:eq:noslender}, displaying its width parameters.
For arbitrary \(\nu>0\), choose \(0<\nu_0<\nu\) with
\(100\nu_0(1-\beta)<\beta\), and set \(a_0=a^{\nu/\nu_0}\).
Since \(N_j/n\le1\),
\[
 \bigl[a_0(N_j/n)^{1/3}\bigr]^{\nu_0}
 =a^\nu(N_j/n)^{\nu_0/3}
 \ge\bigl[a(N_j/n)^{1/3}\bigr]^\nu.
\]
Thus the tubes for \((a_0,\nu_0)\) are wider. Every excursion
slender for \((a,\nu)\) is also slender for \((a_0,\nu_0)\),
with the same exceptional-row fraction \(\chi\), so
\[
 \mathcal S_n^t(a,\nu)\subseteq\mathcal S_n^t(a_0,\nu_0).
\]
An estimate proved for \((a_0,\nu_0)\) therefore gives the desired
bound, because \(a_0^{-3\nu_0/8}=a^{-3\nu/8}\). The two events
use the same scales and geodesics \(\Gamma^0,\Gamma^t\): only the
tube widths change, and the dynamical time \(t\) is unchanged.
For sufficiently small \(a\), the parameter \(a_0\) is in the
required range, and a lower bound on \(n\) depending on \(a_0\)
is allowed by the statement.

It therefore suffices to prove the result under the restriction
\begin{equation}\label{sl:eq:nuchoice}
 100\nu(1-\beta)<\beta,
\end{equation}
which we assume for the remainder of the proof. This will ensure
the condition \(N\ge C\theta^{-300}\) in
Proposition~\ref{sl:prop:local} at every retained scale.
Fix \(a\) and the deterministic dynamical time \(t\), and abbreviate
\(\mathcal S_n^t(a,\nu)\) to \(\mathcal S_n^t\). The tube widths
remain \(\theta_j=[a(N_j/n)^{1/3}]^\nu\), even when
\(t\ne a n^{-1/3}\).
\par\endgroup
We first define, in each row window, the event of an actual
slender excursion and the larger family of candidate paths used
in the weight comparison. We then use regularity to relate these
events and the exterior comparison to bound the candidate event.

\paragraph{\textbf{{Step 2}. Actual excursions and candidate paths in row windows.}}
For \(j\in\cI_n\), define the deterministic row windows
and their index set by
\begin{equation}\label{sl:eq:rowwindows}
 \begin{split}
 I_{j,k}&=[kN_j,(k+2)N_j]\cap[0,n],\\
 \mathcal K_j&=\{k\in\mathbb Z:0\le k\le\lfloor n/N_j\rfloor,
                                      \ |I_{j,k}|\ge N_j/2\}.
 \end{split}
\end{equation}
Here \(|I_{j,k}|\) is interval length.
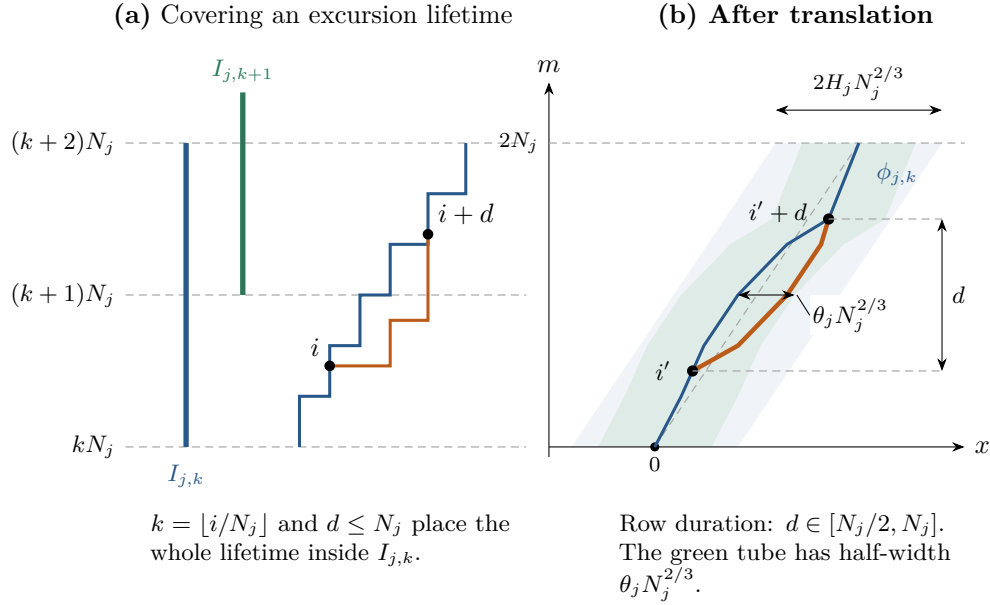
\begin{figure}[tbp]

\centering
\begin{tikzpicture}[x=1cm,y=.67cm]
 \node[anchor=west] at (-.3,8.5) {\textbf{(a)} Covering an excursion lifetime};
 \foreach \y/\lab in {0/kN_j,3/(k+1)N_j,6/(k+2)N_j,9/(k+3)N_j} {
  \ifnum\y<9
   \draw[guide] (0,\y)--(5.3,\y);
   \node[left,font=\footnotesize] at (0,\y) {$\lab$};
  \fi
 }
 \draw[pathblue,line width=2pt] (.8,0)--(.8,6);
 \node[below,pathblue,font=\footnotesize] at (.8,-.15) {$I_{j,k}$};
 \draw[pathgreen,line width=2pt] (1.55,3)--(1.55,7);
 \node[above,pathgreen,font=\footnotesize] at (1.55,7) {$I_{j,k+1}$};
 \draw[staticpath] (2.3,0)--(2.3,1)--(2.7,1)--(2.7,2)
 --(3.1,2)--(3.1,3)--(3.5,3)--(3.5,4)--(4,4)--(4,5)--(4.5,5)--(4.5,6);
 \draw[dynamicpath] (2.7,1.6)--(3.5,1.6)--(3.5,2.5)
 --(4,2.5)--(4,4.2);
 \fill (2.7,1.6) circle (2pt) node[above left] {$i$};
 \fill (4,4.2) circle (2pt) node[above right] {$i+d$};
 \node[anchor=north west,align=left,text width=4.8cm,font=\footnotesize] at (.2,-1.1)
 {$k=\lfloor i/N_j\rfloor$ and $d\le N_j$ place the whole lifetime inside $I_{j,k}$.};
   \begin{scope}[xshift=7cm]
    \node[anchor=west] at (-.1,8.5) {\textbf{(b) After translation}};
    \fill[pathblue!7] (-1.1,0)--(1.6,6)--(3.8,6)--(1.1,0)--cycle;
    \fill[pathgreen!15]
      (-.75,0)--(-.4,1)--(-.25,1.5)--(-.1,2)--(.35,3)
      --(1,4)--(1.55,4.5)--(1.95,6)--(3.45,6)
      --(3.05,4.5)--(2.5,4)--(1.85,3)--(1.4,2)
      --(1.25,1.5)--(1.1,1)--(.75,0)--cycle;
    \draw[->] (-1.4,0)--(4.1,0) node[right,black] {$x$};
    \draw[->] (-1.4,-.2)--(-1.4,7.2) node[above,black] {$m$};
    \fill (0,0) circle (1.7pt);
    \node[below,font=\scriptsize,black] at (0,0) {$0$};
    \draw[guide] (-1.4,6)--(4.1,6);
    \node[left,font=\scriptsize,black] at (-1.4,6) {$2N_j$};
    \draw[guide] (0,0)--(2.7,6);
    \draw[staticpath] (0,0)--(.35,1)--(.5,1.5)--(.65,2)
      --(1.1,3)--(1.75,4)--(2.3,4.5)--(2.7,6);
    \draw[dynamicpath,line width=1.7pt] (.5,1.5)--(1.1,2)
      --(1.75,3)--(2.2,4)--(2.3,4.5);
    \foreach \x/\y in {.5/1.5,2.3/4.5}
      \fill (\x,\y) circle (2.1pt);
    \node[black,anchor=east,font=\footnotesize] at (.35,1.5) {$i'$};
    \node[black,anchor=east,font=\footnotesize] at (2.15,4.65) {$i'+d$};
    \node[pathblue,right,font=\footnotesize] at (2.8,5.4) {$\phi_{j,k}$};
    \draw[<->,black] (1.1,3)--(1.85,3);
    \draw[black,line width=.4pt] (1.88,3)--(2.05,2.6);
    \node[black,anchor=west,font=\footnotesize,fill=white,inner sep=1pt]
      at (2.05,2.6) {$\theta_jN_j^{2/3}$};
    \draw[guide] (2.3,4.5)--(3.8,4.5);
    \draw[guide] (.5,1.5)--(3.8,1.5);
    \draw[<->,black] (3.8,1.5)--node[right,font=\footnotesize] {$d$}(3.8,4.5);
    \draw[<->,black] (1.6,6.65)--node[above,font=\scriptsize]
      {$2H_jN_j^{2/3}$}(3.8,6.65);
    \node[black,anchor=north west,align=left,text width=5cm,font=\footnotesize]
      at (-.6,-1.1) {Row duration: $d\in[N_j/2,N_j]$.\\
      The green tube has half-width $\theta_jN_j^{2/3}$.};
   \end{scope}
\end{tikzpicture}
\caption{The windows used in the dynamical application. Adjacent
windows overlap by $N_j$ rows, so a duration-$N_j$ excursion never
requires a new window tailored to its endpoints.
{In the right panel, translation sends
$(\Gamma^0(kN_j),kN_j)$ to the origin. The orange excursion runs
from row $i'=i-kN_j$ to row $i'+d$, where $d\in[N_j/2,N_j]$.
The green band shows the tube of horizontal half-width
$\theta_jN_j^{2/3}$ around $\phi_{j,k}$; slenderness permits at most
$\chi d$ rows outside this tube. Geometric regularity places the
reference curve and excursion endpoints in the wider strip
$|x-m|\le H_jN_j^{2/3}$. These are the geometric conditions needed
for Proposition~\ref{sl:prop:local}.}}
\label{fig:slenderwindows}
\end{figure}

We have
\(|\mathcal K_j|\le Cn/N_j\le C2^j\).
Every excursion of row duration in \([N_j/2,N_j]\) lies
in one of these windows: if its first row is \(i\), take
\(k=\lfloor i/N_j\rfloor\).

Let \(\mathcal S^t_{j,k}\) be the event that a scale-\(j\)
slender excursion of \(\Gamma^t\) from \(\Gamma^0\)
has its lifetime contained in \(I_{j,k}\). {Every such event
implies \(\mathcal S_n^t\), and the window cover gives the reverse
inclusion. Hence we have the equality}
\begin{equation}\label{sl:eq:windowcover}
 \mathcal S_n^t=\bigcup_{j\in\cI_n}
                      \ \bigcup_{k\in\mathcal K_j}\mathcal S^t_{j,k}.
\end{equation}
This cover replaces a count over all possible pairs of
microscopic endpoint rows by a count of order \(n/N_j\)
windows at scale \(j\).

For a {staircase} \(\pi\), write \(d(\pi)\) for its row duration.
Alongside the actual-excursion event, define the candidate-path event
\begin{equation}\label{sl:eq:Adefinition}
 \mathcal A^t_{j,k}=\left\{
 \begin{array}{l}
 \text{there exists a }{\text{staircase}}\ \pi\text{ with lifetime in }I_{j,k},\\
 d(\pi)\in[N_j/2,N_j],\quad
 \text{endpoints on }\Gamma^0\text{ and interior disjoint from }\Gamma^0,\\
 \#\{r\text{ a row of }\pi:
    |\pi(r)-\Gamma^0(r)|>\theta_jN_j^{2/3}\}\le{\chi d(\pi)},\\
 \W_{N_j}^t(\pi)\ge-c_*\theta_j^{-1}
 \end{array}\right\}.
\end{equation}
The event \(\mathcal S^t_{j,k}\) requires an excursion of the
actual geodesic \(\Gamma^t\), whereas \(\mathcal A^t_{j,k}\) allows
any exterior path with the stated geometry and sufficiently high
weight. Conditional on \(\Gamma^0\), its permitted path family is
fixed; the time-\(t\) noise is used only to test the weights. This
is why the candidate event is suitable for the FKG comparison.

We will show that regularity places each actual slender excursion
in this candidate event, and then bound the probability of the
candidate event after conditioning on \(\Gamma^0\). Geometric
regularity will place the reference curve and endpoints in the
region permitted by Proposition~\ref{sl:prop:local}; weight
regularity will give an actual excursion the required weight.

\paragraph{\textbf{{Step 3}. Geometric and weight regularity.}}
\begingroup
Set $R_*=a^{-\nu/8}$. The geometric and weight regularity bounds
{\eqref{st:eq:georeg} and \eqref{st:eq:weightreg} in
Lemma~\ref{st:lem:regularityinput}} have respective sizes $R_*(1+j)^{1/3}$ and
$R_*^2(1+j)^{2/3}$, so we choose a sufficiently large fixed
constant $C$ and write the allowances as
\begin{equation}\label{sl:eq:allowances}
 H_j=CR_*(1+j)^{1/3},\qquad B_j=CH_j^2.
\end{equation}
{Here $H_j$ denotes the allowance obtained from the present
choice $R_*=a^{-\nu/8}$; it replaces the earlier choice of $H_j$
in the stability argument of Section~\ref{sec:stability}.}
\par\endgroup
The geometric event at scale \(j\) is
\[
 \mathcal G^0_{{\rm geom},j}
 =\left\{
 \begin{array}{l}
 |x-r-(\Gamma^0(kN_j)-kN_j)|\le H_jN_j^{2/3}\\
 \text{for every }k\in\mathcal K_j\text{ and }(x,r)\in\Gamma^0
                              \text{ with }r\in I_{j,k}{\cap\ZZ}
 \end{array}\right\}.
\]
The weight event at scale \(j\) is
\begin{equation}\label{sl:eq:weightlower}
 \mathcal G^t_{{\rm wt},j}
 =\left\{\begin{array}{l}
   \W_{N_j}^t(\pi)\ge-B_j
   \text{ for every subpath }\pi\text{ of }\Gamma^t\\
   {\text{with endpoints in }\RR\times\ZZ}
   \text{ and }d(\pi)\in[N_j/2,N_j]
   \end{array}\right\}.
\end{equation}
Finally, set
\[
 \mathcal G^0_{\rm geom}=\bigcap_{j\in\cI_n}\mathcal G^0_{{\rm geom},j},
 \qquad
 \mathcal G^t_{\rm wt}=\bigcap_{j\in\cI_n}\mathcal G^t_{{\rm wt},j}.
\]
The first event depends only on the reference path \(\Gamma^0\)
and controls all its horizontal row segments, including possible
excursion endpoints. The second concerns weights in the
time-\(t\) environment. We keep them separate because the
geometric event is known after conditioning on \(\Gamma^0\),
whereas the weight event will only be used to show that an
actual excursion belongs to \(\mathcal A^t_{j,k}\). We do not
condition on the weight event when applying FKG in {Step 5}.

{Apply Lemma~\ref{st:lem:regularityinput} with $R=R_*$
and $c_0=1/16$, so its estimates cover row separations in
$[N_j/16,2N_j]$. Its upper bound on $R_*$ holds for sufficiently
large $n$ depending on $a$. Taking the allowances above sufficiently
large, the lemma gives}
\begin{equation}\label{sl:eq:Gprob}
 \PP\bigl((\mathcal G^0_{\rm geom})^c\bigr)
 +\PP\bigl((\mathcal G^t_{\rm wt})^c\bigr)
 \le Ce^{-cR_*^3}.
\end{equation}
\begingroup
For the weight event, divide \eqref{st:eq:weightreg} by
$N_j^{1/3}$ and use that every subpath of $\Gamma^t$
{with endpoints on integer rows} is {maximising} between its endpoints.
The choice $B_j=CH_j^2$ therefore gives the lower bound in
\eqref{sl:eq:weightlower}.
\par\endgroup

For the geometric event, compare {a geodesic point on an integer row in $I_{j,k}$} with a
departure at the farther endpoint row of this interval, and also
compare the two endpoint departures when necessary. The interval
has length between $N_j/2$ and $2N_j$, so these comparisons use
separations in $[N_j/4,2N_j]$. Equation~\eqref{st:eq:georeg}
then gives the claimed allowance $H_j$ for {every geodesic point on an integer row of the window}. The lemma already includes all scales and all point pairs;
no union over microscopic rows is needed here.

\paragraph{\textbf{{Step 4}. Check admissibility in one window.}}
\begingroup
We now check that every path permitted in the definition of
\(\mathcal A^t_{j,k}\) is admissible for
Proposition~\ref{sl:prop:local} on \(\mathcal G^0_{\rm geom}\).
This check concerns only the path geometry.
\begingroup
Fix $j\in\cI_n$ and $k\in\mathcal K_j$. We translate the
window by sending its initial reference point to the origin:
\[
 (x,m)\longmapsto
 \bigl(x-\Gamma^0(kN_j),\ m-kN_j\bigr).
\]
No rescaling is involved. In the translated coordinates, define
the reference function on $[0,2N_j]\cap\mathbb Z$ by
\[
 \phi_{j,k}(m)=
 \begin{cases}
  \Gamma^0(kN_j+m)-\Gamma^0(kN_j),&kN_j+m\le n,\\
  \Gamma^0(n)-\Gamma^0(kN_j)+(kN_j+m-n),&kN_j+m>n.
 \end{cases}
\]
The second line extends the reference function with slope one
if the window reaches beyond row $n$. On $\mathcal G^0_{\rm geom}$,
\[
 |\phi_{j,k}(m)-m|\le H_jN_j^{2/3}
 \quad\text{for every }m\in[0,2N_j]\cap\mathbb Z.
\]
The same strip bound holds for the translated candidate endpoints,
since the original endpoints lie on $\Gamma^0$ in the window.
{For a candidate path $\pi_{\rm old}$ in the original window,
write $\pi$ for its image under the displayed coordinate shift; its
departure coordinates are
$\pi(m)=\pi_{\rm old}(kN_j+m)-\Gamma^0(kN_j)$ on its translated
lifetime.} Its lifetime lies in $[0,2N_j]$, its row duration
is still in $[N_j/2,N_j]$, and at most $\chi d(\pi)$ rows satisfy
\[
 |\pi(m)-\phi_{j,k}(m)|>\theta_jN_j^{2/3}.
\]
The same count holds with the larger threshold $2\theta_jN_j^{2/3}$.
Thus the translated path is admissible with $N=N_j$, $R=H_j$
and $\theta=\theta_j$, provided the parameter restrictions
of Proposition~\ref{sl:prop:local} hold.
\par\endgroup
\par\endgroup

There are two restrictions to check. First,
\begin{equation}\label{sl:eq:Hfit}
 H_j\theta_j^{1/4}
 \le Ca^{\nu/8}(1+j)^{1/3}2^{-j\nu/12}\le1
\end{equation}
for small enough \(a_*\), uniformly in \(j\), because
the factor depending on \(j\) is bounded. This gives
\(H_j\le\theta_j^{-1/4}\), the permitted strip width.
Second,
\begin{equation}\label{sl:eq:finitesize}
 {N_j\theta_j^{300}
 =a^{300\nu}N_j(N_j/n)^{100\nu}
 \ge a^{300\nu}n^{\beta-100\nu(1-\beta)}
 \longrightarrow\infty.}
\end{equation}
The exponent is positive by the temporary
restriction \eqref{sl:eq:nuchoice}.
Hence \({N_j\ge C\theta_j^{-300}}\) simultaneously at all
retained scales when \(n\) is sufficiently large depending
on \(a\). {This lower bound on \(n\) does not
introduce a factor in the probability estimate.} Also
\(\theta_j\le a^\nu\le\theta_*\) for small enough \(a_*\).

We have therefore checked the local proposition with
\(N=N_j\), \(R=H_j\), and \(\theta=\theta_j\).
Notice that the geometric allowance grows only as a power
of \(1+j\), while the allowable strip
\(\theta_j^{-1/4}\) grows exponentially in \(j\).

\paragraph{\textbf{{Step 5}. Bound the candidate event by conditioning on the reference path.}}
We now estimate \(\mathcal A^t_{j,k}\), defined in
\eqref{sl:eq:Adefinition}. {Condition on
\(\Gamma^0\). The permitted path family, the translation in
{Step 4}, and the translated reference function \(\phi_{j,k}\)
are now fixed.} Since the paths have interiors
disjoint from \(\Gamma^0\), increasing exterior Brownian
increments can only increase their weights. Hence
\(\mathcal A^t_{j,k}\) is increasing in those increments.
Lemma~\ref{sl:lem:exteriorinput} bounds its conditional
probability by the probability of the same event in fresh
Brownian noise. {Apply the fixed translation from Step 4
to the candidate paths and this fresh environment. The translated
noise is again a Brownian environment, and the centered weights
are unchanged. On \(\mathcal G^0_{\rm geom}\), Step 4 verifies that
the reference function and candidate endpoints lie in the required
strips, each candidate has row duration in \([N_j/2,N_j]\), and
its departures satisfy the tube condition. Thus every candidate
is admissible for Proposition~\ref{sl:prop:local}, with
\(N=N_j\), \(R=H_j\), and \(\theta=\theta_j\). That proposition
bounds the probability that any admissible path has centered
weight at least \(-c_*\theta_j^{-1}\), and hence bounds the
fresh-environment probability of our candidate event. Together
with the FKG comparison, this gives}
\begin{equation}\label{sl:eq:conditional}
 \ind_{\mathcal G^0_{\rm geom}}
 \PP(\mathcal A^t_{j,k}\mid\Gamma^0)
 \le \ind_{\mathcal G^0_{\rm geom}}Ce^{-c\theta_j^{-{3/2}}}.
\end{equation}
\begingroup
The bound is uniform in the conditioned reference curve.
The translation was fixed before applying the fresh-noise
estimate, so there is no union over its possible values.
\par\endgroup
Since \(\mathcal G^0_{\rm geom}\) is measurable with
respect to \(\Gamma^0\), taking expectations yields
\begin{equation}\label{sl:eq:windowprob}
 \PP(\mathcal G^0_{\rm geom}\cap\mathcal A^t_{j,k})
 \le Ce^{-c\theta_j^{-{3/2}}}.
\end{equation}
We have conditioned only on the reference path. In particular,
\(\mathcal G^t_{\rm wt}\) and the occurrence of an actual
excursion have not entered the conditioning.

\paragraph{\textbf{{Step 6}. An actual slender excursion forces the high-weight event.}}
\begingroup
The strip-width condition checked in \eqref{sl:eq:Hfit} also
makes the weight allowance negligible compared with the slender-path
penalty. Indeed, $B_j=CH_j^2$ and $H_j\le\theta_j^{-1/4}$ give
\[
 B_j\theta_j\le C\theta_j^{1/2}\le Ca^{\nu/2}.
\]
Thus, after decreasing $a_*$ if necessary, we have uniformly in $j$
\begin{equation}\label{sl:eq:contradiction}
 B_j<\frac{c_*}{2}\theta_j^{-1}.
\end{equation}
\par\endgroup

Suppose \(\mathcal S^t_{j,k}\) occurs, witnessed by an
actual excursion \(\pi\). On \(\mathcal G^t_{\rm wt}\),
its weight is at least \(-B_j\), because it is a subpath
of \(\Gamma^t\). By \eqref{sl:eq:contradiction},
\[
 \W_{N_j}^t(\pi)\ge-B_j>-c_*\theta_j^{-1}.
\]
Its remaining properties are precisely those in the
definition of \(\mathcal A^t_{j,k}\). Thus, for every window,
\begin{equation}\label{sl:eq:windowimplication}
 \mathcal S^t_{j,k}\cap\mathcal G^0_{\rm geom}\cap\mathcal G^t_{\rm wt}
 \subseteq\mathcal A^t_{j,k}\cap\mathcal G^0_{\rm geom}.
\end{equation}
This is the weight contradiction: the weight of a geodesic
segment cannot be as negative as the typical penalty for
a slender path. Both weights are evaluated at time \(t\);
no passage-time stability estimate is used in this step.

\paragraph{\textbf{{Step 7}. Summing over windows and scales.}}
Combining \eqref{sl:eq:windowcover} and
\eqref{sl:eq:windowimplication} gives the deterministic inclusion
\begin{equation}\label{sl:eq:finalinclusion}
 \mathcal S_n^t
 \subseteq (\mathcal G^0_{\rm geom})^c
       \cup(\mathcal G^t_{\rm wt})^c
       \cup\bigcup_{j\in\cI_n}\ \bigcup_{k\in\mathcal K_j}
                    (\mathcal G^0_{\rm geom}\cap\mathcal A^t_{j,k}).
\end{equation}
We now take probabilities. By \eqref{sl:eq:windowprob}
and \(|\mathcal K_j|\le C2^j\),
\begin{equation}\label{sl:eq:scalesum}
 \sum_{j\in\cI_n}\sum_{k\in\mathcal K_j}
  \PP(\mathcal G^0_{\rm geom}\cap\mathcal A^t_{j,k})
 \le\sum_{j\in\cI_n}C2^j e^{-c\theta_j^{-{3/2}}}
 \le Ce^{-c'a^{-{3\nu/2}}}.
\end{equation}
Indeed, \(\theta_j^{-{3/2}}\ge a^{-{3\nu/2}}2^{{j\nu/2}}\).
The positive term \(j\log2\) from the window count is
dominated by this exponential growth in \(j\). More explicitly,
for \(a\) small enough it can be absorbed into half of the
negative exponent, and
\[
 \sum_{j\ge0}\exp\{-c a^{-{3\nu/2}}2^{{j\nu/2}}/2\}
 \le Ce^{-c'a^{-{3\nu/2}}}.
\]
Here we have extended the finite sum to all \(j\ge0\),
so the bound does not acquire a factor from the number
of retained scales.

Finally, \eqref{sl:eq:Gprob}, \eqref{sl:eq:finalinclusion}
and \eqref{sl:eq:scalesum} imply
\[
 \PP(\mathcal S_n^t)
 \le Ce^{-cR_*^3}+Ce^{-c'a^{-{3\nu/2}}}
 \le Ce^{-c''a^{-3\nu/8}},
\]
as \(R_*=a^{-\nu/8}\). {The checks of the
scale and parameter ranges in {Steps 3 and 4} explain why
\(n\) must be sufficiently large depending on \(a\).} The only location count left in the probability
estimate was \(n/N_j\), and it was summed in
\eqref{sl:eq:scalesum}. No independence between windows,
scales, or the two regularity events is used.

\end{proof}

\par\endgroup
\appendix
\section{Parameter conditions and auxiliary results}\label{app:parameters}

We collect the parameter conditions and the auxiliary estimates
used in the proof, including the conventions needed to apply them.

\subsection{A guide to the notation}
The scale parameters are chosen for the full argument and
then used consistently in its two geometric parts. The
following table records the principal objects; the cutoff
\(n^\beta\) and the tube exponent \(\nu\)
have different roles and may be chosen independently within
the ranges in \eqref{eq:parameters}.

\begin{center}
\small
\renewcommand{\arraystretch}{1.25}
\begin{tabular}{@{}p{.23\textwidth}p{.70\textwidth}@{}}
\toprule
Notation & Meaning and first use\\
\midrule
\(n,t,a\) & Original path scale, physical dynamical time,
and its critical-scale coefficient: \(t=a n^{-1/3}\).\\
\(O_n(t)\) & Horizontal length shared by \(\Gamma^0\)
and \(\Gamma^t\); \eqref{st:eq:overlap}.\\
\(D_{\rm long}\) & Sum of row durations of excursions
lasting at least \(n^\beta\) rows; \eqref{st:eq:longtarget}.\\
\(N_j,\cI_n\) & Dyadic row-duration scales and the set
of scales above the short-excursion cutoff; \eqref{st:eq:scales}.\\
\(\tau_j\) & Effective dynamical time at scale \(N_j\):
\(a(N_j/n)^{1/3}\); \eqref{st:eq:tau}.\\
\(A_j\) & Uniform local passage-time tolerance
\(C\tau_j^{1/4}N_j^{1/3}\); \eqref{st:eq:localstability}.\\
\(\theta_j,w_j\) & Tube width in local KPZ units and
original horizontal units:
\(\theta_j=\tau_j^\nu\), \(w_j=\theta_jN_j^{2/3}\).\\
\(Z_m,\cP_m(A,w)\) & Static routed profile and its
separated twin-peak event; \eqref{wd:eq:peak}.\\
\(\cE_{n,a}\) & One stability event for every possible
subsegment at every retained scale;
Proposition~\ref{st:prop:local}.\\
\(N\kappa^{3/2}\) &
Row spacing of the local sampling grid;
\eqref{sl:eq:gridspacing}.\\
\(\cH_r,\cU\) &
Fixed-offset high-weight event and event
controlling discarded end pieces; \eqref{sl:eq:Hrdefinition}
and \eqref{sl:eq:U}.\\
\(\mathcal S^t_{j,k},\mathcal A^t_{j,k}\) &
Actual slender-excursion event and the event
that an exterior candidate path has sufficiently high weight
in row window \(I_{j,k}\); Section~\ref{sec:slenderapplication}.
Conditional on \(\Gamma^0\), the candidate path family is fixed.\\
\(\xi,\cB_{n,\xi}\) & Bulk fraction and retained bulk
rows in the final count; Section~\ref{sec:assembly}.\\
\bottomrule
\end{tabular}
\end{center}

\subsection{Parameters and limits}
\begingroup
The wide-excursion calculation requires \(0<\nu<1/2\),
because the exponent in \eqref{wd:eq:substitution} must be
positive. Proposition~\ref{sl:prop:dynamic} holds for every
fixed \(\nu>0\). {The initial wider-tube reduction allows
its proof to assume \eqref{sl:eq:nuchoice}, which is used to check
\eqref{sl:eq:finitesize}. The reduction uses a smaller exponent
and a correspondingly smaller width parameter, preserving the
probability bound.} The short-excursion reduction {in Lemma~\ref{lem:fullreduction}} requires
\(\beta<1/12\). With this range its auxiliary exponent can be
chosen as \(\lambda=1/12\), so that \(\beta<\lambda\) and
\(\beta+\lambda<1/6\).
\par\endgroup
Each lower bound on \(n\) may depend on the fixed \(a\).
The bulk parameter \(\xi\) is held fixed through the limits
\(n\to\infty\) and \(a\downarrow0\), and is sent to zero
only afterward.

The regularity parameters in the two arguments serve different
purposes. The stability proof uses
\(R\asymp a^{-1/16}\), so the regularity failure
\(Ce^{-cR^3}\) has the same stretched-exponential order
as the endpoint-comparison error in \eqref{st:eq:scaleunion}. The slender proof uses
\(R_*=a^{-\nu/8}\), which fits its local strip
\(\theta_j^{-1/4}\) and still makes the weight lower bound
larger than \(-c_*\theta_j^{-1}\).
In the wider-tube reduction,
{\(a_0^{-\nu_0/8}=a^{-\nu/8}\)}, so the same regularity allowance
is retained. Neither parameter grows with \(n\) when \(a\) is fixed.

\subsection{Random endpoints and conditioning}
The stability endpoint families are deterministic, and their estimates are
proved before any endpoints are selected from the geodesics.
The box count and the compact-set constants are explicit in
\eqref{st:eq:boxcount} and the ensuing coordinate calculation.
The three-value identity {\eqref{st:eq:threevalues}} uses the row-disjoint routed profile
with its upper leg starting at \((x,m+1)\).
{This is also the routed-profile convention in
the imported parametrised estimate; the identification is made
in the proof of Proposition~\ref{wd:prop:static}.}

For the slender estimate, the reference path is fixed by the
conditioning in Lemma~\ref{sl:lem:exteriorinput}. Only its geometric regularity is used in choosing
the local coordinate system. {The candidate event
\(\mathcal A^t_{j,k}\) in \eqref{sl:eq:Adefinition} is increasing in
exterior increments, so the FKG comparison applies. The weight event
\(\mathcal G^t_{\rm wt}\), defined using \eqref{sl:eq:weightlower},
provides lower bounds for actual geodesic segments. The event
\(\cU\) in \eqref{sl:eq:U} controls the weights of discarded end
pieces in the deterministic-path estimate. Neither event, nor the
occurrence of an excursion, is included in the conditioning.}

\subsection{Auxiliary results}
\begin{center}
\small
\renewcommand{\arraystretch}{1.2}
\begin{tabular}{@{}p{.34\textwidth}p{.59\textwidth}@{}}
\toprule
Input & Source and use\\
\midrule
OU passage-time increments &
{\cite[Theorem 5]{Bha26}; imported as
Lemma~\ref{lem:rotation}, with the scaling argument in
Section~\ref{sec:outightness}.}\\
Tightness of the OU dynamics &
{Proved in Section~\ref{sec:outightness}, using
\cite[Theorem 5]{Bha26} and static endpoint regularity.}\\
Static endpoint increments &
Lemma~\ref{st:lem:endpointinput}, from
{\cite[Proposition 2.6, Lemma 11.2 and the proof of Theorem 11.1]{DOV22}}.\\
Static geodesic and weight regularity &
Lemmas~\ref{st:lem:regularityinput} and
\ref{sl:lem:compactweightinput}, from
{\cite[Theorems 1.4 and 1.6, Corollary 1.5 and Proposition 1.8]{GH23}}.\\
{Parametrised static near peaks} &
{\cite[Proposition 82]{Bha25}, with exponent
parameter $1/2$; imported in Proposition~\ref{wd:prop:static}
and applied in Corollary~\ref{wd:cor:sum}.}\\
Prescribed sparse constraints &
{Lemma~\ref{sl:lem:sparseinput}, adapted from
{\cite[Proposition 4.5]{GH23}}.}\\
{FKG comparison} &
Lemma~\ref{sl:lem:exteriorinput}, from
{\cite[Lemma 9.6 and its proof]{GH24}};
applied conditionally on the time-zero geodesic.\\
Short excursions and horizontal overlap &
{Lemma~\ref{lem:fullreduction}, using
{\cite[Proposition 8.4 and its proof]{GH24}},
with arbitrary fixed thresholds and a conversion
to horizontal length.}\\
\bottomrule
\end{tabular}
\end{center}

{The required ranges of scales, parameters and
endpoint coordinates are checked at each application.}

\begingroup
\printbibliography
\par\endgroup
\end{document}